\documentclass[11pt,a4paper, twoside]{article}
\usepackage[english]{babel}
\usepackage{geometry}
\usepackage{amsmath}
\usepackage{hyperref}
\usepackage{amsthm}
\usepackage{amsfonts}
\usepackage{amssymb}
\usepackage[english]{varioref}
\usepackage{graphicx}
\usepackage{amsthm} 
\usepackage{amsmath}
\usepackage{eucal}
\usepackage{setspace}
\usepackage{version}
\usepackage{color}
\numberwithin{equation}{section}

\newcommand{\T}{\mathbb{T}}

\newtheorem{theorem}{Theorem}[section]   
\newtheorem{corollary}[theorem]{Corollary}
\newtheorem{lemma}[theorem]{Lemma}
\newtheorem{proposition}[theorem]{Proposition}

\theoremstyle{definition}
\newtheorem{definition}[theorem]{Definition}
\newtheorem{remark}[theorem]{Remark}

\usepackage{mathtools}

\usepackage{tocloft}

    \DeclareUnicodeCharacter{2212}{-}
\begin{document}
\title{ \textbf{Global regularity for the one-dimensional stochastic Quantum-Navier-Stokes equations}}
\author{D. Donatelli\footnote{Department of Information Engineering, Computer Science and Mathematics, University of L'Aquila, 67100 L'Aquila, Italy, email: donatella.donatelli@univaq.it}, L. Pescatore\footnote{Department of Information Engineering, Computer Science and Mathematics, University of L'Aquila, 67100 L'Aquila, Italy, email: lorenzo.pescatore@graduate.univaq.it}, S. Spirito\footnote{Department of Information Engineering, Computer Science and Mathematics, University of L'Aquila, 67100 L'Aquila, Italy, email: stefano.spirito@univaq.it}}
\date{}


\maketitle
\begin{abstract}
In this paper we prove the global in time well-posedness of strong solutions to the Quantum-Navier-Stokes equation driven by random initial data and stochastic external force. In particular, we first give a general local well-posedness result. Then, by means of the Bresch-Desjardins entropy, higher order energy estimates, and a continuation argument we prove that the density never vanishes, and thus that local strong solutions are indeed global. 
\end{abstract}
\medbreak
{\small\textbf{Key words and phrases:}
Stochastic compressible fluids; Navier-Stokes-Korteweg equations; Global well-posedness. 
\medbreak
}
\section{Introduction}
Let $\mathbb{T}$ be the one-dimensional flat torus and $T>0$. In this paper we consider the one-dimensional Quantum-Navier-Stokes equations for a compressible fluid with capillarity effects subject to a stochastic force. Precisely, in $(0,T)\times\T$ the equations we study are the following
 \begin{equation} \label{stoc quantum}
\begin{cases}
\text{d}\rho+\partial_x(\rho u)\text{d}t=0,\\
\text{d}(\rho u)+[ \partial_x(\rho u^2+p(\rho))]\text{d}t=[\partial_x(\mu(\rho)\partial_x u)+ \rho \partial_x \bigg( \dfrac{\partial_{xx} \sqrt{\rho}}{\sqrt{\rho}}\bigg) ]\text{d}t +\mathbb{G}(\rho,\rho u)\text{d}W,
\end{cases}
\end{equation}
where the unknowns $\rho>0$ and $u\in\mathbb{R}$ denotes the density and the velocity of the fluid, respectively. Moreover,  
\begin{equation*}
p(\rho)=\rho^{\gamma} , \quad \gamma > 1 , \quad \mu(\rho)= \rho^\alpha, \quad \alpha \geq 0, 
\end{equation*}
denote the pressure of the fluid and the viscosity coefficient. Finally, $\mathbb{G}(\rho, u)\text{d}W$ is a stochastic external force whose properties will be described later. 
For the system \eqref{stoc quantum} we impose the following initial conditions 
\begin{equation} \label{C.I}
(\rho(x,t), u(x,t))_{|_{t=0}}=(\rho_0(x), u_0(x)),
\end{equation} 
and periodic boundary conditions 
\begin{equation}
\rho(0,t)=\rho(1,t), \quad u(0,t)=u(1,t).
\end{equation}
The system \eqref{stoc quantum} belongs to the wider class of fluid dynamical evolution equations called the Navier-Stokes-Korteweg equations, which in the multi-dimensional case read
\begin{equation} \label{3D}
\begin{cases}
\partial_t \rho + \operatorname{div}(\rho u)=0, \\
\partial_t (\rho u) + \operatorname{div} (\rho u \otimes u)+ \nabla p = \operatorname{div} \mathbb{S} + \operatorname{div}\mathbb{K}.
\end{cases}
\end{equation}
In \eqref{3D}, $\mathbb{S}=\mathbb{S}(\nabla u)$ denotes the viscosity stress tensor $$ \mathbb{S}= \mu(\rho) D u+ g(\rho) \operatorname{div}u \mathbb{I},$$
and $\mu(\rho), g(\rho)$ are the first and the second viscosity coefficients, which must satisfy the Lam\'e relation $\mu(\rho) \ge 0, \; \mu(\rho)+d g(\rho) \ge 0$. Moreover, the capillarity tensor $\mathbb{K}=\mathbb{K}(\rho, \nabla \rho)$ is given by 
$$ \mathbb{K}= \bigg ( \rho \operatorname{div}(k(\rho) \nabla \rho)-\dfrac{1}{2}(\rho k'(\rho)-k(\rho)) | \nabla \rho |^2 \bigg ) \mathbb{I}-k(\rho) \nabla \rho \otimes \nabla \rho$$
and $k(\rho)$ is called capillary coefficient. The Navier-Stokes-Korteweg equations are used to model fluids when the effects of capillarity are not negligible and the derivation in the context of continuous mechanics of the Korteweg tensor $\mathbb{K}$ has been performed in \cite{Dunn}. We note that system \eqref{stoc quantum} is the one dimensional stochastic version of the Navier-Stokes-Korteweg equations \eqref{3D} with the choice of the capillary coefficient $k(\rho)=\frac{1}{\rho}$. In particular, the name Quantum-Navier-Stokes is due to the fact that for this choice of capillarity coefficient, the Korteweg tensor reduces to the so-called Bohm potential and in the inviscid case the system \eqref{stoc quantum} is the well known Quantum Hydrodynamics (QHD) model for superfluids, which arises in quantum semiconductors theory,  see \cite{Landau}.\\
\\
The local well-posedness in the inviscid case of system \eqref{3D} has been studied in \cite{Benz}, while global well posedenss was proved only in the case of small irrotational data in \cite{Aud}.  Moreover,  the global existence of finite energy weak solutions for QHD has been proved for the three-dimensional case in \cite{Ant1} and for the two-dimensional case in \cite{Ant2}, see also the recent studies \cite{Ant3} and \cite{Ant4}.  Finally, in \cite{Feir. Don.}  a well/ill posedness analysis has been performed for the Euler-Korteweg-Poisson system in the class of weak solutions. Concerning the viscous problem, global existence of finite energy weak solutions in several space dimensions for the Quantum-Navier-Stokes equations with viscosity coefficients linear in $\rho$ has been established in \cite{Spirito},\cite{Lacroix}, see also \cite{Bresch} for a generalized version of the Quantum-Navier-Stokes equations and \cite{Spirito2} for the case of constant capillarity coefficient.  Also, some singular limits have been studied for the Quantum-Navier-Stokes equations. In particular, the low Mach number limit \cite{Ant5}, the high Mach number analysis for Korteweg fluids with density dependent viscosity has been performed in \cite{Donatelli}. Moreover, a relaxation limit of the Quantum-Navier-Stokes equations toward the quantum drift diffusion has been performed in \cite{Ant0}.\\
\\
\\
In this paper we prove existence and uniqueness of global solutions which are strong both in PDE and in probability sense.  Our result is of interest from both a stochastic and a deterministic point of view.  Indeed, the analysis of stochastic compressible fluids is quite recent. We refer to the recent monograph \cite{Feir} for an overview on stochastically forced compressible Navier-Stokes equations with constant viscosity. In the present paper, we consider  the presence of the Korteweg tensor. Due to the highly non-linear nature of the Korteweg tensor, the extension of the theory of \cite{Feir} requires non-trivial arguments. Moreover, the global well posedness of strong solutions of the stochastic system \eqref{stoc quantum} is very useful in studying several different problems, {\em e.g.}  the existence of invariant measures, as in \cite{Coti}, which seems completely new in the context of Korteweg fluids.  Finally, it is important to mention that we consider stochastic effects also in the initial data $(\rho_0,u_0)$, which are random variables ranging in a suitable regularity space. This is particularly important from a point of view of applications.\\

Concerning the deterministic setting, that is $\mathbb{G}=0$, the study of density dependent viscous fluids in the one-dimensional case has also been subject of interest in the recent years. In particular, in the case of $\mu(\rho)=\rho^{\alpha}$ with $\alpha\in [0,1/2]$ the global regularity and uniqueness of strong solutions for the one-dimensional compressible Navier-Stokes equations has been proved in \cite{Mellet}, see also \cite{Const} and \cite{Const2}. The present paper is therefore an extension of these results to the case of compressible fluids with capillarity and stochastic force. Both in our result and in \cite{Mellet} the main tool is the so-called Bresch-Desjardin inequality, \cite{BD}, which provides additional regularity on the density. In particular, in the range $\alpha\in [0,1/2]$ for the one-dimensional case the BD inequality provided a control on the vacuum region. This is consistent with the analysis performed in \cite{Le Floch}, where a general form of this estimate has been proved, and the restriction $\alpha\in [0,1/2]$ has been also found. Moreover, our setting includes more general assumptions on viscosity and capillarity coefficients with respect to the results in \cite{Charve, Burtea1, Burtea2}. Indeed, in those papers the capillarity coefficients and the viscosity coefficients are related trough a relation which generalize the case of linear viscosity and $\kappa(\rho)=1/\rho$. 
\\
\\
The outline of the paper is the following: in Section 2 we discuss the assumption on the stochastic force,  we introduce the definitions and we state Theorem \ref{Main Theorem local}, Theorem \ref{Main Theorem global} which we aim to prove.  Section 3 is dedicated to the proof of Theorem \ref{Main Theorem global} by assuming the local existence result in Theorem \ref{Main Theorem local}. In Section 4 we prove Theorem \ref{Main Theorem local}, establishing the local well posedness for strong solutions in both PDEs and probability sense. 

\section{Main results}
In this section we fix the notation used in the sequel,  we write the assumption on the stochastic setting and we state some definitions and the main results.
\subsection{Preliminaries}
Beyond the standard notations and conventions which are commonly used in the literature,  we use the following:
\begin{enumerate}
\item
Given two functions $F, G$ and a variable $p,$ we denote by $F \lesssim G$ and $F \lesssim_{p} G$ the existence of a generic constant $c>0$ and $c(p)>0$  such that $F \le cG$ and $F \le c(p)G$ respectively.
\item
We denote by $L^p(\mathbb{T})$ the Lebesgue space on the one dimensional torus $\mathbb{T}$ and with $\| \cdot \|_{L^p}$ its norm. $W^{k,p}(\mathbb{T})$ denotes the Sobolev space of $L^p(\mathbb{T})$ functions with $k$ distributional derivatives in $L^p(\mathbb{T})$ and $H^k(\mathbb{T})$ corresponds to the case $p=2.$ For a given Banach space $X$ we consider the Bochner space for time dependent functions with values in $X$,  namely $C(0,T;X)$, $L^p(0,T;X).$
\item
For a given random variable $f$,  we denote by $\mathfrak{L}[f]$ the Law of $f$ or also by $\mathfrak{L}_X[f]$ the law of $f$ on the space $X$.  We denote by $L^p_{\text{prog}}(\Omega \times [0,T])$ the Lebesgue space of functions that are measurable with respect to the $\sigma$-field of $(\mathfrak{F}_t)$-progressively measurable sets in $\Omega \times [0,T]$ and given a stochastic process $U,$ $(\sigma_t [U])_{t \ge 0}$ represents the canonical filtration of $U.$
\end{enumerate}
Let $(\Omega, \mathfrak{F},(\mathfrak{F}_t)_{t \ge 0},\mathbb{P})$ be a stochastic basis with a complete right-continuous filtration, the driving process 
\begin{equation} \label{Cyl W}
W(t)=\sum_{k=1}^\infty e_kW_k(t), \quad t\in [0,T]
\end{equation}
is a cylindrical $(\mathfrak{F}_t)$-Wiener process,  i.e.
$ (W_k)_{k\in \mathbb{N}}$ is a family of mutually independent real-valued Brownian motions and $(e_k)_{k\in \mathbb{N}}$ are an orthonormal basis of a separable Hilbert space $\mathfrak{U}.$
The stochastic integral is understood in the  It$\hat{\text{o}}$ sense,  we refer to \cite{Feir} and \cite{Da Prato} for its construction  and to \cite{Mik1}-\cite{Mik2} for detailed discussion on physical interpretation of stochastic fluid flows. The diffusion coefficient $\mathbb{G}(\rho,q): \mathfrak{U} \rightarrow L^2(\mathbb{T})$ is defined as a superposition operator by 
\begin{equation*}
\mathbb{G}(\rho,q)e_k=G_k(\cdot,\rho(\cdot),q(\cdot)),
\end{equation*}
where $\rho \in L^2(\mathbb{T}),  \ \rho \ge 0,  \ q \in L^2(\mathbb{T})$ and the coefficients $$G_k: \mathbb{T} \times [0, \infty) \times \mathbb{R} \rightarrow \mathbb{R}$$ are $C^s$-functions, for $s \in \mathbb{N}$ specified below, which satisfy the following growth assumptions uniformly in $x \in \mathbb{T}$ 
\begin{equation}\label{G1}
G_k(\cdot,0,0)=0,
\end{equation}
\begin{equation}\label{G w}
| G_k(x,\rho,q) | \le g_k (\rho+|q|),
\end{equation}
\begin{equation}\label{G2}
| \partial^l_{x,\rho,q} G_k(x, \rho,q) | \le g_k,  \quad \sum_{k=1}^\infty g_k < \infty \quad  \text{for all} \; l\in {1,...,s}.
\end{equation}
If $\rho, \ q=\rho u$ are $(\mathfrak{F}_t)$-progressively measurable $L^2(\mathbb{T})$-valued stochastic processes such that $$ \rho \in L^2(\Omega \times [0,T]; L^2(\mathbb{T})), \quad q \in L^2(\Omega \times [0,T]; L^2(\mathbb{T}))$$ and $\mathbb{G}$ satisfies \eqref{G1}-\eqref{G2},  then the stochastic integral $$ \int_{0}^{t} \mathbb{G}(\rho,\rho u) \text{d}W= \sum_{k=1}^{\infty} \int_{0}^{t} G_k( \cdot, \rho, \rho u) \text{d}W_k$$
is a well defined $(\mathfrak{F}_t)$-martingale ranging in $L^2(\mathbb{T}).$ \\
We observe that in the case of $q=\rho u $,  by virtue of \eqref{G w} we can rewrite the noise coefficients as $G_k(x,\rho,\rho u)= \rho F_k(x,\rho,u)$ for suitable functions $F_k.$ Hence the noise vanishes on the vacuum regions as for the deterministic case where the external force is usually considered in the form $\rho f$. More precisely, $\mathbb{G}(\rho, \rho u)=\rho \mathbb{F}(\rho,u)$ where $\mathbb{F}=(F_k)_{k\in \mathbb{N}}$ is defined by $$F_k(\rho,u)=\dfrac{G_k(\rho,\rho u)}{\rho}$$
so that the coefficients $F_k: \mathbb{T} \times [0,\infty) \times \mathbb{R} \rightarrow \mathbb{R}, \quad F_k \in C^1(\mathbb{T} \times (0,\infty) \times \mathbb{R})$, satisfy
\begin{equation}\label{f1}
F_k(\cdot,0,0)=0,
\end{equation}
\begin{equation}\label{f2}
| \partial^l_{x,\rho,q} F_k(x, \rho,q) | \le \alpha_k,  \quad \sum_{k=1}^\infty \alpha_k < \infty \quad  \text{for all} \; l\in {1,...,s},
\end{equation}
and since we are interested in strong solution for which both $\rho$ and $u$ are bounded and $\rho$ is bounded from below away from zero, then \eqref{G2} holds for $F_k$ restricted for this range, and in addition we have 
\begin{equation}\label{fw}
\sum_{k=1}^{\infty} | F_k(\cdot{},\rho,u) | \le c(1+|u|).
\end{equation}
\\
Since the sum in \eqref{Cyl W} is not a priori expected to be converging in $\mathfrak{U}$,  we construct a larger space $\mathfrak{U}_0 \supset \mathfrak{U}$ as follows
\begin{equation}
\mathfrak{U}_0= \bigg\{ v= \sum_{k \ge 1}^{} c_k e_k; \quad \sum_{k \ge 1}^{} \dfrac{c^2_k}{k^2} < \infty \bigg\}
\end{equation}
endowed with the norm
\begin{equation*}
\| v \|^2_{\mathfrak{U}_0}= \sum_{k \ge 1}^{} \dfrac{c^2_k}{k^2}, \quad v= \sum_{k \ge 1}^{} c_k e_k
\end{equation*}
Furthermore,  $W$ has $\mathbb{P}$-a.s. $C([0,T];\mathfrak{U}_0)$ sample paths and the embedding $\mathfrak{U} \hookrightarrow \mathfrak{U}_0$ is Hilbert-Schmidt, see Section 2.3 in \cite{Feir}.\\
By virtue of the hypothesis \eqref{G1}-\eqref{G2} and the regularity of the solution $(\rho,u),$ our analysis holds also for the case of an additive noise $\sigma(x)\text{d}W$ i.e.  a noise which does not depend on the solution. This is of interest since perturbation of deterministic PDEs by additive noise are frequently considered in order to study existence and uniqueness of invariant measures and related phenomena arising in the ergodic theory,  see \cite{Coti} for the compressible Navier Stokes equations. We stress that random effects are considered in both the stochastic force and the random initial data.  The initial data $(\rho_0,u_0)$ are indeed $\mathfrak{F}_0$-measurable, $H^{s+1}(\mathbb{T})\times H^{s}(\mathbb{T})$, with $s>7/2$, valued random variables such that there exists a deterministic constant $C>0$ such that 
\begin{equation} \label{C.I STRONG}
 C^{-1} \le \rho_0 \le C,  \quad \rho_0 \in H^{s+1}(\mathbb{T}),  \quad u_0 \in H^s(\mathbb{T}) \quad \mathbb{P}\text{-a.s.}
\end{equation}
In order to perform the a priori estimates provided in Section 3,  we also assume the following regularity on the momenta of the initial data
\begin{equation} \label{C.I. Momenta strong}
\rho_0 \in L^p(\Omega; H^{s+1}(\mathbb{T})),  \quad u_0 \in L^p(\Omega; H^s(\mathbb{T})), \quad \text{for all} \;p  \in [1, \infty).
\end{equation}
\subsection{Definitions and main theorems}
Our analysis focuses on the existence of strong pathwise solutions to \eqref{stoc quantum}-\eqref{C.I} i.e.  solutions which are strong both in PDEs and probability sense. This means that the system \eqref{stoc quantum} is satisfied pointwise and the stochastic integral is well defined for any fixed $x$. Furthermore this solutions are defined on a given probability space with a given cylindrical Wiener process which is not to be considered as part of the solution itself as in the case of weak solution in probability.  In order to establish the existence of a strong pathwise solution we construct an approximating system for which we prove the existence of martingale solutions i.e.  a weaker notion of solution in probability sense. Then,  after having derived a pathwise uniqueness estimate we will make use of the Gy\"ongy-Krylov method, see Section $2.10$ in \cite{Feir}, which allows to establish the existence of a strong pathwise solution to system \eqref{stoc quantum}. \\
In particular we prove the local well posedness of the problem up to a maximal stopping time $\tau$ which depends on the $W^{2,\infty}$ norm of the solution $(\rho,u)$ and we derive some uniform estimates in the case of the viscosity exponent $\alpha \in [0, \frac{1}{2}],$ which allow us to extend the local in time strong solution to a global one.  For the local existence results we adapt the techniques developed in the manuscript \cite{Feir} for compressible flows to system \eqref{stoc quantum}.  \\
We point out that,  even in the deterministic setting,  system \eqref{stoc quantum} should not be seen as a perturbation of the compressible Navier-Stokes equations. Indeed the presence of the third-order term leads to a different linear structure, and the unknowns must lie in a different regularity class, namely $$(\rho,u) \in H^{s+1}(\mathbb{T}) \times H^s (\mathbb{T}).$$
In particular, both in the a priori estimates and in the local existence results it is important to keep track of appropriate cancellations of the high order derivatives.\\
\\
We start by introducing the definition of strong pathwise solution
\begin{definition}(Local strong pathwise solution). \label{Def1}
Let $(\Omega, \mathfrak{F},(\mathfrak{F}_t)_{t \ge 0},\mathbb{P}), $ be a stochastic basis with a complete right-continuous filtration and let $W$ be an $(\mathfrak{F}_t)$-cylindrical Wiener process. Let $(\rho_0,u_0)$ be an $\mathfrak{F}_0-$measurable random variable with regularity \eqref{C.I STRONG} for some $s > \frac{5}{2}.$
A triplet $(\rho,u,\tau)$ is called a local strong pathwise solution to system \eqref{stoc quantum}-\eqref{C.I}, provided
\begin{itemize}
\item[(1)]
 $\tau$ is a $\mathbb{P}-a.s.$ strictly positive $(\mathfrak{F}_t)$-stopping time;
 \item[(2)]
 the density $\rho$ is a $H^{s+1}(\mathbb{T})$-valued $(\mathfrak{F}_t)$-progressively measurable stochastic process such that 
 $$ \rho(\cdot \land \tau) >0, \quad \rho(\cdot \land \tau) \in C([0,T];H^{s+1}(\mathbb{T})) \quad \mathbb{P}-a.s.; $$
 \item[(3)]
 the velocity $u$ is a $H^s(\mathbb{T})$-valued $(\mathfrak{F}_t)$-progressively measurable stochastic process such that 
 $$ u(\cdot  \land \tau) >0, \quad u(\cdot \land \tau) \in C([0,T];H^s(\mathbb{T})) \quad \mathbb{P}-a.s.; $$
\item[(4)]
the continuity equation $$ \rho(t \land \tau)= \rho_0- \int_{0}^{t\land \tau} \partial_x(\rho u) ds $$
holds for all $t \in [0,T] \ \mathbb{P}-a,s,;$
\item[(5)] the momentum equation 
\begin{equation*}
\begin{split}
(\rho u)(t \land \tau)&= \rho_0 u_0 -\int_{0}^{t \land \tau} \partial_x(\rho u^2+p(\rho))ds +\int_{0}^{t \land \tau} \partial_x(\mu(\rho) \partial_x u)ds \\ &+\int_{0}^{t \land \tau} \rho \partial_x \bigg( \dfrac{\partial_{xx}\sqrt{\rho}}{\sqrt{\rho}}\bigg)ds + \int_{0}^{t \land \tau} \mathbb{G}(\rho,\rho u)dW
\end{split}
\end{equation*}
holds for all $t \in [0,T] \ \mathbb{P}-a,s,;$
\end{itemize}
\end{definition}
\noindent
Now we introduce the definition of maximal and global strong pathwise solution
\begin{definition}(Maximal strong pathwise solution). \label{Def2}
Let $(\Omega, \mathfrak{F},(\mathfrak{F}_t)_{t \ge 0},\mathbb{P}), $ be a stochastic basis with a complete right-continuous filtration and let $W$ be an $(\mathfrak{F}_t)$-cylindrical Wiener process. Let $(\rho_0,u_0)$ be an $\mathfrak{F}_0-$measurable random variable in the space $H^{s+1}(\mathbb{T})\times H^s(\mathbb{T}).$
A quadruplet $(\rho,u,(\tau_R)_{R \in \mathbb{N}},\tau)$ is called a maximal strong pathwise solution to system \eqref{stoc quantum}-\eqref{C.I}, provided
\begin{itemize}
\item[(1)]
 $\tau$ is a $\mathbb{P}-a.s.$ strictly positive $(\mathfrak{F}_t)$-stopping time;
 \item[(2)] 
 $(\tau_R)_{R \in \mathbb{N}}$ is an increasing sequence of $(\mathfrak{F}_t)$-stopping times such that $\tau_R < \tau$ on the set $[\tau < \infty], \ \lim_{R \rightarrow \infty} \tau_R=\tau$ a.s. and 
\begin{equation}\label{limit cond}
\max \bigg\{ \sup_{t \in [0,\tau_R]} \| \log\rho (t) \|_{W_x^{2,\infty}},\;\sup_{t \in [0,\tau_R]}  \| u(t) \|_{W_x^{2,\infty}} \bigg\} \ge R \ \text{in} \ [\tau < \infty];
\end{equation}
\item[(3)]
each triplet $(\rho,u,\tau_R),  R\in \mathbb{N},$ is a local strong pathwise solution in the sense of Definition \ref{Def1}
\end{itemize}
If $(\rho,u,(\tau_R)_{R \in \mathbb{N}},\tau)$ is a maximal strong pathwise solution  and $\tau=\infty$ a.s.  then we say that the solution is global. 
\end{definition}
\noindent
\newline
The maximal life span of the solution is determined by its $W^{2,\infty}$ norm. Condition \eqref{limit cond} indeed implies that for local maximal strong pathwise solutions
\begin{equation*}
H(\tau)=\max \bigg\{ \sup_{t \in [0, \tau)} \| \log \rho(t) \|_{W^{2,\infty}},  \sup_{t \in [0, \tau)} \| u(t) \|_{W^{2,\infty}} \bigg\} = \infty \quad  \text{in} \ [\tau < \infty],
\end{equation*}
while concerning global maximal strong pathwise solutions,  where $\tau= \infty$,  $H(\infty)$ may be finite or infinite.  We refer the reader to \cite{Deb} for a similar discussion for global solutions of incompressible fluid equations. \\
We point out that the announcing stopping time $\tau_R$ is not unique,  indeed uniqueness has to be understood in the sense that the solution $(\rho,u)$ and the maximal stopping time $\tau$ are unique.
The main results of our paper are the following two theorems.  Theorem \ref{Main Theorem local} establishes existence and uniqueness of a local in time maximal strong pathwise solution.  The second is an extension theorem that allows to extend the local solution derived in Theorem \ref{Main Theorem local} to a global one,  however an additional hypothesis on the viscosity exponent $\alpha$ is required.
\begin{theorem}\label{Main Theorem local} (Local existence)
Let $s \in \mathbb{N}$ satisfy $ s > \frac{7}{2}, \; \gamma > 1.$ Let the coefficients $G_k$ satisfy \eqref{G1}-\eqref{G2} and $(\rho_0,u_0)$ be an $\mathfrak{F}_0$-measurable, $H^{s+1}(\mathbb{T}) \times H^s(\mathbb{T})$- valued random variable satisfying \eqref{C.I STRONG}. Then there exists a unique maximal strong pathwise solution $(\rho,u,(\tau_R)_{R \in \mathbb{N}},\tau)$ to problem \eqref{stoc quantum}-\eqref{C.I} in the sense of Definition \ref{Def2} with initial condition $(\rho_0,u_0).$
\end{theorem}
\begin{theorem}\label{Main Theorem global} (Global existence) 
Under the same hypothesis of Theorem \ref{Main Theorem local},  assume in addition that $0 \le \alpha \le \frac{1}{2}$ and \eqref{C.I. Momenta strong},   then there exists a unique maximal global strong pathwise solution $(\rho,u,(\tau_R)_{R \in \mathbb{N}},\tau)$ to problem \eqref{stoc quantum}-\eqref{C.I} in the sense of Definition \ref{Def2} with initial condition $(\rho_0,u_0).$
\end{theorem}
\begin{remark}
Our results hold for more general pressure laws satisfying $$ p \in C^1((0,\infty)) \cap C([0,\infty)), \quad p'(\rho) \ge 0 \;  \text{for} \; \rho >0, \quad p'(\rho) \approx \rho^{\gamma-1} \; \text{for} \; \rho \rightarrow \infty.$$
\end{remark}
\begin{remark}
Theorem \ref{Main Theorem local} provides a local existence result for strong patwhise solution in the case of general viscosity $\mu(\rho)=\rho^\alpha$ with $\alpha \ge 0$. The restriction $\alpha \in [0,\frac{1}{2}]$ is needed only in Theorem \ref{Main Theorem global} in order to extend the local maximal strong patwhise solution to a global one. 
\end{remark}
\begin{remark}
Concerning the regularity exponent $s$ we point out that it is chosen in such a way that the embedding $H^s(\mathbb{T}) \hookrightarrow W^{2,\infty}(\mathbb{T})$ holds,  hence we have $s> \frac{N}{2}+2=\frac{5}{2}.$
The higher integrability condition requested in Theorem \ref{Main Theorem local} and Theorem \ref{Main Theorem global} that is $s> \frac{N}{2}+3= \frac{7}{2}$ is needed in order to derive a pathwise uniqueness estimate which is essential in order to apply our convergence argument.  Moreover since the solution satisfies $$ \rho( \cdot \land \tau_R) \in C([0,T]; H^{s+1}(\mathbb{T})), \quad u( \cdot \land \tau_R) \in C([0,T]; H^s(\mathbb{T})) \quad  \mathbb{P}-a.s.$$
then by continuity in $[0,\tau)$ we have that a blow up of the $H^s$ norm of the solution coincides with a blow up in the $W^{2,\infty}$ norm at time $\tau.$
\end{remark}
\noindent
In the deterministic case,  that is $\mathbb{G}=0,$ our argument for the local well posedness applies also for a differentiability space exponent $s'=s-1.$ This difference is due to the construction of the local strong solution in the stochastic setting. \\
In order to prove Theorem \ref{Main Theorem local} we will use an approximating system introducing suitable smooth cut-off operators $ \varphi_R \; : \: [0,\infty) \rightarrow [0,1]$ satisfying 
\begin{equation*}
\varphi_R(y)=\begin{cases} 1,  \quad 0 \le y \le R,  \\ 0, \quad R+1 \le y
\end{cases}
\end{equation*}
applied to the $W^{2,\infty}$ norm of the solution and we will solve it by means of a stochastic compactness method.
Concerning the extension argument used in order to prove Theorem \ref{Main Theorem global},  we derive a global in time $H^{s+1}(\mathbb{T}) \times H^s(\mathbb{T})$ estimates which allows us to control the blow up of the $W^{2,\infty}$ norm stated in Definition \ref{Def2} and as a consequence we get $\tau= \infty$. The key point to control the  $W^{2,\infty}$ norm of the solution will be given by the fact that the Bresch-Desjardin Entropy, see Proposition \ref{BD prop} below, allows us to prove that almost surely vacuum region cannot be formed if they are not present at the initial time, see Proposition \ref{prop:vacuum}.

\section{Global regularity}

This section is dedicated to the proof of Theorem \ref{Main Theorem global}.  In particular, for the rest the present section, we fix an arbitrary finite time $T>0$, and we assume the existence of a maximal strong pathwise solution  of the system \eqref{stoc quantum} in the regularity class of Theorem \ref{Main Theorem local}. Then, the proof is based on high order energy estimates which hold uniformly in $T$ and a standard continuity argument.\\

\subsection{A priori estimates}
We start by proving the following standard energy estimate
\begin{proposition} (Energy estimate) \\
Let $(\rho,u)$ be a strong solution of \eqref{stoc quantum}-\eqref{C.I}, then the following energy estimate holds
\begin{equation}\label{energy visk}
\begin{split}
& \mathbb{E} \bigg| \sup_{t \in [0,T \land \tau]} \int_{\mathbb{T}} \bigg[\dfrac{1}{2}\rho |u|^2+ \dfrac{\rho^\gamma}{\gamma-1}+ | \partial_x \sqrt{\rho}|^2 \bigg] dx \bigg|^p +\mathbb{E} \bigg| \int_{0}^{T\land \tau} \int_{\mathbb{T}} \mu(\rho)| \partial_x u |^2 dxdt \bigg|^p \\ & \lesssim 1+ \mathbb{E}\bigg| \sup_{t \in [0,T\land \tau]} \int_{\mathbb{T}} \bigg[\dfrac{1}{2}\rho_0 |u_0|^2+ \dfrac{\rho_0^\gamma}{\gamma-1}+ | \partial_x \sqrt{\rho_0}|^2 \bigg] dx \bigg|^p,
\end{split}
\end{equation}
for any $p\ge 1.$
\end{proposition}
\begin{proof}
Applying It$\hat{\text{o}}$ lemma, Theorem \ref{Ito Lemma} for $s=\rho, \; r=u, \;  Q(r)=\frac{1}{2}|u|^2$  to the functional 
\begin{equation}
F(\rho,u)(t) =\int_{\mathbb{T}} \frac{1}{2}\rho(t)|u(t)|^2dx,
\end{equation}
we get:
\begin{equation} \label{En F}
\begin{split}
& \int_{\mathbb{T}} \frac{1}{2} \rho |u|^2(t)dx=\int_{\mathbb{T}} \frac{1}{2} \rho_0 |u_0|^2dx-\int_{0}^{t} \int_{\mathbb{T}} [ \rho u^2 \partial_x u +u\partial_x p(\rho)]dxds \\ & +\int_{0}^{t} \int_{\mathbb{T}} u\partial_x(\mu(\rho)\partial_x u)+ \rho u \partial_x \bigg( \dfrac{\partial_{xx} \sqrt{\rho}}{\sqrt{\rho}} \bigg)dxds+\frac{1}{2}\int_{0}^{t} \int_{\mathbb{T}} \sum_{k\in \mathbb{N}} \rho | F_k(\rho,u)|^2dxdt \\ & + \int_{0}^{t} \int_{\mathbb{T}} \frac{1}{2} |u|^2 \partial_x(\rho u )dxds+ \int_{0}^{t} \int_{\mathbb{T}} \sum_{k \in \mathbb{N}} \rho u \mathbb{F}(\rho,u)dxdW
\end{split}
\end{equation}
and we observe that since $-\rho u^2 \partial_x u=-\frac{1}{2} \rho u \partial_x ( |u|^2)$ then the second term in the right hand side of \eqref{En F} $\rho u^2 \partial_x u=\frac{1}{2} \rho u \partial_x |u|^2$ cancels with the last non noise term after integrating by parts.  Moreover, by using the continuity equation
\begin{equation*}
\begin{split}
 & \int_{} \partial_x \rho^\gamma u=\dfrac{d}{dt} \int_{} \dfrac{\rho^{\gamma}}{\gamma-1}.
  \end{split}
\end{equation*}
Concerning the quantum term we have
\begin{equation*}
\begin{split}
& \int_{} \rho u \partial_x (\dfrac{\partial_{xx} \sqrt{\rho}}{\sqrt{\rho}})= -\int_{} \partial_x (\rho u)(\dfrac{\partial_{xx} \sqrt{\rho}}{\sqrt{\rho}})= \int_{} \partial_t \rho(\dfrac{\partial_{xx} \sqrt{\rho}}{\sqrt{\rho}})= \int_{} 2 \sqrt{\rho} \partial_t \sqrt{\rho} (\dfrac{\partial_{xx} \sqrt{\rho}}{\sqrt{\rho}}) \\ & =2 \int_{} \partial_t \sqrt{\rho} \partial_{xx} \sqrt{\rho}= -2 \int_{} \partial_t \partial_x \sqrt{\rho} \partial_x \sqrt{\rho}=-\dfrac{d}{dt} \int_{} |\partial_x \sqrt{\rho}|^2.
\end{split}
\end{equation*}
Hence the energy balance \eqref{En F} can be rewritten as 
\begin{equation}
\begin{split}
& \int_{\mathbb{T}} \dfrac{1}{2} \rho |u|^2+ \dfrac{\rho^\gamma}{\gamma-1}+ | \partial_x \sqrt{\rho}|^2 dx + \int_{0}^{t} \int_{\mathbb{T}} \mu(\rho)| \partial_x u |^2 dxds\\ & =\frac{1}{2} \sum_{k=1}^{\infty}\int_{0}^{t}\int_{\mathbb{T}} \rho |F_k|^2dxds+\int_{0}^{t} \int_{\mathbb{T}} \rho u \mathbb{F}(\rho,u) dxdW.
\end{split}
\end{equation}
By taking the $\sup$ in time,  the $p$-th power and applying the expectation we have
\begin{equation}\label{energy stoch}
\begin{split}
& \mathbb{E} \bigg| \sup_{t \in [0.T\land \tau]} \int_{} \bigg[ \dfrac{1}{2} \rho |u|^2+ \dfrac{\rho^\gamma}{\gamma-1}+ | \partial_x \sqrt{\rho}|^2 \bigg] dx \bigg|^p + \mathbb{E} \bigg| \int_{0}^{T\land \tau} \int_{} \mu(\rho)| \partial_x u |^2 dxdt \bigg|^p \\ &  \lesssim \mathbb{E} \bigg | \frac{1}{2} \sum_{k=1}^{\infty}\int_{0}^{T\land \tau}\int_{} \rho |F_k|^2dxdt \bigg|^p+\mathbb{E} \bigg| \int_{0}^{T\land \tau} \int_{} \rho u \mathbb{F}(\rho,u) dxdW \bigg|^p.
\end{split}
\end{equation}
The two terms in the right hand side of \eqref{energy stoch} can be estimated as follows
\begin{equation}
\begin{split}
\sum_{k \in \mathbb{N}} \int_{0}^{T\land \tau} \int_{\mathbb{T}} \rho |F_k|^2dxdt \le \int_{0}^{T\land \tau} \sum_{k \in \mathbb{N}} \alpha^2_k \int_{\mathbb{T}} ( \rho + \rho |u|^2) dxdt \lesssim \int_{0}^{T\land \tau} \int_{\mathbb{T}} ( \rho+ \rho |u|^2) dxdt,
\end{split}
\end{equation}
while the stochastic integral can be estimated by means of the Burkholder-Davis-Gundy inequality,  Proposition \ref{BDG}, as 
\begin{equation}\label{stoch int enrgy estimate}
\begin{split}
& \mathbb{E} \bigg| \int_{0}^{t} \int_{} \rho u \mathbb{F}(\rho,u) dxdW \bigg|^p \le  \mathbb{E} \bigg[ \sup_{t \in [0.T\land \tau]} \bigg|  \int_{0}^{t} \int_{} \rho u \mathbb{F}(\rho,u) dxdW \bigg|^p \bigg] \\ & \lesssim \mathbb{E} \bigg[ \int_{0}^{T\land \tau} \sum_{k \in \mathbb{N}} \bigg| \int_{\mathbb{T}} \rho u F_k dx \bigg|^2 dt \bigg]^{\frac{p}{2}}\le \mathbb{E}\bigg[ \int_{0}^{T\land \tau}  \bigg| \int_{\mathbb{T}} ( \rho+ \rho |u|^2) dx  \bigg|^2 dt \bigg]^{\frac{p}{2}}.
\end{split}
\end{equation}
Finally applying Young inequality in \eqref{stoch int enrgy estimate} and using the Gronwall Lemma we get 
\begin{equation}
\begin{split}
& \mathbb{E} \bigg| \sup_{t \in [0,T\land \tau]} \int_{} \bigg[\dfrac{1}{2}\rho |u|^2+ \dfrac{\rho^\gamma}{\gamma-1}+ | \partial_x \sqrt{\rho}|^2 \bigg] dx \bigg|^p +\mathbb{E} \bigg| \int_{0}^{T\land \tau} \int_{} \mu(\rho)| \partial_x u |^2 dxdt \bigg|^p \\ & \lesssim 1+ \mathbb{E}\bigg| \sup_{t \in [0,T\land \tau]} \int_{} \bigg[\dfrac{1}{2}\rho_0 |u_0|^2+ \dfrac{\rho_0^\gamma}{\gamma-1}+ | \partial_x \sqrt{\rho_0}|^2 \bigg] dx \bigg|^p 
\end{split}
\end{equation}
which concludes our proof.
\end{proof}
\noindent
Hence we end up with the following regularity estimates for the unknowns $(\rho,u)$
\begin{equation}\label{reg energy}
\begin{split}
& \sqrt{\rho}u \in L^p(\Omega; L^\infty(0,T\land \tau; L^2 (\mathbb{T}))), \quad \rho \in L^p(\Omega; L^\infty(0,T\land \tau; L^\gamma (\mathbb{T}))) \\ & \partial_x \sqrt{\rho} \in L^p(\Omega; L^\infty(0,T\land \tau; L^2 (\mathbb{T}))), \quad \sqrt{\mu(\rho)}\partial_x u \in L^p(\Omega; L^2(0,T\land \tau; L^2 (\mathbb{T}))).
\end{split}
\end{equation}
In the following Lemma we derive a Bresch-Desjardins type entropy inequality.
\begin{proposition} \label{BD prop}(B.D Entropy estimate) \\
Let $(\rho,u)$ be a strong solution of \eqref{stoc quantum}-\eqref{C.I}, then the following entropy estimate holds
\begin{itemize}
\item[1.]
for $\alpha=0$
\begin{equation}\label{BD ENTROPY}
\begin{split}
& \mathbb{E} \bigg| \sup_{t \in [0,T\land \tau]} \int_{\mathbb{T}} \bigg[\dfrac{1}{2}\rho |V|^2+ \dfrac{\rho^\gamma}{\gamma-1}+ | \partial_x \sqrt{\rho}|^2 \bigg] dx \bigg|^p \\ & +\mathbb{E} \bigg|  \dfrac{4 \gamma \mu}{(\gamma-1)^2} \int_{0}^{T\land \tau}\int_{\mathbb{T}} | \partial_x \rho^{\frac{\gamma-1}{2} }|^2 dxdt + \dfrac{1}{2} \int_{0}^{T\land \tau} \int_{\mathbb{T}} (\partial_{xx} \log \rho )^2 dxdt \bigg|^p \\ & \lesssim 1+ \mathbb{E}\bigg| \sup_{t \in [0,T\land \tau]} \int_{\mathbb{T}} \bigg[\dfrac{1}{2}\rho_0 |V_0|^2+ \dfrac{\rho_0^\gamma}{\gamma-1}+ | \partial_x \sqrt{\rho_0}|^2 \bigg] dx \bigg|^p 
\end{split}
\end{equation}
for any $p \ge 1.$
\item[2.]
for $\alpha \neq 0$
\begin{equation}\label{BD ENTROPY visk}
\begin{split}
& \mathbb{E} \bigg| \sup_{t \in [0,T\land \tau]} \int_{\mathbb{T}} \bigg[\dfrac{1}{2}\rho |V|^2+ \dfrac{\rho^\gamma}{\gamma-1}+ | \partial_x \sqrt{\rho}|^2 \bigg] dx \bigg|^p \\ & +\mathbb{E} \bigg|  \dfrac{4 \gamma}{(\gamma+\alpha -1)^2} \int_{0}^{T\land \tau}\int_{\mathbb{T}} | \partial_x \rho^{\frac{\gamma+\alpha-1}{2} } |^2 dxdt \\ & + \dfrac{4}{\alpha^2} \int_{0}^{T\land \tau} \int_{\mathbb{T}} | \partial_{xx} 	\rho^{\frac{\alpha}{2}} |^2 dxdt +\dfrac{4(4-3 \alpha) }{3 \alpha^3}\int_{0}^{T\land \tau} \int_{\mathbb{T}} \rho^{-\alpha} | \partial_x \rho^\frac{\alpha}{2}|^4 dxdt  \bigg|^p \\ & \lesssim 1+ \mathbb{E}\bigg| \sup_{t \in [0,T\land \tau]} \int_{\mathbb{T}} \bigg[\dfrac{1}{2}\rho_0 |V_0|^2+ \dfrac{\rho_0^\gamma}{\gamma-1}+ | \partial_x \sqrt{\rho_0}|^2 \bigg] dx \bigg|^p 
\end{split}
\end{equation}
for any $p\ge 1,$
\end{itemize}
where $V$ is the effective velocity  $$V= u+ Q$$
and $$Q= \mu(\rho) \dfrac{\partial_x \rho}{\rho^2}$$
is the solution of the following transport equation 
$$ \partial_t Q+ u \partial_x Q= - \rho^{-1}\partial_x(\mu (\rho) \partial_{x} u )$$

\end{proposition}
\begin{proof}
We rewrite system \eqref{stoc quantum} in the new variables $(\rho, \rho V)$
\begin{equation}\label{ sist W}
\begin{cases}
\text{d}\rho+\partial_x (\rho u)\text{d}t=0 \\
\text{d}(\rho V)+ \partial_x (\rho uV)\text{d}t+ \partial_x \rho^{\gamma}\text{d}t=\rho \partial_x (\dfrac{\partial_{xx} \sqrt{\rho}}{\sqrt{\rho}})\text{d}t+\rho \mathbb{F}(\rho,u) \text{d}W
\end{cases}
\end{equation}
and, similarly to the energy estimate, we apply It$\hat{\text{o}}$ formula for $s=\rho, \;r=V,\; Q(r)=\frac{1}{2}|V|^2$  to the functional 
\begin{equation}
F(\rho,V)(t) =\int_{\mathbb{T}} \frac{1}{2}\rho(t)|V(t)|^2dx,
\end{equation}
we get:
\begin{equation} \label{Ent F}
\begin{split}
& \int_{\mathbb{T}} \frac{1}{2} \rho |V|^2(t)dx=\int_{\mathbb{T}} \frac{1}{2} \rho_0 |V_0|^2dx-\int_{0}^{t} \int_{\mathbb{T}} [ \rho V u \partial_x V +V\partial_x p(\rho)]dxds \\ & +\int_{0}^{t} \rho V \partial_x \bigg( \dfrac{\partial_{xx} \sqrt{\rho}}{\sqrt{\rho}} \bigg)dxds+\frac{1}{2}\int_{0}^{t} \int_{\mathbb{T}} \sum_{k\in \mathbb{N}} \rho | F_k(\rho,u)|^2dxdt \\ & + \int_{0}^{t} \int_{\mathbb{T}} \frac{1}{2} |V|^2 \partial_x(\rho u )dxds+ \int_{0}^{t} \int_{\mathbb{T}} \sum_{k \in \mathbb{N}} \rho V \mathbb{F}(\rho,u)dxdW.
\end{split}
\end{equation}
We now estimate the new terms appearing in the right hand side of \eqref{Ent F}.
The pressure term can be rewritten as \\ \\
for $\alpha=0$
\begin{equation*}
\begin{split}
& \int_{} \partial_x \rho^\gamma Q= \int_{} \partial_x \rho ^\gamma \mu \dfrac{\partial_x \rho}{\rho^2}= \gamma \mu \int_{} \rho^{\gamma-3} | \partial_x \rho |^2 = \dfrac{4 \gamma \mu}{(\gamma-1)^2} \int_{} | \partial_x \rho^{\frac{\gamma-1}{2}}|^2,
\end{split}
\end{equation*}
for $\alpha \neq 0$
\begin{equation*}
\begin{split}
& \int_{} \partial_x \rho^\gamma Q= \int_{} \partial_x \rho ^\gamma \rho^{\alpha-2}\partial_x \rho= \gamma \int_{}\rho^{\gamma-1} \partial_x \rho \rho^{\alpha-2} \partial_x \rho = \dfrac{4\gamma}{(\gamma+\alpha-1)^2}\int_{} | \partial_x \rho^{\frac{\gamma+\alpha-1}{2}} |^2.
\end{split}
\end{equation*}
Regarding the quantum term we observe that for $\alpha=0$
\begin{equation*}
\begin{split}
& \int_{} \rho Q \partial_x (\dfrac{\partial_{xx} \sqrt{\rho}}{\sqrt{\rho}})= \int_{} \mu \dfrac{\partial_x \rho}{\rho} \partial_x (\dfrac{\partial_{xx} \sqrt{\rho}}{\sqrt{\rho}})= \mu \int_{} \partial_x \log \rho \partial_x (\dfrac{\partial_{xx} \sqrt{\rho}}{\sqrt{\rho}}) \\ & =-\mu \int_{} \partial_{xx} \log \rho \partial_x ( \dfrac{\partial_x  \sqrt{\rho}}{\sqrt{\rho}})- \mu \int_{} \partial_{xx} \log \rho \dfrac{\partial_x \sqrt{\rho}}{\sqrt{\rho}} \dfrac{\partial_x  \sqrt{\rho}}{\sqrt{\rho}} \\ & =-\dfrac{1}{2} \mu  \int_{} \partial_{xx} \log \rho \partial_{xx} \log \rho- \dfrac{1}{4} \mu \int_{} \partial_{xx} \log \rho \partial_x \log \rho \partial_x \log \rho  = - \dfrac{1}{2} \mu \int_{} (\partial_{xx} \log \rho)^2.
\end{split}
\end{equation*}
The estimate of the second quantum term for $\alpha \neq 0$ is quite technical.  First we observe that
\begin{equation*}
\begin{split}
& I= \int_{} \partial_x (\rho Q) \frac{\partial_{xx}\sqrt{\rho}}{\sqrt{\rho}}= \int_{} \partial_x \bigg( \rho \rho^{\alpha-2} \partial_x \rho \bigg) \frac{\partial_{xx}\sqrt{\rho}}{\sqrt{\rho}}= \frac{1}{(\alpha-1)} \int_{} \partial_x \bigg ( \rho \partial_x {\rho^{\alpha-1}} \bigg )\frac{\partial_{xx}\sqrt{\rho}}{\sqrt{\rho}} \\ & =-\frac{1}{\alpha-1} \int_{} \partial_x {\rho^{\alpha-1}} \rho \partial_x \bigg( \frac{\partial_{xx}\sqrt{\rho}}{\sqrt{\rho}} \bigg)= -\frac{1}{(\alpha-1)} \int_{} \partial_x {\rho^{\alpha-1}}\partial_x ( \rho \partial_{xx} \log \rho ) \\ & = \frac{1}{(\alpha-1)} \int_{} \partial_x ( \partial_x {\rho^{\alpha-1}}) \rho \partial_{xx} \log \rho = \frac{1}{(\alpha-1)}\int_{} \partial_x(\partial_x \rho^{\alpha-1})\bigg[ \partial_x(\rho \partial_x \log 	\rho)-\partial_x\rho \partial_x \log \rho\bigg].
\end{split}
\end{equation*}
Now we recall the following elementary identities for a general exponent $\theta$ that we will properly choose in the sequel
$$ \partial_x \rho= \dfrac{\rho^{1-\theta}}{\theta} \partial_x \rho^\theta \text{ for all} \ \theta \ne 0 $$
$$ \partial_x \rho= \rho \partial_x \log \rho \ \text{for} \ \theta=0$$
and we rewrite
\begin{equation*}
\partial_x \rho^{\alpha-1} = (\alpha-1) \rho^{\alpha-2} \partial_x \rho= \frac{(\alpha-1)}{\theta} \rho^{\alpha-2+1-\theta} \partial_x \rho^\theta=\frac{(\alpha-1)}{\theta} \rho^{\alpha-\theta-1} \partial_x \rho^\theta
\end{equation*}
\begin{equation*}
\begin{split}
& \partial_x(\partial_x \rho^{\alpha-1})= \dfrac{(\alpha-1)(\alpha-\theta-1)}{\theta}\rho^{\alpha-\theta-2} \partial_x \rho \partial_x \rho^\theta+ \dfrac{(\alpha-1)}{\theta} \rho^{\alpha-\theta-1} \partial_{xx} \rho^\theta \\ & =\dfrac{(\alpha-1)(\alpha-\theta-1)}{\theta^2}\rho^{\alpha-\theta-2+1-\theta} | \partial_x \rho^\theta |^2+\dfrac{(\alpha-1)}{\theta} \rho^{\alpha-\theta-1} \partial_{xx} \rho^\theta \\ & =\dfrac{(\alpha-1)(\alpha-\theta-1)}{\theta^2} \rho^{\alpha-2\theta-1} | \partial_x \rho^\theta |^2+\dfrac{(\alpha-1)}{\theta} \rho^{\alpha-\theta-1} \partial_{xx} \rho^\theta
\end{split}
\end{equation*}
\begin{equation*}
\rho \partial_x \log \rho= \dfrac{\rho^{1-\theta}}{\theta} \partial_x \rho^\theta
\end{equation*}
\begin{equation*}
\partial_x(\rho \partial_x \log \rho)= \dfrac{1-\theta}{\theta} \rho^{-\theta} \partial_x \rho \partial_x \rho^\theta + \dfrac{\rho^{1-\theta}}{\theta} \partial_{xx} \rho^{\theta}
\end{equation*}
\begin{equation*}
\partial_x \rho \partial_x \log \rho= \dfrac{\rho^{1-\theta}}{\theta} \partial_x \rho^\theta \dfrac{\rho^{-\theta}}{\theta} \partial_x \rho^\theta= \dfrac{\rho^{1-2\theta}}{\theta^2}| \partial_x \rho^\theta |^2
\end{equation*}
substituting these identities in $I$ we get
\begin{equation}
I=\int_{} \bigg[ \dfrac{-\alpha^2+ \alpha(1+2\theta)-3\theta^2}{3\theta^4}\bigg] \rho^{\alpha-4\theta} | \partial_x \rho^\theta |^4+ \int_{} \dfrac{1}{\theta^2} \rho^{\alpha-2\theta} | \partial_{xx} \rho^\theta|^2
\end{equation}
which by choosing $\theta=\dfrac{\alpha}{2}$ becomes
\begin{equation}\label{BD quantum}
\begin{split}
& I=\int_{} \bigg[ \dfrac{-\alpha^2+ \alpha(1+\alpha)-3\alpha^2/4
}{3\alpha^4/16}\bigg] \rho^{-\alpha} | \partial_x \rho^{\frac{\alpha}{2}} |^4+ \int_{} \dfrac{4}{\alpha^2} | \partial_{xx} \rho^{\frac{\alpha}{2}}|^2 \\ & = \dfrac{4(4-3 \alpha) }{3 \alpha^3}\int_{} \rho^{-\alpha} | \partial_x \rho^\frac{\alpha}{2}|^4+\dfrac{4}{\alpha^2} \int_{} | \partial_{xx} 	\rho^{\frac{\alpha}{2}} |^2
\end{split}
\end{equation}
Hence by taking the $\sup$ in time we get the following balance law
\begin{itemize}
\item[1.]
for $\alpha=0$
\begin{equation}\label{alpha 0 entropy stoch}
\begin{split}
& \sup_{t \in [0,T\land \tau]} \int_{\mathbb{T}} \dfrac{1}{2}\rho |V|^2+ \dfrac{\rho^\gamma}{\gamma-1}+ | \partial_x \sqrt{\rho}|^2 dx + \dfrac{4 \gamma \mu}{(\gamma-1)^2} \int_{0}^{T\land \tau}\int_{\mathbb{T}} | \partial_x \rho^{\frac{\gamma-1}{2} }|^2 dxdt \\ &+ \dfrac{1}{2} \int_{0}^{T\land \tau} \int_{\mathbb{T}} (\partial_{xx} \log \rho )^2 dxdt = \frac{1}{2} \sum_{k=1}^{\infty}\int_{0}^{T\land \tau}\int_{\mathbb{T}} \rho |F_k|^2dxdt \\ & + \int_{0}^{T\land \tau}\int_{\mathbb{T}} \rho V \mathbb{F}(\rho,u) dxdW,
\end{split}
\end{equation}
\item[2.]
for $\alpha \neq 0$
\begin{equation}\label{entropy stoch}
\begin{split}
& \sup_{t \in [0,T\land \tau]} \int_{\mathbb{T}} \dfrac{1}{2}\rho |V|^2+ \dfrac{\rho^\gamma}{\gamma-1}+ | \partial_x \sqrt{\rho}|^2 dx + \dfrac{4 \gamma}{(\gamma+\alpha -1)^2} \int_{0}^{T\land \tau}\int_{\mathbb{T}} | \partial_x \rho^{\frac{\gamma+\alpha-1}{2} } |^2 dxdt \\ & + \dfrac{4}{\alpha^2} \int_{0}^{T\land \tau} \int_{\mathbb{T}} | \partial_{xx} 	\rho^{\frac{\alpha}{2}} |^2 dxdt +\dfrac{4(4-3 \alpha) }{3 \alpha^3}\int_{0}^{T\land \tau} \int_{\mathbb{T}} \rho^{-\alpha} | \partial_x \rho^\frac{\alpha}{2}|^4 dxdt \\ & = \frac{1}{2} \sum_{k=1}^{\infty}\int_{0}^{T\land \tau}\int_{\mathbb{T}} \rho |F_k|^2dxdt+ \int_{0}^{T\land \tau}\int_{\mathbb{T}} \rho V \mathbb{F}(\rho,u) dxdW.
\end{split}
\end{equation}
\end{itemize}
Finally we estimate the first integral in the right hand side of \eqref{entropy stoch} in the following way
\begin{equation}
\begin{split}
\sum_{k \in \mathbb{N}} \int_{0}^{T\land \tau} \int_{\mathbb{T}} \rho |F_k|^2dxdt \le \int_{0}^{T\land \tau} \sum_{k \in \mathbb{N}} \alpha^2_k \int_{\mathbb{T}} ( \rho + \rho |u|^2) dxdt \lesssim \int_{0}^{T\land \tau} \int_{\mathbb{T}} ( \rho+ \rho |u|^2) dxdt,
\end{split}
\end{equation}
while the stochastic integral is estimated by using the Burkholder-Davis-Gundy inequality as follows
\begin{equation}\label{stoch int bd entr}
\begin{split}
& \mathbb{E} \bigg| \int_{0}^{t} \int_{} \rho V \mathbb{F}(\rho,u) dxdW \bigg|^p \le  \mathbb{E} \bigg[ \sup_{t \in [0.T\land \tau]} \bigg| \int_{0}^{t} \int_{} \rho V \mathbb{F}(\rho,u) dxdW \bigg|^p \bigg] \\ & \lesssim \mathbb{E} \bigg[ \int_{0}^{T\land \tau} \sum_{k \in \mathbb{N}} \bigg| \int_{\mathbb{T}} \rho V F_k dx \bigg|^2 dt \bigg]^{\frac{p}{2}} \le \mathbb{E}\bigg[ \int_{0}^{T\land \tau}  \bigg| \int_{\mathbb{T}} ( \rho+ \rho |V|^2) dx \bigg|^2 dt \bigg]^{\frac{p}{2}}.
\end{split}
\end{equation}
By using Young inequality in \eqref{stoch int bd entr} and applying the Gronwall Lemma we end up with 
\begin{itemize}
\item[1.]
for $\alpha=0$
\begin{equation}
\begin{split}
& \mathbb{E} \bigg| \sup_{t \in [0,T\land \tau]} \int_{} \bigg[\dfrac{1}{2}\rho |V|^2+ \dfrac{\rho^\gamma}{\gamma-1}+ | \partial_x \sqrt{\rho}|^2 \bigg] dx \bigg|^p \\ & +\mathbb{E} \bigg| \dfrac{4 \gamma \mu}{(\gamma-1)^2} \int_{0}^{T\land \tau}\int_{\mathbb{T}} | \partial_x \rho^{\frac{\gamma-1}{2} }|^2 dxdt + \dfrac{1}{2} \int_{0}^{T\land \tau} \int_{\mathbb{T}} (\partial_{xx} \log \rho )^2 dxdt \bigg|^p \\ & \lesssim 1+ \mathbb{E}\bigg| \sup_{t \in [0,T\land \tau]} \int_{} \bigg[\dfrac{1}{2}\rho_0 |V_0|^2+ \dfrac{\rho_0^\gamma}{\gamma-1}+ | \partial_x \sqrt{\rho_0}|^2 \bigg] dx \bigg|^p,
\end{split}
\end{equation}
\item[2.]
for $\alpha \neq 0$
\begin{equation}
\begin{split}
& \mathbb{E} \bigg| \sup_{t \in [0,T\land \tau]} \int_{} \bigg[\dfrac{1}{2}\rho |V|^2+ \dfrac{\rho^\gamma}{\gamma-1}+ | \partial_x \sqrt{\rho}|^2 \bigg] dx \bigg|^p \\ & +\mathbb{E} \bigg| \dfrac{4 \gamma}{(\gamma+\alpha -1)^2} \int_{0}^{T\land \tau}\int_{\mathbb{T}} | \partial_x \rho^{\frac{\gamma+\alpha-1}{2} } |^2 dxdt \\ & + \dfrac{4}{\alpha^2} \int_{0}^{T\land \tau} \int_{\mathbb{T}} | \partial_{xx} 	\rho^{\frac{\alpha}{2}} |^2 dxdt +\dfrac{4(4-3 \alpha) }{3 \alpha^3}\int_{0}^{T\land \tau} \int_{\mathbb{T}} \rho^{-\alpha} | \partial_x \rho^\frac{\alpha}{2}|^4 dxdt  \bigg|^p \\ & \lesssim 1+ \mathbb{E}\bigg| \sup_{t \in [0,T\land \tau]} \int_{} \bigg[\dfrac{1}{2}\rho_0 |V_0|^2+ \dfrac{\rho_0^\gamma}{\gamma-1}+ | \partial_x \sqrt{\rho_0}|^2 \bigg] dx \bigg|^p,
\end{split}
\end{equation}
\end{itemize}
which concludes our proof.
\end{proof}
\noindent
Now we prove that \eqref{BD quantum} is non negative for an appropriate range of $\alpha.$  \\
Applying Lemma \ref{Functional ineq } to \eqref{BD quantum} with $f=\rho^{\frac{\alpha}{2}} $ we get 
\begin{equation}
\begin{split}
& \dfrac{4(4-3 \alpha) }{3 \alpha^3}\int_{} \rho^{-\alpha} | \partial_x \rho^\frac{\alpha}{2}|^4+\dfrac{4}{\alpha^2} \int_{} | \partial_{xx} 	\rho^{\frac{\alpha}{2}} |^2 \\ & \ge \dfrac{4(4-3 \alpha) }{3 \alpha^3}\int_{} \rho^{-\alpha} | \partial_x \rho^\frac{\alpha}{2}|^4+\dfrac{4}{9\alpha^2} \int_{} \rho^{-\alpha} | \partial_x \rho^{\frac{\alpha}{2} }|^4 \\ & =\dfrac{16(3-2\alpha)}{9\alpha^3} \int_{} \rho^{-\alpha} | \partial_x \rho^{\frac{\alpha}{2} }|^4 \ge 0 \iff \alpha \in (0, \frac{3}{2}].
\end{split}
\end{equation}
We summarize the new regularity estimates obtained by virtue of Proposition \ref{BD prop} and Lemma \ref{Functional ineq }
\begin{itemize}
\item[1.] if $\alpha=0$
\begin{equation}\label{reg entro 0}
\begin{split}
& \partial_x \bigg( \dfrac{1}{\sqrt{\rho}}\bigg)\in L^p(\Omega;L^{\infty}(0,T\land \tau;L^2(\mathbb{T}))), \quad \partial_x (\rho^{\frac{\gamma-1}{2} }) \in L^p(\Omega;L^{2}(0,T\land \tau;L^2(\mathbb{T}))), \\ &  \partial_{xx} \log \rho \in L^p(\Omega;L^{2}(0,T\land \tau;L^2(\mathbb{T}))).
\end{split}
\end{equation}
\item[2.] if $\alpha \in (0,\frac{1}{2})$
\begin{equation}\label{reg entro alpha}
\begin{split}
& \partial_x (\rho^{\alpha-\frac{1}{2}} )\in L^p(\Omega;L^{\infty}(0,T\land \tau;L^2(\mathbb{T}))), \quad \partial_x (\rho^{\frac{\gamma+\alpha-1}{2} }) \in L^p(\Omega;L^{2}(0,T\land \tau;L^2(\mathbb{T}))), \\ &  \partial_{xx} \rho^{\frac{\alpha}{2}} \in L^p(\Omega;L^{2}(0,T\land \tau;L^2(\mathbb{T}))),  \quad  \rho^{-\alpha} | \partial_x \rho^\frac{\alpha}{2}|^4 \in L^p(\Omega;L^{1}(0,T\land \tau;L^1(\mathbb{T}))).
\end{split}
\end{equation}
\item[3.] if $\alpha=\frac{1}{2}$
\begin{equation}\label{reg entro 1/2}
\begin{split}
& \partial_x \log \rho \in L^p(\Omega;L^{\infty}(0,T\land \tau;L^2(\mathbb{T}))), \quad \partial_x (\rho^{\frac{2\gamma-1}{4} }) \in L^p(\Omega;L^{2}(0,T\land \tau;L^2(\mathbb{T}))), \\ &  \partial_{xx} \rho^{\frac{1}{4}} \in L^p(\Omega;L^{2}(0,T\land \tau;L^2(\mathbb{T}))),  \quad  \rho^{-\frac{1}{2}} | \partial_x \rho^\frac{1}{4}|^4 \in L^p(\Omega;L^{1}(0,T\land \tau;L^1(\mathbb{T}))).
\end{split}
\end{equation}
\end{itemize}
\begin{proposition}\label{prop:vacuum}
Let $(\rho,u)$ be a strong solution of \eqref{stoc quantum}-\eqref{C.I}, then
$$ \sqrt{\rho} \in L^p(\Omega;L^{\infty}(0,T\land \tau;L^\infty(\mathbb{T}))).$$
Moreover if $\alpha \in [0,\frac{1}{2}],$ then $$ \dfrac{1}{\rho^\beta} \in L^p(\Omega;L^{\infty}(0,T\land \tau;L^\infty(\mathbb{T}))),$$
for all $\beta > 0,  \, p\in [1, \infty).$
\end{proposition}
\begin{proof}
By mean value theorem,  for any $t \in (0,T\land \tau)$ and $\omega \in \Omega,$ there exists a point $\bar x$ such that $$\rho(\bar x,t,  \omega)= \int_{\mathbb{T}} \rho(y,t,\omega) dy$$
Therefore,  by the conservation of the mass and the hypothesis \eqref{C.I STRONG},  there exists a deterministic constant $C>0 $ such that 
\begin{equation}\label{eq:a}
\frac{1}{C} \le \rho(\bar x,t,\omega) \le C, \quad \mathbb{P}\text{-a.s.}
\end{equation}
Next,  by using the fundamental theorem of calculus, for any fixed $t \in (0,T\land \tau), \; \omega \in \Omega$ and $x \in \mathbb{T},$ the following identities hold
\begin{equation}\label{FTC 0}
 \sqrt{\rho(x,t,\omega)}-\sqrt{\rho(\bar x,t,\omega)}= \int_{\bar x}^{ x} \partial_x \sqrt{\rho(y,t,\omega)}dy,
 \end{equation}
\begin{equation}\label{FTC alpha}
\frac{1}{{\rho}^{\frac{1}{2}-\alpha}(x,t,\omega)}-\frac{1}{{\rho}^{\frac{1}{2}-\alpha}(\bar x,t,\omega)}= \int_{\bar x}^{x} \partial_x \bigg( \frac{1}{{\rho}^{\frac{1}{2}-\alpha}(y,t,\omega)} \bigg) dy
\end{equation}
and similarly 
\begin{equation}\label{FTC 1/2}
\log \rho(x,t,\omega)-\log \rho (\bar x,t,\omega)= \int_{\bar x}^{x} \partial_x \log \rho(y,t,\omega)dy.
\end{equation}
Hence by taking the the $\sup$ in both space and time, the $p-$th power in \eqref{FTC 0}-\eqref{FTC 1/2},  and then the expectation, by \eqref{reg entro alpha} and \eqref{eq:a}, we get
\begin{equation}
\begin{split}
& \sqrt{\rho} \in L^p(\Omega;L^{\infty}(0,T\land \tau;L^\infty(\mathbb{T}))),  \quad \dfrac{1}{\rho^{\frac{1}{2}-\alpha}} \in L^p(\Omega;L^{\infty}(0,T\land \tau;L^\infty(\mathbb{T}))),  \\ &  \log \rho \in L^p(\Omega;L^{\infty}(0,T\land \tau;L^\infty(\mathbb{T}))).
\end{split}
\end{equation}
In order to conclude the proof of our second statement we observe that in the range $\alpha \in [0,\frac{1}{2}]$ we control a negative exponent of the density and by using H\"{o}lder inequality combined with the first statement we get our claim.
\end{proof}

A fundamental consequence of Proposition \ref{prop:vacuum} is the following corollary. 
\begin{corollary}\label{cor:szero}
Let $(\rho,u)$ be a strong solution of \eqref{stoc quantum}-\eqref{C.I}, let $\alpha \in [0, \frac{1}{2}],$ then the following holds
\begin{equation}
\begin{split}
& \partial_x \log\rho \in L^p(\Omega; L^\infty (0,T\land \tau;L^2(\mathbb{T}))),  \quad  u \in L^p(\Omega; L^\infty (0,T\land \tau;L^2(\mathbb{T}))), \\ & \partial_x u \in L^p(\Omega; L^2(0,T\land \tau;L^2(\mathbb{T}))).
\end{split}
\end{equation}
\end{corollary}
\begin{proof}
The proof follows by using the regularity inferred by the energy estimate and Bresch-Desjardins entropy inequality combined with Proposition \ref{prop:vacuum}. 
\end{proof}

For the next proposition it is convenient to formulate the system \eqref{stoc quantum} in the variables $(\psi,u)$ where $\psi= \log \rho$. This change of variable allows to get an important cancellation of the high order derivatives terms. We start by recalling also the following elementary identity concerning the quantum term
\begin{equation}\label{eq:idquantum}
2\rho \partial_x \bigg( \dfrac{\partial_{xx} \sqrt{\rho}}{\rho} \bigg)= \partial_x (\rho \partial_{xx} \log \rho ). 
\end{equation}

\noindent
Dividing by $\rho$ both the continuity and the momentum equations in \eqref{stoc quantum} we obtain
\begin{equation*}
\frac{\text{d}\rho}{\rho}+[u \frac{\partial_x \rho}{\rho}+ \partial_x u]\text{d}t=0 
\end{equation*}
\begin{equation*}
\text{d} u+[u \partial_x u + \frac{\partial_xp(\rho)}{\rho}]\text{d}t=[\frac{\partial_x(\mu(\rho)\partial_x u)}{\rho}+ \partial_x \bigg( \dfrac{\partial_{xx} \sqrt{\rho}}{\sqrt{\rho}}\bigg) ]\text{d}t +\mathbb{F}(\rho,\rho u)\text{d}W.
\end{equation*}
Then, in the new variables $(\psi,u)$ we obtain the system
\begin{equation}\label{psi nl}
\text{d}\psi+[u \partial_x \psi+ \partial_x u]\text{d}t=0 
\end{equation}
\begin{equation}\label{u nl}
\begin{split}
& \text{d}u+[u\partial_x u + \gamma e^{(\gamma-1)\psi}\partial_x \psi]dt=[e^{\alpha-1)\psi}\partial_{xx}u+\alpha e^{(\alpha-1)\psi}\partial_x \psi \partial_x u]dt \\ & + [\partial_{xxx}\psi+ \partial_x\psi \partial_{xx} \psi]dt+ \mathbb{F}(\rho, u)\text{d}W.
\end{split}
\end{equation}
\noindent
The following proposition contains the high order estimates needed to prove Theorem \ref{Main Theorem global}.  We refer the reader to \cite{Deb},  Theorem 3.2, for a similar construction to prove global well posedness for incompressible flows.
\begin{proposition}\label{Prop global s+1 s}
Let $(\rho,u)$ be a strong solution of \eqref{stoc quantum}-\eqref{C.I},  let $\alpha \in [0,\frac{1}{2}].$ Then the following holds
\begin{equation}\label{p=1 global hs}
\begin{split}
& \mathbb{E} \sup_{ s \in [0, T \land \gamma_M \land \tau] } \dfrac{1}{2}\bigg( \| \psi \|^2_{H^4}+ \| u \|^2_{H^3} \bigg)  \\ & + \mathbb{E} \int_{0}^{T \land \gamma_M \land \tau} \int_{\mathbb{T}} e^{(\alpha-1)\psi} ( |u|^2+|\partial_x u |^2|+\partial_{xx} u |^2+|\partial^3_x u |^2+|\partial^4_x u |^2) ds  < \infty
\end{split}
\end{equation}
where $\gamma_M$ is an appropriate stopping time defined in \eqref{min gamma M} satisfying $\lim_{ \rightarrow \infty} \gamma_M = \infty.$
\end{proposition}
\begin{proof}
We start by proving that 
\begin{equation}\label{eq:s1}
\begin{split}
& \partial_{xx} \psi \in L^1(\Omega;L^\infty(0,T \land \gamma_M \land \tau;L^2(\mathbb{T}))),\quad \partial_x u \in L^1(\Omega;L^\infty(0,T \land \gamma_M \land \tau;L^2(\mathbb{T}))), \\ & 
\rho^{\frac{\alpha-1}{2}} \partial_{xx} u \in L^1(\Omega;L^2(0,T \land \gamma_M \land \tau;L^2(\mathbb{T}))).
\end{split}
\end{equation}
We differentiate \eqref{psi nl} two times in the space variable,  we multiply it by $\partial_{xx} \psi$ and integrating in space we get 
\begin{equation}\label{ito rho s=1}
\text{d} \int_{} \frac{ | \partial_{xx} \psi|^2}{2} dx+ \int_{} \partial_{xx}(u \partial_x \psi)\partial_{xx} \psi dx\text{d}t+ \int_{} \partial^3_{x} u \partial_{xx} \psi dx \text{d}t=0, 
\end{equation}
then we apply $\partial_x $ in the momentum equation \eqref{u nl}.  We use It$\hat{\text{o}}$ formula to the functional $F(u)=\frac{1}{2}\int_{} |\partial_x u|^2$ in order to get 
\begin{equation}\label{ito u s=1}
\begin{split}
& \text{d}\int_{}\frac{| \partial_xu|^2}{2}dx+\int_{}\partial_x(u\partial_x u)\partial_x u dx\text{d}t+ \int_{} \partial_x(\gamma e^{(\gamma-1)\psi }\partial_x \psi) \partial_xudxdt= \\ &\int_{} \partial_x (e^{(\alpha-1)\psi}\partial_{xx}u) \partial_x u dxdt + \int_{} \partial_x(\alpha e^{(\alpha-1)\psi}\partial_x \psi \partial_x u)\partial_x udxdt + \int_{}\partial^4_{x}\psi \partial_x udxdt \\  &+\int_{} \partial_x(\partial_x\psi \partial_{xx} \psi) \partial_x udt + \int_{}\partial_x \mathbb{F}(\rho, u)\partial_x u\text{d}W+ \frac{1}{2} \sum_{k=1}^{\infty} \int_{} |\partial_x F_k|^2 dx
\end{split}
\end{equation}
and summing up \eqref{ito rho s=1}-\eqref{ito u s=1} we end up with the following equality 
\begin{equation}
\begin{split}
& \text{d}\bigg( \int_{}\frac{ | \partial_{xx}\psi|^2}{2}+ \frac{| \partial_xu|^2}{2}dx \bigg)+ \int_{} \partial_{xx}(u \partial_x \psi)\partial_{xx} \psi dx\text{d}t+ \int_{} \partial^3_{x} u \partial_{xx} \psi dx \text{d}t \\ & +\int_{}\partial_x(u\partial_x u)\partial_x u dx\text{d}t+ \int_{} \partial_x(\gamma e^{(\gamma-1)\psi }\partial_x \psi) \partial_xudxdt=\int_{} \partial_x (e^{(\alpha-1)\psi}\partial_{xx}u) \partial_x u dxdt \\ & + \int_{} \partial_x(\alpha e^{(\alpha-1)\psi}\partial_x \psi \partial_x u)\partial_x udxdt + \int_{}\partial^4_{x}\psi \partial_x udxdt+\int_{} \partial_x(\partial_x\psi \partial_{xx} \psi) \partial_x udt \\ & + \int_{}\partial_x \mathbb{F}(\rho, u)\partial_x udx\text{d}W+ \frac{1}{2} \sum_{k=1}^{\infty} \int_{} |\partial_x F_k|^2 dxdt,
\end{split}
\end{equation}
which can be rewritten as
\begin{equation}\label{s=1}
\begin{split}
& \text{d}\bigg( \int_{}\frac{ | \partial_{xx} \psi|^2}{2}+ \frac{| \partial_x u|^2}{2}dx \bigg)+ \int_{} e^{(\alpha-1) \psi} | \partial_{xx} u |^2 dxdt=- \int_{} \partial_{xx}(u \partial_x \psi)\partial_{xx} \psi dx\text{d}t \\ &-\int_{}\partial_x(u\partial_x u)\partial_x u dx\text{d}t- \int_{} \partial_x(\gamma e^{(\gamma-1)\psi }\partial_x \psi) \partial_xudxdt+\int_{} \partial_x(\alpha e^{(\alpha-1)\psi}\partial_x \psi \partial_x u)\partial_x udxdt \\ & + \int_{} \partial_x(\partial_x\psi \partial_{xx} \psi) \partial_x udxdt + \int_{}\partial_x \mathbb{F}(\rho, u)\partial_x udx\text{d}W+ \frac{1}{2} \sum_{k=1}^{\infty} \int_{} |\partial_x F_k|^2 dxdt= \sum_{i=1}^{7} I_i.
\end{split}
\end{equation}
In order to estimate the integrals in the right hand side of \eqref{s=1} we make and extensive use of H\"{o}lder,  Sobolev and Young inequalities
\begin{equation*}
\begin{split}
|I_1|&=| -\int{} \partial_{xx}(u\partial_x \psi)\partial_{xx}\psi dx|  \le \int_{} | \partial_{xx}u| | \partial_x \psi| |  \partial_{xx}\psi| +\frac{3}{2}| \partial_x u|  |\partial_{xx} \psi |^2 dx \\ & \le  \| \rho^{\frac{\alpha-1}{2}}\partial_{xx}u \|_{L^2} \| \partial_x \psi \|_{L^\infty} \| \partial_{xx} \psi \|_{L^2} \| \frac{1}{\rho^{\frac{\alpha-1}{2}}} \|_{L^\infty} + c \| \partial_x u \|_{L^\infty} \| \partial_{xx} \psi \|^2_{L^2}.
\end{split}
\end{equation*}
\begin{equation}\label{I_1 s=1}
\begin{split}
|I_1|& \le \delta \| \rho^{\frac{\alpha-1}{2}}\partial_{xx} u \|^2_{L^2} + C(\delta) \| \partial_{xx} \psi \|^2_{L^2} \| \partial_{xx} \psi \|^2_{L^2}\| \frac{1}{\rho^{\frac{\alpha-1}{2}}} \|^2_{L^\infty} \\ & + C(\delta) \| \frac{1}{\rho^{\frac{\alpha-1}{2}}} \|^2_{L^\infty} \| \partial_{xx} \psi \|^2_{L^2},
\end{split}
\end{equation}
\begin{equation}
\begin{split}
|I_2| & = | -\int_{} \partial_x (u \partial_x u ) \partial_x u dx | \le  \int_{} |u| | \partial_x u| | \partial_{xx} u| dx \\ & \le \| u \|_{L^\infty} \| \partial_x u \|_{L^2} \| \rho^{\frac{\alpha-1}{2}}\partial_{xx} u \|_{L^2} \| \frac{1}{\rho^{\frac{\alpha-1}{2}}}\|_{L^\infty}  \\ & \le C(\delta) \| u \|^2_{L^\infty} \| \partial_x u \|^2_{L^2}\| \frac{1}{\rho^{\frac{\alpha-1}{2}}}\|^2_{L^\infty}+ \delta \| \rho^{\frac{\alpha-1}{2}}\partial_{xx} u \|^2_{L^2} \\ & \le C(\delta) \| \partial_x u \|^2_{L^2} \| \frac{1}{\rho^{\frac{\alpha-1}{2}}} \|^2_{L^\infty}( \| u \|^2_{L^2}+ \| \partial_x u \|^2_{L^2})+ \delta \| \rho^{\frac{\alpha-1}{2}}\partial_{xx} u \|^2_{L^2},
\end{split}
\end{equation}
\begin{equation}
\begin{split}
|I_3| & = | -\int_{} \partial_x (\gamma e^{(\gamma-1) \psi } \partial_x \psi ) \partial_x u dx | \le  \int_{}\gamma e^{(\gamma-1) \psi } | \partial_x \psi| | \partial_{xx} u| dx  \\ & \le c \| \rho^{\frac{2\gamma-\alpha-1}{2}} \|_{L^\infty} \| \partial_x \psi \|_{L^2} \| \rho^{\frac{\alpha-1}{2}}\partial_{xx} u \|_{L^2} \\ & \le C(\delta) \| \rho^{\frac{2\gamma-\alpha-1}{2}} \|^2_{L^\infty}  \| \partial_x \psi \|^2_{L^2}+\delta \| \rho^{\frac{\alpha-1}{2}}\partial_{xx} u \|^2_{L^2},
\end{split}
\end{equation}
\begin{equation}
\begin{split}
| I_4| & =| \int_{} \partial_x(\alpha e^{(\alpha-1)\psi}\partial_x \psi \partial_x u)\partial_x udx | \le  \int_{} \frac{| \partial_x \rho^{\alpha}|}{\rho} | \partial_x u| | \partial_{xx} u| dx  \\ & \le C \| \frac{1}{\rho^{\frac{\alpha+1}{2}}} \|_{L^\infty} \| \partial_x \rho^{\alpha} \|_{L^\infty} \| \partial_x u \|_{L^2} \| \rho^{\frac{\alpha-1}{2}}\partial_{xx} u \|_{L^2} \\ & \le C \| \frac{1}{\rho^{\frac{\alpha+1}{2}}} \|_{L^\infty} \| \partial_{xx} \rho^{\alpha} \|_{L^2} \| \partial_x u \|_{L^2} \| \rho^{\frac{\alpha-1}{2}}\partial_{xx} u \|_{L^2}\\ & \le C(\delta) \| \frac{1}{\rho^{\frac{\alpha+1}{2}}} \|^2_{L^\infty} \| \partial_{xx} \rho^{\alpha} \|^2_{L^2} \| \partial_x u \|^2_{L^2}  +\delta \| \rho^{\frac{\alpha-1}{2}}\partial_{xx} u \|^2_{L^2},
\end{split}
\end{equation}
\begin{equation}
\begin{split}
| I_5 | & = | \int_{} \partial_x ( \partial_x \psi \partial_{xx} \psi) \partial_x u dx | \le \int_{} | \partial_x \psi | | \partial_{xx} \psi | | \partial_{xx} u| dx  \\ & \le \| \partial_x \psi \|_{L^\infty} \| \partial_{xx} \psi \|_{L^2} \| \rho^{\frac{\alpha-1}{2}}\partial_{xx} u \|_{L^2} \| \frac{1}{\rho^{\frac{\alpha-1}{2}}} \|_{L^\infty} \\ & \le C \| \partial_{xx} \psi \|_{L^2}\| \partial_{xx}\psi \|_{L^2} \| \rho^{\frac{\alpha-1}{2}}\partial_{xx} u \|_{L^2} \| \frac{1}{\rho^{\frac{\alpha-1}{2}}} \|_{L^\infty} \\ & \le \delta \| \rho^{\frac{\alpha-1}{2}}\partial_{xx} u \|^2_{L^2} +C(\delta) \| \partial_{xx} \psi \|^2_{L^2} \| \frac{1}{\rho^{\frac{\alpha-1}{2}}} \|^2_{L^\infty} \| \partial_{xx} \psi \|^2_{L^2}.
\end{split}
\end{equation}
Finally we can estimate the It$\hat{\text{o}}$ reminder term $I_7$ as follows
\begin{equation}\label{I_7 s=1}
\begin{split}
|I_7| &=\frac{1}{2} \sum_{k=1}^{\infty} \int_{\mathbb{T}} |\partial_x F_k|^2 dx=\frac{1}{2} \sum_{k=1}^{\infty} \| \partial_x F_k \|^2_{L^2} \le c \sum_{k=1}^{\infty} \| \partial_x F_k \|^2_{L^\infty} \le c \sum_{k=1}^{\infty} {\alpha_k}^2 \le C.
\end{split}
\end{equation}
Hence we sum up the estimates \eqref{I_1 s=1}-\eqref{I_7 s=1} and by choosing properly the constant $\delta$ we end up with the following inequality
\begin{equation}\label{ineq with a(t)}
\begin{split}
& \dfrac{1}{2} \text{d} ( \| \partial_{xx} \psi \|^2_{L^2}+ \| \partial_x u \|^2_{L^2}) + c\int_{\mathbb{T}} (e^{(\alpha-1) \psi }) | \partial_{xx} u|^2 dx \\ & \le  \bigg(\| \partial_{xx} \psi \|^2_{L^2}+ \| \partial_x u \|^2_{L^2}\bigg)\bigg[1+ \| \partial_{xx}\psi \|^2_{L^2} \| \frac{1}{\rho^{\frac{\alpha-1}{2}}}\|^2_{L^\infty} + \| \frac{1}{\rho^{\frac{\alpha-1}{2}}}\|^2_{L^\infty} \\ & + \| u \|^2_{L^2} \| \frac{1}{\rho^{\frac{\alpha-1}{2}}}\|^2_{L^\infty}+ \| \partial_x u \|^2_{L^2}\| \frac{1}{\rho^{\frac{\alpha-1}{2}}}\|^2_{L^\infty}+ \| \partial_{xx} \rho^\alpha \|^2_{L^2}\| \frac{1}{\rho^{\frac{\alpha+1}{2}}} \|^2_{L^\infty} \bigg] \\ & + \| \rho^{\frac{2\gamma-\alpha-1}{2} \|^2_{L^\infty}} \| \partial_x \psi \|^2_{L^2} +I_6 = \bigg(\| \partial_{xx} \psi \|^2_{L^2}+ \| \partial_x u \|^2_{L^2}\bigg)a^{(1)}(t)+b^{(1)}(t)+I_6.
\end{split}
\end{equation}
We define the following stopping time 
\begin{equation}
\gamma^{(1)}_M= \inf \big\{ t\ge 0 \; :  \int_{0}^{t \land \tau} a^{(1)}(s)ds > M \big\},
\end{equation}
with the convention $\inf \emptyset= +\infty.$ \\
Clearly by definition of $\gamma^{(1)}_M$ we have 
\begin{equation}
\int_{0}^{\gamma^{(1)}_M} a^{(1)}(t)dt \le M \quad a.s.
\end{equation}
We perform a localization argument.  Let $\tau_a, \; \tau_b$ be a pair of stopping times s.t $0 \le \tau_a < \tau_b \le \gamma^{(1)}_M \land \tau \land T,$ we integrate in time and we take the $\sup$ over the interval $[\tau_a, \tau_b].$ After applying expectation we get
\begin{equation}
\begin{split}
& \mathbb{E} \big(\sup_{ \tau_a \le s \le \tau_b}\dfrac{1}{2}( \| \partial_{xx} \psi \|^2_{L^2}+ \| \partial_x u \|^2_{L^2}) + \int_{\tau_a}^{\tau_b}\int_{\mathbb{T}} (e^{(\alpha-1) \psi }) | \partial_{xx} u|^2 dx \\ & \le \mathbb{E} \big(\| \partial_{xx} \psi (\tau_a)\|^2_{L^2}+ \| \partial_x u (\tau_a)\|^2_{L^2}\big)+ c \mathbb{E} \int_{\tau_a}^{\tau_b} \big(\| \partial_{xx} \psi \|^2_{L^2}+ \| \partial_x u \|^2_{L^2}\big)a^{(1)}(t)dt \\ & +c\mathbb{E} \int_{\tau_a}^{\tau_b}b^{(1)}(t)dt+ c\mathbb{E} \sup_{\tau_a \le s \le \tau_b} | \int_{\tau_a}^{s}I_6 |.
\end{split}
\end{equation}
Concerning the stochastic integral $I_6,$ we observe that by using the Burkholder-Davis-Gundy inequality we get
\begin{equation*}
\begin{split}
\mathbb{E} \bigg[ \sup_{ \tau_a \le s \le \tau_b} \bigg| {\int_{\tau_a}^{s} \int_{\mathbb{T}} \partial_x \mathbb{F}(\rho,u)\partial_x udxdW \bigg| } \bigg] \lesssim \mathbb{E} \bigg[ \int_{\tau_a}^{\tau_b} \sum_{k=1}^{\infty} \bigg| \int_{\mathbb{T}} \partial_x F_k \partial_x u dx \bigg|^2 dt \bigg]^{\frac{1}{2}}
\end{split}
\end{equation*}
and 
\begin{equation*}
\begin{split}
 \bigg| \int_{\mathbb{T}} \partial_x F_k \partial_x u dx \bigg|^2 \le \| \partial_x F_k \|^2_{L^2} \| \partial_x u \|^2_{L^2} \le C \| \partial_x F_k \|^2_{L^\infty} \| \partial_x u \|^2_{L^2} \le C {\alpha_k}^2 \| \partial_x u \|^2_{L^2},
\end{split}
\end{equation*}
hence we have 
\begin{equation*}
\begin{split}
& \mathbb{E} \bigg[ \int_{\tau_a}^{\tau_b} \sum_{k=1}^{\infty} \bigg| \int_{\mathbb{T}} \partial_x F_k \partial_x u dx \bigg|^2 dt \bigg]^{\frac{1}{2}} \le  \mathbb{E} \bigg[ \int_{\tau_a}^{\tau_b} \sum_{k=1}^{\infty}C {\alpha_k}^2 \| \partial_x u \|^2_{L^2} \bigg]^{\frac{1}{2}} \\ & \le \mathbb{E} \bigg[ \int_{\tau_a}^{\tau_b} \tilde{c} \| \partial_x u \|^2_{L^2} \bigg]^{\frac{1}{2}} \le \mathbb{E} \bigg( 1+ \int_{\tau_a}^{\tau_b} c \| \partial_x u \|^2_{L^2} \bigg)
\end{split}
\end{equation*}
and we end up with the following inequality 
\begin{equation}
\begin{split}
& \mathbb{E} \big(\sup_{ \tau_a \le s \le \tau_b}\dfrac{1}{2}( \| \partial_{xx} \psi \|^2_{L^2}+ \| \partial_x u \|^2_{L^2}) + \int_{\tau_a}^{\tau_b}\int_{\mathbb{T}} (e^{(\alpha-1) \psi }) | \partial_{xx} u|^2 dx \\ & \le \mathbb{E} \big(\| \partial_{xx} \psi (\tau_a)\|^2_{L^2}+ \| \partial_x u (\tau_a)\|^2_{L^2}\big)+ c \mathbb{E} \int_{\tau_a}^{\tau_b} \big(\| \partial_{xx} \psi \|^2_{L^2}+ \| \partial_x u \|^2_{L^2}\big)a^{(1)}(t)dt \\ & +c\mathbb{E} \int_{\tau_a}^{\tau_b}b^{(1)}(t)dt.
\end{split}
\end{equation}
Now,  applying the stochastic version of the Gronwall Lemma,  see Proposition  \ref{stoch gronw}, we infer the following result
\begin{equation}
\mathbb{E} \bigg(\sup_{  s \in [0, T \land \gamma^{(1)}_M \land \tau] }\dfrac{1}{2}( \| \partial_{xx} \psi \|^2_{L^2}+ \| \partial_x u \|^2_{L^2}) + \int_{0}^{T \land \gamma^{(1)}_M \land \tau}\int_{\mathbb{T}} e^{(\alpha-1) \psi } | \partial_{xx} u|^2 dx \bigg) < +\infty.
\end{equation}
Now we want to prove that $\lim_{M \rightarrow \infty} \gamma^{(1)}_M= \infty$ a.s.  so that such an estimate is indeed global.  To this purpose we observe that by using the Markov inequality we have
\begin{equation}\label{markov M s=1}
 \mathbb{P} \big( \gamma^{(1)}_M \le T \big) \le \mathbb{P} \big( \int_{0}^{T \land \tau} a^{(1)}(t)dt \ge M \big) \le \frac{1}{M} \mathbb{E} \big( \int_{0}^{T \land \tau} a^{(1)}(t)dt \big).
\end{equation}
We focus our attention on the estimate of the term $ \| \partial_{xx} \psi \|^2_{L^2} \| \frac{1}{\rho^{\frac{\alpha-1}{2}}} \|^2_{L^\infty}$ which appears in the right hand side of \eqref{markov M s=1}. This term can be handled by using the regularity of the $p-$th momenta in the energy and BD entropy estimates.  Indeed we have
\begin{equation}\label{interp a(t) s=1}
\begin{split}
& \mathbb{E} \bigg( \int_{0}^{T \land \tau} \| \partial_{xx} \psi \|^2_{L^2} \| \frac{1}{\rho^{\frac{\alpha-1}{2}}} \|^2_{L^\infty}dt \bigg) \le \mathbb{E} \sup_{[0,T \land \tau]}\| \frac{1}{\rho^{\frac{\alpha-1}{2}}} \|^2_{L^\infty} \bigg( \int_{0}^{T \land \tau} \| \partial_{xx} \psi \|^2_{L^2} dt \bigg) \\ & \le \bigg( \mathbb{E} \bigg( \sup_{[0,T \land \tau]}\| \frac{1}{\rho^{\frac{\alpha-1}{2}}} \|^2_{L^\infty} \bigg)^2 \bigg)^{\frac{1}{2}} \bigg( \mathbb{E} \bigg( \int_{0}^{T \land \tau} \| \partial_{xx} \psi \|^2_{L^2} dt \bigg)^2 \bigg)^{\frac{1}{2}} < \infty
\end{split}
\end{equation}
hence sending $M \rightarrow \infty$ in \eqref{markov M s=1} we get our claim.
The estimates of the other terms in the definition of $a^{(1)}(t)$ follow the same lines of argument of the previous case.
\\
\\
Next, we prove that 
\begin{equation}\label{eq:s2}
\begin{split}
& \partial^3_{x} \psi \in L^1(\Omega;L^\infty(0,T \land \gamma_M \land \tau;L^2(\mathbb{T}))),\quad \partial_{xx} u \in L^1(\Omega;L^\infty(0,T \land \gamma_M \land \tau;L^2(\mathbb{T}))), \\ & 
\rho^{\frac{\alpha-1}{2}}\partial^3_{x} u \in L^1(\Omega;L^2(0,T \land \gamma_M \land \tau;L^2(\mathbb{T}))).
\end{split}
\end{equation}
We apply $\partial^3_{x}$ to \eqref{psi nl}, we multiply it by $\partial^3_{x} \psi$ and integrating in space we get 
\begin{equation}\label{ito rho s=2}
\text{d} \int_{} \frac{ | \partial^3_{x} \psi|^2}{2} dx+ \int_{} \partial^3_{x}(u \partial_x \psi)\partial^3_{x} \psi dx\text{d}t+ \int_{} \partial^4_{x} u \partial^3_{x} \psi dx \text{d}t=0.
\end{equation}
Similarly we apply $\partial_{xx}$ to  \eqref{u nl} and we use It$\hat{\text{o}}$ formula to the functional $F(u)=\frac{1}{2}\int_{} |\partial_{xx} u|^2$
\begin{equation}\label{ito u s=2}
\begin{split}
& \text{d}\int_{}\frac{ | \partial_{xx}u|^2}{2}dx+\int_{}\partial_{xx}(u\partial_x u)\partial_{xx} u dx\text{d}t+ \int_{} \partial_{xx}(\gamma e^{(\gamma-1)\psi }\partial_x \psi) \partial_{xx}udxdt= \\ &\int_{} \partial_{xx} (e^{(\alpha-1)\psi}\partial_{xx}u) \partial_{xx} u dxdt + \int_{} \partial_{xx}(\alpha e^{(\alpha-1)\psi}\partial_x \psi \partial_x u)\partial_{xx} udxdt + \int_{}\partial^5_{x}\psi \partial_{xx} udxdt \\  &+\int_{} \partial_{xx}(\partial_x\psi \partial_{xx} \psi) \partial_{xx} udt + \int_{}\partial_{xx} \mathbb{F}(\rho, u)\partial_{xx} u\text{d}W+ \frac{1}{2} \sum_{k=1}^{\infty} \int_{} |\partial_{xx} F_k|^2 dx
\end{split}
\end{equation}
and summing up the \eqref{ito rho s=2}-\eqref{ito u s=2} we end up with the following equality 
\begin{equation}
\begin{split}
& \text{d}\bigg( \int_{}\frac{ | \partial^3_{x}\psi|^2}{2}+ \frac{| \partial_{xx}u|^2}{2}dx \bigg)+ \int_{} \partial^3_{x}(u \partial_x \psi)\partial^3_{x} \psi dx\text{d}t+ \int_{} \partial^4_{x} u \partial^3_{x} \psi dx \text{d}t \\ & +\int_{}\partial_{xx}(u\partial_x u)\partial_{xx} u dx\text{d}t+ \int_{} \partial_{xx}(\gamma e^{(\gamma-1)\psi }\partial_x \psi) \partial_{xx}udxdt=\int_{} \partial_{xx} (e^{(\alpha-1)\psi}\partial_{xx}u) \partial_{xx} u dxdt \\ & + \int_{} \partial_{xx}(\alpha e^{(\alpha-1)\psi}\partial_x \psi \partial_x u)\partial_{xx} udxdt + \int_{}\partial^5_{x}\psi \partial_{xx} udxdt+\int_{} \partial_{xx}(\partial_x\psi \partial_{xx} \psi) \partial_{xx} udt \\ & + \int_{}\partial_{xx} \mathbb{F}(\rho, u)\partial_{xx} u\text{d}W+ \frac{1}{2} \sum_{k=1}^{\infty} \int_{} |\partial_{xx} F_k|^2 dx,
\end{split}
\end{equation}
which after integrating by parts can be rewritten as
\begin{equation}\label{s=2}
\begin{split}
& \text{d}\bigg( \int_{}\frac{ | \partial^3_{x} \psi|^2}{2}+ \frac{| \partial_{xx} u|^2}{2}dx \bigg)+ \int_{} e^{(\alpha-1) \psi} | \partial^3_{x} u |^2 dxdt=- \int_{} \partial^3_{x}(u \partial_x \psi)\partial^3_{x} \psi dx\text{d}t \\ &-\int_{}\partial_{xx}(u\partial_x u)\partial_{xx} u dx\text{d}t- \int_{} \partial_{xx}(\gamma e^{(\gamma-1)\psi }\partial_x \psi) \partial_{xx}udxdt-\int_{} \partial_x( e^{(\alpha-1)\psi} )\partial_{xx}u \partial^3_x u dxdt \\ &+\int_{} \partial_{xx}(\alpha e^{(\alpha-1)\psi}\partial_x \psi \partial_{xx} u)\partial_{xx} udxdt + \int_{} \partial_{xx}(\partial_x\psi \partial_{xx} \psi) \partial_{xx} udt + \int_{}\partial_{xx} \mathbb{F}(\rho, u)\partial_{xx} u\text{d}W \\ &+ \frac{1}{2} \sum_{k=1}^{\infty} \int_{} |\partial_{xx} F_k|^2 dx= \sum_{i=1}^{8} I_i.
\end{split}
\end{equation}
The integrals in the right hand side of \eqref{s=2} are estimated by using the same lines of argument of the previous case
\begin{equation*}
\begin{split}
|I_1| & =| -\int_{} \partial^3_x(u \partial_x \psi) \partial^3_x \psi dx | \le \int_{} | \partial^3_x u | | \partial_x \psi | | \partial^3_x \psi | + 3 | \partial_{xx} u | | \partial_{xx} \psi | | \partial^3_x \psi | + \frac{5}{2} | \partial_x u | | \partial^3_{x} \psi| ^2 dx  \\ & \le \| \rho^{\frac{\alpha-1}{2}}\partial^3_x u \|_{L^2} \| \partial_x \psi \|_{L^\infty} \| \partial^3_x \psi \|_{L^2} \| \frac{1}{\rho^{\frac{\alpha-1}{2}}}\|_{L^\infty} + c \| \partial_{xx} u \|_{L^\infty} \| \partial_{xx} \psi \|_{L^2} \| \partial^3_x \psi \|_{L^2} \\ & +c \| \partial_x u \|_{L^\infty} \| \partial^3_x \psi \|^2_{L^2}.
\end{split}
\end{equation*}
By Sobolev embeddings and Young inequality we have
\begin{equation}\label{I_1 s=2}
\begin{split}
|I_{1}| & \le  \delta \| \rho^{\frac{\alpha-1}{2}}\partial^3_x u \|^2_{L^2} + C(\delta) \| \partial^3_x \psi \|^2_{L^2} \| \frac{1}{\rho^{\frac{\alpha-1}{2}}} \|^2_{L^\infty} \| \partial_{xx} \psi \|^2_{L^2} \\ & + C_1\| \partial^3_x \psi \|^2_{L^2}  \| \frac{1}{\rho^{\frac{\alpha-1}{2}}} \|^2_{L^\infty} +C_2 \| \partial^3_x \psi \|^2_{L^2} \| \rho^{\frac{\alpha-1}{2}}\partial_{xx} u \|^2_{L^2},
\end{split}
\end{equation}
\begin{equation*}
\begin{split}
|I_2| &=| -\int_{} \partial_{xx}( u \partial_x u ) \partial_{xx} u dx | \le  \int_{} 3 |\partial_{xx} u|^2 | \partial_x u| + |u | | \partial^3_x u| |  \partial_{xx} u|  dx \\ & \le c  \| \partial_{xx} u \|^2_{L^2} \| \partial_x u \|_{L^\infty}+ \| u \|_{L^\infty} \| \partial_{xx} u \|_{L^2} \| \rho^{\frac{\alpha-1}{2}}\partial^3_x u \|^2_{L^2} \| \frac{1}{\rho^{\frac{\alpha-1}{2}}} \|_{L^\infty},
\end{split}
\end{equation*}
hence we end up with
\begin{equation}
\begin{split}
| I_2 | &  \le \| \partial_{xx} u \|^2_{L^2} ( \| \rho^{\frac{\alpha-1}{2}} \partial_{xx} u \|^2_{L^2}+\| \frac{1}{\rho^{\frac{\alpha-1}{2}}} \|^2_{L^\infty} +  \| u \|^2_{L^\infty} \| \frac{1}{\rho^{\frac{\alpha-1}{2}}} \|^2_{L^\infty}) \\ & + \delta \| \rho^{\frac{\alpha-1}{2}}\partial^3_x u \|^2_{L^2}.
\end{split}
\end{equation}
Similarly,  the other integrals $I_i$, $i=1,..,8$ can be handled as follows
\begin{equation*}
\begin{split}
|I_3| & =| -\int_{} \partial_{xx} (\gamma e^{(\gamma-1) \psi } \partial_x \psi) \partial_{xx} u dx |=| -\int_{} \partial_{xx} \bigg( \frac{\partial_x \rho^\gamma}{\rho} \bigg) \partial_{xx} u dx | \\ & \le \int_{} |\frac{\partial_{xx} \rho^\gamma}{\rho}| | \partial^3_x u | +| \partial_x \rho^\gamma | | \partial_x \bigg(\frac{1}{\rho} \bigg) | |  \partial^3_x u | dx \\ & \le \|\frac{1}{\rho^{\frac{\alpha+1}{2}}} \|_{L^\infty} \| \partial_{xx} \rho^\gamma \|_{L^2} \| \rho^{\frac{\alpha-1}{2}}\partial^3_x u \|_{L^2}+ \| \partial_x \rho^\gamma \|_{L^\infty} \| \partial_x \bigg( \frac{1}{\rho} \bigg) \|_{L^2} \|\frac{1}{\rho^{\frac{\alpha-1}{2}}}\|_{L^\infty} \| \rho^{\frac{\alpha-1}{2}}\partial^3_x u \|_{L^2}.
\end{split}
\end{equation*}
\begin{equation}
\begin{split}
|I_3| & \le C(\delta) \|\frac{1}{\rho^{\frac{\alpha+1}{2}}} \|^2_{L^\infty} \| \partial_{xx} \rho^\gamma \|^2_{L^2} + \delta \| \rho^{\frac{\alpha-1}{2}}\partial^3_x u \|^2_{L^2} \\ & + C(\delta) \| \partial_x \rho^\gamma \|^2_{L^\infty} \| \partial_x \bigg( \frac{1}{\rho} \bigg) \|^2_{L^2}\| \frac{1}{\rho^{\frac{\alpha-1}{2}}}\|^2_{L^\infty}.
\end{split}
\end{equation}
\begin{equation*}
\begin{split}
|I_4| & =| -\int_{} \partial_x (e^{(\alpha-1)\psi})\partial_{xx} u \partial^3_x u dx | \\ & \le \int_{}  | \bigg(\frac{\partial_x \rho^\alpha}{\rho}\bigg) | | \partial_{xx} u | |  \partial^3_x u | + \rho^\alpha  | \partial_x \bigg( \frac{1}{\rho}\bigg)| | \partial_{xx} u| | \partial^3_x u | dx \\ & \le \| \partial_x \rho^{\alpha} \|_{L^\infty} \| \frac{1}{\rho^{\frac{\alpha+1}{2}}}\|_{L^{\infty}} \| \partial_{xx} u \|_{L^2} \|\rho^{\frac{\alpha-1}{2}} \partial^3_x u \|_{L^2} \\ & + \| \rho^\alpha \|_{L^\infty} \| \partial_x \bigg( \frac{1}{\rho} \bigg ) \|_{L^\infty} \| \partial_{xx} u \|_{L^2} \| \rho^{\frac{\alpha-1}{2}} \partial^3_x u \|_{L^2}\| \frac{1}{\rho^{\frac{\alpha-1}{2}}}\|_{L^\infty},
\end{split}
\end{equation*}
hence by using similar lines of argument we have 
\begin{equation}
\begin{split}
| I_4| & \le C(\delta)\| \partial_x \rho^{\gamma} \|^2_{L^\infty} \| \frac{1}{\rho^{\frac{\alpha+1}{2}}}\|^2_{L^{\infty}} \| \partial_{xx} u \|^2_{L^2} + \delta \| \rho^{\frac{\alpha-1}{2}}\partial^3_x u \|^2_{L^2} \\ & + C(\delta)  \| \rho^\alpha \|^2_{L^\infty} \| \partial_x \bigg( \frac{1}{\rho} \bigg ) \|^2_{L^\infty} \| \partial_{xx} u \|^2_{L^2}\| \frac{1}{\rho^{\frac{\alpha-1}{2}}}\|^2_{L^\infty}.
\end{split}
\end{equation}
\begin{equation*}
\begin{split}
|I_5| & = | \int_{} \partial_{xx} (\alpha e^{(\alpha-1)\psi} \partial_x \psi \partial_x u ) \partial_{xx} u dx |= | \int_{} \partial_{xx} (\frac{\partial_x \rho^{\alpha}}{\rho} \partial_x u) \partial_{xx} u dx | \\ & \le  \int_{} | \frac{\partial_{xx} \rho^{\alpha}}{\rho}| |  \partial_x u | | \partial^3_x u | + | \partial_x \rho^{\alpha} | | \partial_x \bigg(\frac{1}{\rho} \bigg)| | \partial_x u| |  \partial^3_x u | + | \frac{\partial_x \rho^\alpha}{\rho} | | \partial_{xx} u | | \partial^3_x u|  dx \\ & \le \| \frac{1}{\rho^{\frac{\alpha+1}{2}}} \|_{L^\infty} \| \partial_{xx} \rho^{\alpha} \|_{L^2} \| \partial_x u \|_{L^\infty} \| \rho^{\frac{\alpha-1}{2}}\partial^3_x u \|_{L^2} \\ & + \| \partial_x \rho^\alpha \|_{L^\infty} \| \partial_x \bigg( \frac{1}{\rho} \bigg) \|_{L^2} \| \partial_x u \|_{L^\infty} \| \rho^{\frac{\alpha-1}{2}}\partial^3_x u \|_{L^2}\| \frac{1}{\rho^{\frac{\alpha-1}{2}}} \|_{L^\infty} \\ & +  \| \frac{1}{\rho^{\frac{\alpha+1}{2}}} \|_{L^\infty} \| \partial_x \rho^\alpha \|_{L^\infty} \| \partial_{xx} u \|_{L^2} \| \rho^{\frac{\alpha-1}{2}}\partial^3_x u \|_{L^2},
\end{split}
\end{equation*}
from which we deduce
\begin{equation}
\begin{split}
| I_5| & \le C(\delta) \| \frac{1}{\rho^{\frac{\alpha+1}{2}}} \|^2_{L^\infty} \| \partial_{xx} \rho^{\alpha} \|^2_{L^2} \| \partial_{xx} u \|^2_{L^2}+ \delta \| \rho^{\frac{\alpha-1}{2}}\partial^3_x u \|^2_{L^2} \\ & +C(\delta) \| \partial_{xx} \rho^\alpha \|^2_{L^2} \| \partial_x \bigg( \frac{1}{\rho} \bigg) \|^2_{L^2} \| \partial_{xx} u \|^2_{L^2} \| \frac{1}{\rho^{\frac{\alpha-1}{2}}} \|^2_{L^\infty} \\ & + C(\delta) \| \frac{1}{\rho^{\frac{\alpha+1}{2}}} \|^2_{L^\infty} \| \partial_x \rho^\alpha \|^2_{L^2} \| \partial_{xx} u \|^2_{L^2}.
\end{split}
\end{equation}
\begin{equation*}
\begin{split}
|I_6| & =| \int_{} \partial_{xx} ( \partial_x \psi \partial_{xx} \psi ) \partial_{xx} u dx | \le \int_{} 2 | \partial_{xx} \psi | | \partial^3_x \psi | | \partial_{xx} u | + |  \partial_x \psi | | \partial^3_x \psi | |  \partial^3_x u | dx \\ & \le c \| \partial_{xx} \psi \|_{L^2} \| \partial^3_x \psi \|_{L^2} \| \partial_{xx} u \|_{L^\infty} +\| \partial_x \psi \|_{L^\infty} \| \partial^3_x \psi \|_{L^2} \| \rho^{\frac{\alpha-1}{2}}\partial^3_x u \|_{L^2} \| \frac{1}{\rho^{\frac{\alpha-1}{2}}} \|_{L^\infty},
\end{split}
\end{equation*}
\begin{equation}
\begin{split}
| I_6| & \le \delta\| \rho^{\frac{\alpha-1}{2}} \partial^3_x u \|^2_{L^2}+ C(\delta) \| \partial^3_x \psi \|^2_{L^2} \| \partial_{xx} \psi \|^2_{L^2} \| \frac{1}{ \rho^{\frac{\alpha-1}{2}}} \|^2_{L^\infty}.
\end{split}
\end{equation}
The It$\hat{\text{o}}$ reminder term $I_8$ can be estimated as follows
\begin{equation}\label{I_8 s=2}
\begin{split}
|I_8| &=\frac{1}{2} \sum_{k=1}^{\infty} \int_{\mathbb{T}} |\partial_{xx} F_k|^2 dx=\frac{1}{2} \sum_{k=1}^{\infty} \| \partial_{xx} F_k \|^2_{L^2} \le c \sum_{k=1}^{\infty} \| \partial_{xx} F_k \|^2_{L^\infty} \le c \sum_{k=1}^{\infty} {\alpha_k}^2 \le C.
\end{split}
\end{equation}
Summing up the estimates \eqref{I_1 s=2}-\eqref{I_8 s=2} and by choosing properly the Young constant $\delta$ we have 
\begin{equation}
\begin{split}
& \dfrac{1}{2}\text{d} ( \| \partial^3_x \psi \|^2_{L^2} + \| \partial_{xx} u \|^2_{L^2} ) + c\int_{\mathbb{T}} e^{(\alpha-1) \psi } | \partial^3_x u |^2 dx \\ &  \le \bigg(\| \partial^3_x \psi \|^2_{L^2} + \| \partial_{xx} u \|^2_{L^2}\bigg) \bigg[ 1+ \| \frac{1}{\rho^{\frac{\alpha-1}{2}}} \|^2_{L^\infty}\| \partial_{xx} \psi \|^2_{L^2}  \| \frac{1}{\rho^{\frac{\alpha-1}{2}}} \|^2_{L^\infty} + \| \rho^{\frac{\alpha-1}{2}}\partial_{xx} u \|^2_{L^2} \\ & + \| u \|^2_{L^\infty}\| \frac{1}{\rho^{\frac{\alpha-1}{2}}} \|^2_{L^\infty} + \| \partial_x \rho^\alpha \|^2_{L^\infty}  \| \frac{1}{\rho^{\frac{\alpha+1}{2}}} \|^2_{L^\infty} + \| \rho^\alpha \|^2_{L^\infty} \| \partial_x (\frac{1}{\rho})\|^2_{L^\infty}\| \frac{1}{\rho^{\frac{\alpha-1}{2}}} \|^2_{L^\infty} \\ & +  \| \frac{1}{\rho^{\frac{\alpha+1}{2}}} \|^2_{L^\infty} \| \partial_{xx} \rho^\alpha \|^2_{L^2} + \| \partial_{xx} \rho^\alpha \|^2_{L^2} \| \partial_x (\frac{1}{\rho}) \|^2_{L^2}\| \frac{1}{\rho^{\frac{\alpha-1}{2}}} \|^2_{L^\infty} + \| \frac{1}{\rho^{\frac{\alpha+1}{2}}} \|^2_{L^\infty} \| \partial_x \rho^\alpha \|^2_{L^2} \bigg] \\ & + \| \partial_{xx} \rho^\gamma \|^2_{L^2}\| \frac{1}{\rho^{\frac{\alpha+1}{2}}} \|^2_{L^\infty} + \| \partial_x \rho^\gamma \|^2_{L^\infty} \| \partial_x (\frac{1}{\rho}) \|^2_{L^2} \| \frac{1}{\rho^{\frac{\alpha-1}{2}}} \|^2_{L^\infty} \\ & = \bigg(\| \partial^3_x \psi \|^2_{L^2} + \| \partial_{xx} u \|^2_{L^2}\bigg)a^{(2)}(t)+b^{(2)}(t)+I_7.
\end{split}
\end{equation}
We define the stopping time 
\begin{equation}
\gamma^{(2)}_M= \inf \big\{ t\ge 0 \; :  \int_{0}^{t \land \tau} a^{(2)}(s)ds > M \big\},
\end{equation}
and by using the same lines of argument used in order to prove \eqref{eq:s1} we get
\begin{equation}
\begin{split}
& \mathbb{E} \sup_{\tau_a \le s \le \tau_b} \frac{1}{2}( \| \partial^3_x \psi \|^2_{L^2} + \| \partial_{xx} u \|^2_{L^2} ) + c\mathbb{E} \int_{\tau_a}^{\tau_b}\int_{\mathbb{T}} e^{(\alpha-1) \psi } | \partial^3_x u |^2 dxdt \\ &  \le \mathbb{E} ( \| \partial^3_x \psi (\tau_a) \|^2_{L^2} + \| \partial_{xx} u (\tau_a) \|^2_{L^2} )+ \mathbb{E} \int_{\tau_a}^{\tau_b}(\| \partial^3_x \psi \|^2_{L^2} + \| \partial_{xx} u \|^2_{L^2})a^{(2)}(t) \\ & +\mathbb{E} \int_{\tau_a}^{\tau_b}b^{(2)}(t) + \mathbb{E} \sup_{ \tau_a \le s \le \tau_b} | \int_{\tau_a}^{s}I_7|.
\end{split}
\end{equation}
In order to handle the stochastic integral $I_7$ we observe that 
\begin{equation*}
\begin{split}
\mathbb{E} \bigg[\sup_{ \tau_a \le s \le \tau_b}  \bigg| {\int_{0}^{s} \int_{\mathbb{T}} \partial_{xx} \mathbb{F}(\rho,u)\partial_{xx} udxdW \bigg| } \bigg] \lesssim \mathbb{E} \bigg[ \int_{\tau_a}^{\tau_b} \sum_{k=1}^{\infty} \bigg| \int_{\mathbb{T}} \partial_{xx} F_k \partial_{xx} u dx \bigg|^2 dt \bigg]^{\frac{1}{2}}
\end{split}
\end{equation*}
and 
\begin{equation*}
\begin{split}
 \bigg| \int_{\mathbb{T}} \partial_{xx} F_k \partial_{xx} u dx \bigg|^2 \le \| \partial_{xx} F_k \|^2_{L^2} \| \partial_{xx} u \|^2_{L^2} \le C \| \partial_{xx} F_k \|^2_{L^\infty} \| \partial_{xx} u \|^2_{L^2} \le C {\alpha_k}^2 \| \partial_{xx} u \|^2_{L^2},
\end{split}
\end{equation*}
from which we deduce
\begin{equation*}
\begin{split}
& \mathbb{E} \bigg[ \int_{\tau_a}^{\tau_b} \sum_{k=1}^{\infty} \bigg| \int_{\mathbb{T}} \partial_{xx} F_k \partial_{xx} u dx \bigg|^2 dt \bigg]^{\frac{1}{2}} \le \mathbb{E} \bigg[ \int_{\tau_a}^{\tau_b} \sum_{k=1}^{\infty}C {\alpha_k}^2 \| \partial_{xx} u \|^2_{L^2} \bigg]^{\frac{1}{2}} \\ & \le \mathbb{E} \bigg[ \int_{\tau_a}^{\tau_b} \tilde{c} \| \partial_{xx} u \|^2_{L^2} \bigg]^{\frac{1}{2}} \le \mathbb{E} \bigg[ 1+ \int_{\tau_a}^{\tau_b} c \| \partial_{xx} u \|^2_{L^2} \bigg],
\end{split}
\end{equation*}
so we end up with the following inequality
\begin{equation}
\begin{split}
& \mathbb{E} \sup_{ \tau_a \le s \le \tau_b} \frac{1}{2}( \| \partial^3_x \psi \|^2_{L^2} + \| \partial_{xx} u \|^2_{L^2} ) + c\mathbb{E} \int_{\tau_a}^{\tau_b}\int_{\mathbb{T}} e^{(\alpha-1) \psi } | \partial^3_x u |^2 dxdt \\ & \le \mathbb{E} ( \| \partial^3_x \psi (\tau_a) \|^2_{L^2} + \| \partial_{xx} u (\tau_a) \|^2_{L^2} )+ \mathbb{E} \int_{\tau_a}^{\tau_b}(\| \partial^3_x \psi \|^2_{L^2} + \| \partial_{xx} u \|^2_{L^2})a^{(2)}(t)dt \\ & +\mathbb{E} \int_{\tau_a}^{\tau_b}b^{(2)}(t)dt.
\end{split}
\end{equation}
Once more we make use of the stochastic Gronwall Lemma,  Proposition \ref{stoch gronw},  in order to obtain
\begin{equation}
\mathbb{E} \bigg(\sup_{  s \in [0, T \land \gamma^{(2)}_M \land \tau] }\dfrac{1}{2}( \| \partial^3_x \psi \|^2_{L^2}+ \| \partial_{xx} u \|^2_{L^2}) + \int_{0}^{T \land \gamma^{(2)}_M \land \tau}\int_{\mathbb{T}} e^{(\alpha-1) \psi } | \partial^3_x u|^2 dxds \bigg) < +\infty.
\end{equation}
\\
In this case in order to prove that $\lim_{M \to \infty} \gamma^{(2)}_M= \infty$ a.s. we observe that $$\mathbb{E} \int_{0}^{T \land \tau} \| \rho^{\frac{\alpha-1}{2}}\partial_{xx} u \|^2_{L^2} < \infty $$
and the other terms can be estimates as in the case $s=1,$
hence we conclude
\begin{equation}\label{markov M s=2}
\mathbb{P} \big( \gamma^{(2)}_M \le T \big) \le\mathbb{P} \big( \int_{0}^{T \land \tau} a^{(2)}(t)dt \ge M \big) \le \frac{1}{M} \mathbb{E} \big( \int_{0}^{T \land \tau} a^{(2)}(t)dt \big) 
\end{equation}
which goes to zero as $M \rightarrow \infty.$ \\
Finally, we prove that 
\begin{equation}\label{eq:s3}
\begin{split}
& \partial^4_{x} \psi \in L^1(\Omega;L^\infty(0,T \land \gamma_M \land \tau;L^2(\mathbb{T}))),\quad \partial^3_{x} u \in L^1(\Omega;L^\infty(0,T \land \gamma_M \land \tau;L^2(\mathbb{T}))), \\ & 
\rho^{\frac{\alpha-1}{2}}\partial^4_{x} u \in L^1(\Omega;L^2(0,T \land \gamma_M \land \tau;L^2(\mathbb{T}))).
\end{split}
\end{equation}
As in the previous cases we apply $\partial^4_{x}$ to \eqref{psi nl}, we multiply it by $\partial^4_{x} \psi$ and integrating in space we get 
\begin{equation}\label{ito rho s=3}
\text{d} \int_{} \frac{ | \partial^4_{x} \psi|^2}{2} dx+ \int_{} \partial^4_{x}(u \partial_x \psi)\partial^4_{x} \psi dx\text{d}t+ \int_{} \partial^5_{x} u \partial^4_{x} \psi dx, \text{d}t=0,
\end{equation}
we apply $\partial^3_{x}$ to  \eqref{u nl} and we apply It$\hat{\text{o}}$ formula to the functional $F(u)=\frac{1}{2}\int_{} |\partial^3_{x} u|^2$
\begin{equation}\label{ito u s=3}
\begin{split}
& \text{d}\int_{}\frac{| | \partial^3_{x}u|^2|}{2}dx+\int_{}\partial^3_{x}(u\partial_x u)\partial^3_{x} u dx\text{d}t+ \int_{} \partial^3_{x}(\gamma e^{(\gamma-1)\psi }\partial_x \psi) \partial^3_{x}udxdt= \\ &\int_{} \partial^3_{x} (e^{(\alpha-1)\psi}\partial_{xx}u) \partial^3_{x} u dxdt + \int_{} \partial^3_{x}(\alpha e^{(\alpha-1)\psi}\partial_x \psi \partial_x u)\partial^3_{x} udxdt + \int_{}\partial^6_{x}\psi \partial^3_{x} udxdt \\  &+\int_{} \partial^3_x(\partial_x\psi \partial_{xx} \psi) \partial^3_x udt + \int_{}\partial^3_{x} \mathbb{F}(\rho, u)\partial^3_{x} u\text{d}W+ \frac{1}{2} \sum_{k=1}^{\infty} \int_{} |\partial^3_{x} F_k|^2 dx
\end{split}
\end{equation}
and summing up \eqref{ito rho s=3}-\eqref{ito u s=3} we end up with the following equality 
\begin{equation}
\begin{split}
& \text{d}\bigg( \int_{}\frac{ | \partial^4_{x}\psi|^2}{2}+ \frac{| \partial^3_{x}u|^2}{2}dx \bigg)+ \int_{} \partial^4_{x}(u \partial_x \psi)\partial^4_{x} \psi dx\text{d}t+ \int_{} \partial^5_{x} u \partial^4_{x} \psi dx \text{d}t \\ & +\int_{}\partial^3_{x}(u\partial_x u)\partial^3_{x} u dx\text{d}t+ \int_{} \partial^3_{x}(\gamma e^{(\gamma-1)\psi }\partial_x \psi) \partial^3_{x}udxdt=\int_{} \partial^3_{x} (e^{(\alpha-1)\psi}\partial_{xx}u) \partial^3_{x} u dxdt \\ & + \int_{} \partial^3_{x}(\alpha e^{(\alpha-1)\psi}\partial_x \psi \partial_x u)\partial^3_{x} udxdt + \int_{}\partial^6_{x}\psi \partial^3_{x} udxdt+\int_{} \partial^3_{x}(\partial_x\psi \partial_{xx} \psi) \partial^3_{x} udt \\ & + \int_{}\partial^3_{x} \mathbb{F}(\rho, u)\partial^3_{x} u\text{d}W+ \frac{1}{2} \sum_{k=1}^{\infty} \int_{} |\partial^3_{x} F_k|^2 dx,
\end{split}
\end{equation}
which after integrating by parts we rewrite as
\begin{equation}\label{s=3}
\begin{split}
& \text{d}\bigg( \int_{}\frac{ | \partial^4_{x} \psi|^2}{2}+ \frac{| \partial^3_{x} u|^2}{2}dx \bigg)+ \int_{} e^{(\alpha-1) \psi} | \partial^4_{x} u |^2 dxdt=- \int_{} \partial^4_{x}(u \partial_x \psi)\partial^4_{x} \psi dx\text{d}t \\ &-\int_{}\partial^3_{x}(u\partial_x u)\partial^3_{x} u dx\text{d}t- \int_{} \partial^3_{x}(\gamma e^{(\gamma-1)\psi }\partial_x \psi) \partial^3_{x}udxdt-\int_{} \partial_{xx}( e^{(\alpha-1)\psi} )\partial_{xx}u \partial^4_x u dxdt \\ &-2\int_{} \partial_x(e^{(\alpha-1)\psi})\partial^3_x u \partial^4_x u+\int_{} \partial^3_{x}(\alpha e^{(\alpha-1)\psi}\partial_x \psi \partial_{xx} u)\partial^3_x udxdt + \int_{} \partial^3_{x}(\partial_x\psi \partial_{xx} \psi) \partial^3_{x} udt \\ & + \int_{}\partial^3_{x} \mathbb{F}(\rho, u)\partial^3_{x} u\text{d}W + \frac{1}{2} \sum_{k=1}^{\infty} \int_{} |\partial^3_{x} F_k|^2 dx= \sum_{i=1}^{9} I_i.
\end{split}
\end{equation}
The following computations use the same lines of argument of the previous cases
\begin{equation*}
\begin{split}
|I_1| & = | -\int{} \partial^4_x (u \partial_x \psi ) \partial^4_x \psi dx | \\ & \le \int_{} | \partial^4_x u| |  \partial_x \psi | | \partial^4_x \psi| + 4 | \partial^3_x u| |  \partial_{xx} \psi | | \partial^4_x \psi | +5 |  \partial_{xx} u | |  \partial^3_x \psi | | \partial^4_x \psi | + \frac{3}{2} | \partial_x u | | \partial^4_x \psi |^2 dx \\ & \le \| \rho^{\frac{\alpha-1}{2}}\partial^4_x u \|_{L^2} \| \partial_x \psi \|_{L^\infty} \| \partial^4_x \psi \|_{L^2} \| \frac{1}{\rho^{\frac{\alpha-1}{2}}} \|_{L^\infty} + c \| \partial^3_x u \|_{L^2} \| \partial_{xx} \psi \|_{L^\infty} \| \partial^4_x \psi \|_{L^2} \\ & + c \| \partial_{xx} u \|_{L^\infty} \| \partial^3_x \psi \|_{L^2} \| \partial^4_x \psi \|_{L^2} + c \| \partial_x u \|_{L^\infty} \| \partial^4_x \psi \|^2_{L^2},
\end{split}
\end{equation*}
hence 
\begin{equation}\label{I_1 s=3}
\begin{split}
|I_1| & \le C(\delta) \| \partial_x \psi \|^2_{L^\infty} \| \partial^4_x \psi \|^2_{L^2} \| \frac{1}{\rho^{\frac{\alpha-1}{2}}}\|^2_{L^\infty} + \delta \| \rho^{\frac{\alpha-1}{2}}\partial^4_x u \|_{L^2} + C_1 \| \partial^3_x u \|^2_{L^2} \\ & + C_2 \| \partial_{xx} \psi \|^2_{L^\infty} \| \partial^4_x \psi \|^2_{L^2} + C_3 \| \partial^3_x \psi \|^2_{L^2} \| \partial^4_x \psi \|^2_{L^2} + C_4 \| \partial_x u \|_{L^\infty} \| \partial^4_x \psi \|^2_{L^2}.
\end{split}
\end{equation}
\begin{equation*}
\begin{split}
|I_2| & = | -\int_{} \partial^3_x (u \partial_x u ) \partial^3_x u dx |  \le \int_{} ( 2 | \partial_x u| |  \partial_{xx} u| + | \partial_x u| |  \partial_{xx} u | +|u | | \partial^3_x u |)| \partial^4_x u| dx \\ & \le c \| \partial_x u \|_{L^\infty} \| \partial_{xx} u \|_{L^2} \| \rho^{\frac{\alpha-1}{2}}\partial^4_x u \|_{L^2}\| \frac{1}{\rho^{\frac{\alpha-1}{2}}}\|_{L^\infty} + \| \partial_x u \|_{L^\infty} \| \partial_{xx} u \|_{L^2} \| \partial^4_x u \|_{L^2} \\ & + \| u \|_{L^\infty} \| \rho^{\frac{\alpha-1}{2}}\partial^4_x u \|_{L^2} \| \partial^3_x u \|_{L^2}\| \frac{1}{\rho^{\frac{\alpha-1}{2}}}\|_{L^\infty},
\end{split}
\end{equation*}
by Sobolev and Young inequalities we infer
\begin{equation}
\begin{split}
| I_2| & \le \delta \| \rho^{\frac{\alpha-1}{2}}\partial^4_x u \|^2_{L^2}+ C(\delta) \| \partial_x u \|^2_{L^\infty} \| \partial_{xx} u \|^2_{L^2} \| \frac{1}{\rho^{\frac{\alpha-1}{2}}}\|^2_{L^\infty} \\ & + C(\delta) \| u \|^2_{L^\infty} \| \partial^3_x u \|^2_{L^2} \| \frac{1}{\rho^{\frac{\alpha-1}{2}}}\|^2_{L^\infty}.
\end{split}
\end{equation}
\begin{equation*}
\begin{split}
|I_3| & = | -\int_{} \partial_{xx} \bigg( \frac{\partial_x \rho^\gamma}{\rho}\bigg) \partial^4_x u dx | \le  \int_{} \bigg( | \frac{\partial^3_x \rho^ \gamma}{\rho}|+2 | \partial_{xx} \rho^\gamma | | \partial_x ( \frac{1}{\rho})|+ | \partial_x \rho^\gamma | | \partial_{xx}(\frac{1}{\rho})|  \bigg)|  \partial^4_x u|  dx \\ & \le  \| \frac{1}{\rho^{\frac{\alpha+1}{2}}} \|_{L^\infty} \| \partial^3_x \rho^{\gamma}\|_{L^2} \| \rho^{\frac{\alpha-1}{2}}\partial^4_x u \|_{L^2} + c \| \partial_{xx} \rho^\gamma \|_{L^\infty} \| \partial_x ( \frac{1}{\rho} ) \|_{L^2} \| \rho^{\frac{\alpha-1}{2}}\partial^4_x u \|_{L^2} \| \frac{1}{\rho^{\frac{\alpha-1}{2}}}\|_{L^\infty} \\ & + \| \partial_x \rho^\gamma \|_{L^\infty} \| \partial_{xx} (\frac{1}{\rho} )\|_{L^2} \| \rho^{\frac{\alpha-1}{2}}\partial^4_x u \|_{L^2}\| \frac{1}{\rho^{\frac{\alpha-1}{2}}}\|_{L^\infty},
\end{split}
\end{equation*}
from which we deduce
\begin{equation}
\begin{split}
| I_3| & \le C(\delta) \| \frac{1}{\rho^{\frac{\alpha+1}{2}}} \|^2_{L^\infty} \| \partial^3_x \rho^\gamma \|^2_{L^2} + \delta \| \rho^{\frac{\alpha-1}{2}}\partial^4_x u \|^2_{L^2} \\ & + C(\delta) \| \partial_{xx} \rho^{\gamma}\|^2_{L^\infty} \| \partial_x ( \frac{1}{\rho}) \|^2_{L^2} \| \frac{1}{\rho^{\frac{\alpha-1}{2}}}\|^2_{L^\infty}+ \delta \| \rho^{\frac{\alpha-1}{2}}\partial^4_x u \|^2_{L^2} \\ & + C(\delta) \| \partial_x \rho^\gamma \|^2_{L^\infty} \| \partial_{xx} (\frac{1}{\rho})\|^2_{L^2}\| \frac{1}{\rho^{\frac{\alpha-1}{2}}}\|^2_{L^\infty} .
\end{split}
\end{equation}
\begin{equation*}
\begin{split}
|I_4| & = | -\int_{} \partial_{xx} \bigg( \frac{\rho^\alpha}{\rho} \bigg) \partial_{xx} u \partial^4_x u dx | \\ & \le  \int_{} | \frac{\partial_{xx} \rho^\alpha}{\rho}| | \partial_{xx} u| | \partial^4_x u|+2| \partial_x \rho^\alpha| | \partial_x (\frac{1}{\rho})| |  \partial_{xx} u| | \partial^4_x u| +  \rho^\alpha|  \partial_{xx} (\frac{1}{\rho})| |  \partial_{xx} u| | \partial^4_x u | dx \\ & \le  \| \frac{1}{\rho^{\frac{\alpha+1}{2}}}\|_{L^\infty} \| \partial_{xx} \rho^\alpha \|_{L^\infty} \| \partial_{xx} u \|_{L^2} \| \rho^{\frac{\alpha-1}{2}}\partial^4_x u \|_{L^2} \\ & +c \| \partial_x \rho^\alpha \|_{L^\infty} \| \partial_x ( \frac{1}{\rho}) \|_{L^\infty} \| \partial_{xx} u \|_{L^2} \| \rho^{\frac{\alpha-1}{2}}\partial^4_x u \|_{L^2} \| \frac{1}{\rho^{\frac{\alpha-1}{2}}} \|_{L^\infty} \\ & + \| \rho^\alpha \|_{L^{\infty}} \| \partial_{xx} (\frac{1}{\rho})\|_{L^2} \| \partial_{xx} u \|_{L^\infty} \| \rho^{\frac{\alpha-1}{2}}\partial^4_x u \|_{L^2}\| \frac{1}{\rho^{\frac{\alpha-1}{2}}} \|_{L^\infty},
\end{split}
\end{equation*}
hence 
\begin{equation}
\begin{split}
| I_4| & \le C(\delta) \| \frac{1}{\rho^{\frac{\alpha-1}{2}}}\|^2_{L^\infty} \| \partial_{xx} \rho^\alpha \|^2_{L^\infty} \| \partial_{xx} u \|^2_{L^2} +\delta \| \rho^{\frac{\alpha-1}{2}}\partial^4_x u \|^2_{L^2} \\ & + C(\delta) \| \partial_x \rho^\alpha \|^2_{L^\infty} \| \partial_x ( \frac{1}{\rho}) \|^2_{L^\infty} \| \partial_{xx} u \|^2_{L^2}\| \frac{1}{\rho^{\frac{\alpha-1}{2}}}\|^2_{L^\infty} \\ & + C(\delta) \| \rho^\alpha \|^2_{L^{\infty}} \| \partial_{xx} (\frac{1}{\rho})\|^2_{L^2} \| \frac{1}{\rho^{\frac{\alpha-1}{2}}}\|^2_{L^\infty} \| \partial^3_x u \|^2_{L^2}.
\end{split}
\end{equation}
\begin{equation*}
\begin{split}
|I_5| & = |-2\int_{} \partial_x \bigg( \frac{\rho^\alpha}{\rho} \bigg) \partial^3_x u \partial^4_x u dx | \le 2 \int_{} \bigg( | \frac{\partial_x \rho^\alpha}{\rho}|+ \rho^\alpha | \partial_x (\frac{1}{\rho})| \bigg)| \partial^3_x u| | \partial^4_x u | dx \\ & \le c \| \partial_x \rho^\alpha \|_{L^\infty} \| \frac{1}{\rho^{\frac{\alpha+1}{2}}}\|_{L^\infty} \| \partial^3_x u \|_{L^2} \| \rho^{\frac{\alpha-1}{2}}\partial^4_x u \|_{L^2} \\ & + \| \rho^\alpha \|_{L^\infty} \| \partial_x (\frac{1}{\rho}) \|_{L^\infty} \| \partial^3_x u \|_{L^2} \| \rho^{\frac{\alpha-1}{2}}\partial^4_x u \|_{L^2} \| \frac{1}{\rho^{\frac{\alpha-1}{2}}}\|_{L^\infty},
\end{split}
\end{equation*}
that can be estimated as
\begin{equation}
\begin{split}
| I_5| & \le C(\delta)\| \partial_x \rho^\alpha \|^2_{L^\infty} \| \frac{1}{\rho^{\frac{\alpha+1}{2}}}\|^2_{L^\infty} \| \partial^3_x u \|^2_{L^2} + \delta \| \rho^{\frac{\alpha-1}{2}}\partial^4_x u \|^2_{L^2} \\ & + C(\delta) \| \rho^\alpha \|^2_{L^\infty} \| \partial_x (\frac{1}{\rho}) \|^2_{L^\infty} \| \partial^3_x u \|^2_{L^2}\| \frac{1}{\rho^{\frac{\alpha-1}{2}}}\|^2_{L^\infty}.
\end{split}
\end{equation}
\begin{equation*}
\begin{split}
|I_6| & =  | \int_{} \partial^3_x ( \frac{\partial_x \rho^\alpha }{\rho} \partial_x u ) \partial^3_x u dx |= | -\int_{}\partial_{xx} ( \frac{\partial_x \rho^\alpha }{\rho} \partial_x u ) \partial^4_x u dx | \\ & \le  \int_{} \bigg( | \frac{\partial^3_x \rho^\alpha}{\rho}| | \partial_x u| +2 |\partial_{xx} \rho^\alpha| |  \partial_x (\frac{1}{\rho})| | \partial_x u| + 2 | \frac{\partial_{xx} \rho^\alpha}{\rho}| |  \partial_{xx} u| \\ & +| \partial_x \rho^\alpha| | \partial_{xx} (\frac{1}{\rho})| |  \partial_x u| + 2| \partial_x \rho^\alpha| | \partial_x (\frac{1}{\rho})| |  \partial_{xx} u| + | \frac{\partial_x \rho^\alpha}{\rho}| |  \partial^3_x u| \bigg) | \partial^4_x u| dx \\ & \le c \| \frac{1}{\rho^{\frac{\alpha+1}{2}}}\|_{L^\infty} \| \partial^3_x \rho^\alpha \|_{L^2} \| \partial_x u \|_{L^\infty} \|\rho^{\frac{\alpha-1}{2}} \partial^4_x u \|_{L^2} \\ & + c \| \partial_{xx} \rho^\alpha \|_{L^2} \| \partial_x ( \frac{1}{\rho}) \|_{L^\infty}  \| \partial_x u \|_{L^\infty} \| \rho^{\frac{\alpha-1}{2}}\partial^4_x u \|_{L^2}\| \frac{1}{\rho^{\frac{\alpha-1}{2}}}\|_{L^\infty} \\ & + c \| \frac{1}{\rho^{\frac{\alpha+1}{2}}} \|_{L^\infty} \| \partial_{xx} \rho^\alpha \|_{L^\infty} \| \partial_{xx} u \|_{L^2} \|\rho^{\frac{\alpha-1}{2}} \partial^4_x u \|_{L^2} \\ & + \| \partial_x \rho^\alpha \|_{L^\infty} \| \partial_{xx} (\frac{1}{\rho}) \|_{L^2} \| \partial_x u \|_{L^\infty} \| \rho^{\frac{\alpha-1}{2}}\partial^4_x u \|_{L^2} \| \frac{1}{\rho^{\frac{\alpha-1}{2}}}\|_{L^\infty} \\ & + c \| \partial_x \rho^\alpha \|_{L^\infty} \| \partial_x ( \frac{1}{\rho}) \|_{L^\infty} \| \partial_{xx} u \|_{L^2} \| \rho^{\frac{\alpha-1}{2}}\partial^4_x u \|_{L^2} \| \frac{1}{\rho^{\frac{\alpha-1}{2}}}\|_{L^\infty} \\ & + \| \frac{1}{\rho^{\frac{\alpha+1}{2}}} \|_{L^\infty} \| \partial_x \rho^\alpha \|_{L^\infty} \| \partial^3_x u \|_{L^2} \|\rho^{\frac{\alpha-1}{2}} \partial^4_x u \|_{L^2},
\end{split}
\end{equation*}
which after a long but simple computation can be estimated by
\begin{equation}
\begin{split}
| I_6| & \le C(\delta) \| \frac{1}{\rho^{\frac{\alpha+1}{2}}}\|^2_{L^\infty} \| \partial^3_x \rho^\alpha \|^2_{L^2} \| \partial_x u \|^2_{L^\infty} +\delta \| \rho^{\frac{\alpha-1}{2}}\partial^4_x u \|^2_{L^2} \\ & + C(\delta)  \| \partial_{xx} \rho^\alpha \|^2_{L^2} \| \partial_x ( \frac{1}{\rho}) \|^2_{L^\infty}  \| \partial_x u \|^2_{L^\infty} \| \frac{1}{\rho^{\frac{\alpha-1}{2}}}\|^2_{L^\infty} \\ & + C(\delta)\| \frac{1}{\rho^{\frac{\alpha+1}{2}}} \|^2_{L^\infty} \| \partial_{xx} \rho^\alpha \|^2_{L^\infty} \| \partial_{xx} u \|^2_{L^2} \\ & + C(\delta)\| \partial_x \rho^\alpha \|^2_{L^\infty} \| \partial_{xx} (\frac{1}{\rho}) \|^2_{L^2} \| \partial_x u \|^2_{L^\infty} \| \frac{1}{\rho^{\frac{\alpha-1}{2}}}\|^2_{L^\infty} \\ & + C(\delta) \| \partial_x \rho^\alpha \|^2_{L^\infty} \| \partial_x ( \frac{1}{\rho}) \|^2_{L^\infty} \| \partial_{xx} u \|^2_{L^2} \| \frac{1}{\rho^{\frac{\alpha-1}{2}}}\|^2_{L^\infty} \\ & + C(\delta) \| \frac{1}{\rho^{\frac{\alpha+1}{2}}} \|^2_{L^\infty} \| \partial_x \rho^\alpha \|^2_{L^\infty} \| \partial^3_x u \|^2_{L^2}.
\end{split}
\end{equation}
\begin{equation*}
\begin{split}
|I_7| & = | \int_{} \partial^3_x ( \partial_x \psi \partial_{xx} \psi ) \partial^3_x u dx |  \\ & \le  | \int_{} (2 | \partial_{xx} \psi |^2 + 2 | \partial_x \psi | | \partial^3_x \psi| + | \partial^3_x \psi |^2+ | \partial_{xx} \psi| |  \partial^4_x \psi | ) | \partial^3_x u| +| \partial_x \psi | | \partial^4_x \psi | | \partial^4_x u| dx \\ & \le c \| \partial_{xx} \psi \|^2_{L^2} \| \partial^3_x u \|_{L^\infty} + c \| \partial_x \psi \|_{L^\infty} \| \partial^3_x \psi \|_{L^2} \| \partial^3_x u \|_{L^2} + \| \partial^3_x \psi \|^2_{L^2} \| \partial^3_x u \|_{L^\infty} \\ & + \| \partial_{xx} \psi \|_{L^\infty} \| \partial^4_x \psi \|_{L^2} \| \partial^3_x u \|_{L^2} + \| \partial_x \psi \|_{L^\infty} \| \partial^4_x \psi \|_{L^2} \| \rho^{\frac{\alpha-1}{2}}\partial^4_x u \|_{L^2} \| \frac{1}{\rho^{\frac{\alpha-1}{2}}}\|_{L^\infty},
\end{split}
\end{equation*}
hence 
\begin{equation}
\begin{split}
| I_7| & \le C(\delta) \| \partial_{xx} \psi \|^2_{L^2} \| \partial_{xx} \psi \|^2_{L^2}\| \frac{1}{\rho^{\frac{\alpha-1}{2}}}\|^2_{L^\infty} +\delta \| \rho^{\frac{\alpha-1}{2}}\partial^4_x u \|^2_{L^2} \\ & + C_1 \| \partial_x \psi \|^2_{L^\infty} \| \partial^3_x \psi \|^2_{L^2}+C_2 \| \partial^3_x u \|^2_{L^2} + C(\delta)\| \partial^3_x \psi \|^2_{L^2}\| \partial^3_x \psi \|^2_{L^2} \| \frac{1}{\rho^{\frac{\alpha-1}{2}}}\|^2_{L^\infty} \\ & + C_1 \| \partial_{xx} \psi \|^2_{L^\infty} \| \partial^4_x \psi \|^2_{L^2} + C(\delta) \| \partial_x \psi \|^2_{L^\infty} \| \partial^4_x \psi \|^2_{L^2} \| \frac{1}{\rho^{\frac{\alpha-1}{2}}}\|^2_{L^\infty}.
\end{split}
\end{equation}
The It$\hat{\text{o}}$ reminder term $I_9$ is estimated as follows
\begin{equation}\label{I_9 s=3}
\begin{split}
|I_9| &=\frac{1}{2} \sum_{k=1}^{\infty} \int_{\mathbb{T}} |\partial^3_{x} F_k|^2 dx=\frac{1}{2} \sum_{k=1}^{\infty} \| \partial^3_{x} F_k \|^2_{L^2} \le c \sum_{k=1}^{\infty} \| \partial^3_{x} F_k \|^2_{L^\infty} \le c \sum_{k=1}^{\infty} {\alpha_k}^2 \le C.
\end{split}
\end{equation}
Summing up the estimates \eqref{I_1 s=3}-\eqref{I_9 s=3} and by choosing properly the Young constant $\delta$ we get
\begin{equation}
\begin{split}
& \dfrac{1}{2} \text{d} ( \| \partial^4_x \psi \|^2_{L^2} + \| \partial^3_x u \|^2_{L^2}) +c \int_{\mathbb{T}} e^{(\alpha-1) \psi } | \partial^4_x u |^2 dx \\ & \le \bigg( \| \partial^4_x \psi \|^2_{L^2}+ \| \partial^3_x u \|^2_{L^2} \bigg) \bigg[ 1+ \| \partial_x \psi \|^2_{L^\infty} \| \frac{1}{\rho^{\frac{\alpha-1}{2}}}\|^2_{L^\infty} + \| \partial^3_x \psi \|^2_{L^2}+ \| \partial_x u \|_{L^\infty} \\ & + \| u \|^2_{L^\infty} \| \frac{1}{\rho^{\frac{\alpha-1}{2}}}\|^2_{L^\infty}+ \| \rho^\alpha \|^2_{L^\infty} \| \partial_{xx} ( \frac{1}{\rho}) \|^2_{L^2} \| \frac{1}{\rho^{\frac{\alpha-1}{2}}}\|^2_{L^\infty} + \| \partial_x \rho^\alpha \|^2_{L^\infty} \| \frac{1}{\rho^{\frac{\alpha+1}{2}}}\|^2_{L^\infty} \\ & + \| \rho^\alpha \|^2_{L^\infty} \| \partial_x (\frac{1}{\rho}) \|^2_{L^\infty} \| \frac{1}{\rho^{\frac{\alpha-1}{2}}}\|^2_{L^\infty} + \| \frac{1}{\rho^{\frac{\alpha+1}{2}}}\|^2_{L^\infty} \| \partial_x \rho^\alpha \|^2_{L^\infty} \bigg] \\ & 
+ \| \partial_x u \|^2_{L^\infty} \| \partial_{xx} u \|^2_{L^2}\| \frac{1}{\rho^{\frac{\alpha-1}{2}}}\|^2_{L^\infty} + \| \frac{1}{\rho^{\frac{\alpha+1}{2}}}\|^2_{L^\infty} \| \partial^3_x \rho^\gamma \|^2_{L^2} \\ & + \| \partial_{xx} \rho^\gamma \|^2_{L^\infty} \| \partial_x ( \frac{1}{\rho} ) \|^2_{L^2} \| \frac{1}{\rho^{\frac{\alpha-1}{2}}}\|^2_{L^\infty} + \| \partial_x \rho^\gamma \|^2_{L^\infty} \| \partial_{xx} (\frac{1}{\rho}) \|^2_{L^2} \| \frac{1}{\rho^{\frac{\alpha-1}{2}}}\|^2_{L^\infty} \\ & + \| \frac{1}{\rho^{\frac{\alpha-1}{2}}}\|^2_{L^\infty} \| \partial_{xx} \rho^\alpha \|^2_{L^\infty} \| \partial_{xx} u \|^2_{L^2}+ \| \partial_x \rho^\alpha \|^2_{L^\infty}\| \partial_{xx} u \|^2_{L^2}  \| \partial_x (\frac{1}{\rho}) \|^2_{L^2} \| \frac{1}{\rho^{\frac{\alpha-1}{2}}}\|^2_{L^\infty} \\ & + \| \frac{1}{\rho^{\frac{\alpha+1}{2}}}\|^2_{L^\infty} \| \partial^3_x \rho^\alpha \|^2_{L^2} \| \partial_x u \|^2_{L^\infty} +\| \partial_{xx} \rho^\alpha \|^2_{L^2} \| \partial_x ( \frac{1}{\rho}) \|^2_{L^\infty} \| \partial_x u \|^2_{L^\infty} \| \frac{1}{\rho^{\frac{\alpha-1}{2}}}\|^2_{L^\infty} \\ & +\| \frac{1}{\rho^{\frac{\alpha+1}{2}}}\|^2_{L^\infty} \| \partial_{xx} \rho^\alpha \|^2_{L^\infty} \| \partial_{xx} u \|^2_{L^2}  + \| \partial_x \rho^\alpha \|^2_{L^\infty} \| \partial_{xx} ( \frac{1}{\rho}) \|^2_{L^2} \| \partial_x u \|^2_{L^\infty} \| \frac{1}{\rho^{\frac{\alpha-1}{2}}}\|^2_{L^\infty} \\ & + \| \partial_x \rho^\alpha \|^2_{L^\infty} \| \partial_x (\frac{1}{\rho}) \|^2_{L^\infty} \| \partial_{xx} u \|^2_{L^2}\| \frac{1}{\rho^{\frac{\alpha-1}{2}}}\|^2_{L^\infty}  + \| \partial_{xx} \psi \|^2_{L^2} \| \partial_{xx} \psi \|^2_{L^2}\| \frac{1}{\rho^{\frac{\alpha-1}{2}}}\|^2_{L^\infty} \\ & + \| \partial^3_x \psi \|^2_{L^2} \| \partial^3_x \psi \|^2_{L^2} \| \frac{1}{\rho^{\frac{\alpha-1}{2}}}\|^2_{L^\infty}+ I_8  = \bigg( \| \partial^4_x \psi \|^2_{L^2} + \| \partial^3_x u \|^2_{L^2} \bigg) a^{(3)}(t) + b^{(3)}(t) + I_8.
\end{split}
\end{equation}
We define $$[a^{(3)}]= \| \partial^3_x \psi \|^2_{L^2}+ \| \partial_{xx} u \|^2_{L^2}+ \| \frac{1}{\rho^{\frac{\alpha-1}{2}}} \|^2_{L^\infty}$$
and consequently 
\begin{equation}
\gamma^{(3)}_M= \sup\big\{ t\ge 0 \, : \sup_{s \in [0, t \land \tau]} [a^{(3)}] \le M \big\}=\inf\big\{ t\ge 0 \, :\sup_{s \in [0, t \land \tau]} [a^{(3)}] > M \big\}.
\end{equation}
$ [a^{(3)}] $ contains the terms with high order derivatives in $a^{(3)}$ and clearly by definition of $\gamma_M$ 
\begin{equation}\label{M to q }
 \mathbb{E}  \int_{0}^{\gamma^{(3)}_M} ( \| \partial^4_x \psi \|^2_{L^2}+ \| \partial^3_x u \|^2_{L^2})a^{(3)}(s)ds \le M^q \mathbb{E} \int_{0}^{\gamma^{(3)}_M} ( \| \partial^4_x \psi \|^2_{L^2}+ \| \partial^3_x u \|^2_{L^2}) ds 
 \end{equation}
for some $q \ge 1. $ \\
Hence integrating in time and taking the $\sup$ in $[0,  T \land \tau \land \gamma^{(3)}_M],$ after having applied the expectation we get 
\begin{equation}
\begin{split}
& \mathbb{E} \big(\sup_{ 0 \le s \le T \land \tau \land \gamma^{(3)}_M}\dfrac{1}{2}( \| \partial^4_{x} \psi \|^2_{L^2}+ \| \partial^3_x u \|^2_{L^2}) + \int_{0}^{T \land \tau \land \gamma^{(3)}_M}\int_{\mathbb{T}} (e^{(\alpha-1) \psi }) | \partial^4_{x} u|^2 dx \\ & \le \mathbb{E} \big(\| \partial^4_{x} \psi_0)\|^2_{L^2}+ \| \partial^3_x u_0\|^2_{L^2}\big)+ c \mathbb{E} \int_{0}^{T \land \tau \land \gamma^{(3)}_M} \big(\| \partial^4_x \psi \|^2_{L^2}+ \| \partial^3_x u \|^2_{L^2}\big)a^{(3)}(t)dt \\ & +c\mathbb{E} \int_{0}^{T \land \tau \land \gamma^{(3)}_M}b^{(3)}(t)dt+ c\mathbb{E} \sup_{0 \le s \le T \land \tau \land \gamma^{(3)}_M} | \int_{0}^{s}I_8 |.
\end{split}
\end{equation}
The stochastic integral $I_8$ can be estimated as follows
\begin{equation*}
\begin{split}
\mathbb{E} \bigg[ \sup_ {0 \le s \le T \land \tau \land \gamma^{(3)}_M} \bigg| {\int_{0}^{s} \int_{\mathbb{T}} \partial^3_{x} \mathbb{F}(\rho,u)\partial^3_{x} udxdW \bigg| } \bigg] \lesssim \mathbb{E} \bigg[ \int_{0}^{T \land \tau \land \gamma^{(3)}_M} \sum_{k=1}^{\infty} \bigg| \int_{\mathbb{T}} \partial^3_{x} F_k \partial^3_{x} u dx \bigg|^2 dt \bigg]^{\frac{1}{2}}
\end{split}
\end{equation*}
and 
\begin{equation*}
\begin{split}
 & \bigg| \int_{\mathbb{T}} \partial^3_{x} F_k \partial^3_{x} u dx \bigg|^2 \le \| \partial^3_{x} F_k \|^2_{L^2} \| \partial^3_{x} u \|^2_{L^2} \le C \| \partial^3_{x} F_k \|^2_{L^\infty} \| \partial^3_{x} u \|^2_{L^2} \le C {\alpha_k}^2 \| \partial^3_{x} u \|^2_{L^2},
\end{split}
\end{equation*}
hence we have 
\begin{equation*}
\begin{split}
& \mathbb{E} \bigg[ \int_{0}^{T \land \tau \land \gamma^{(3)}_M} \sum_{k=1}^{\infty} \bigg| \int_{\mathbb{T}} \partial^3_{x} F_k \partial^3_{x} u dx \bigg|^2 dt \bigg]^{\frac{1}{2}} \le \mathbb{E} \bigg[ \int_{0}^{T \land \tau \land \gamma^{(3)}_M} \sum_{k=1}^{\infty}C {\alpha_k}^2 \| \partial^3_{x} u \|^2_{L^2} \bigg]^{\frac{1}{2}} \\ &  \le \mathbb{E} \bigg[ \int_{0}^{T \land \tau \land \gamma^{(3)}_M} \tilde{c} \| \partial^3_{x} u \|^2_{L^2} \bigg]^{\frac{1}{2}} \le \mathbb{E} \bigg[ 1+ \int_{0}^{T \land \tau \land \gamma^{(3)}_M} c \| \partial^3_x u \|^2_{L^2} \bigg],
\end{split}
\end{equation*}
which leads to the following inequality 
\begin{equation}
\begin{split}
& \mathbb{E} \big(\sup_{ 0 \le s \le T \land \tau \land \gamma^{(3)}_M}\dfrac{1}{2}( \| \partial^4_{x} \psi \|^2_{L^2}+ \| \partial^3_x u \|^2_{L^2}) + \int_{0}^{T \land \tau \land \gamma^{(3)}_M}\int_{\mathbb{T}} (e^{(\alpha-1) \psi }) | \partial^4_{x} u|^2 dx \\ & \le \mathbb{E} \big(\| \partial^4_{x} \psi_0\|^2_{L^2}+ \| \partial^3_x u_0\|^2_{L^2}\big)+ c \mathbb{E} \int_{0}^{T \land \tau \land \gamma^{(3)}_M} \big(\| \partial^4_x \psi \|^2_{L^2}+ \| \partial^3_x u \|^2_{L^2}\big)a^{(3)}(t)dt \\ & +c\mathbb{E} \int_{0}^{T \land \tau \land \gamma^{(3)}_M}b^{(3)}(t)dt
\end{split}
\end{equation}
and by recalling \eqref{M to q } we can make use of the standard Gronwall Lemma to infer \eqref{eq:s3}.  
With the same lines of arguments of the previous cases we infer 
\begin{equation}\label{markov M s=3}
\mathbb{P} \bigg( \gamma^{(3)}_M \le T \bigg) \le  \mathbb{P} \bigg(  \sup_{ t \in [0,{T \land \tau]}} [a^{(3)}] \ge M \bigg) \le \frac{1}{M} \mathbb{E} \bigg( \sup_{ t \in [0,{T \land \tau]}} [a^{(3)}]\bigg) \longrightarrow 0,
\end{equation}
hence $\lim_{M \to \infty} \gamma^{(3)}_M= \infty$ a.s. \\
To conclude,  we define 
\begin{equation}\label{min gamma M}
\gamma_M= \min_{s=1,2,3} \gamma^{(s)}_M
\end{equation}
and clearly $\gamma_M \rightarrow \infty$ a.s as $M \rightarrow \infty.$ This concludes the proof of Proposition \ref{Prop global s+1 s}.
\end{proof}
\subsection{Proof of Theorem \ref{Main Theorem global}}
In order to prove the existence of a global maximal strong pathwise solution to system \eqref{stoc quantum} we perform a similar analysis as discussed in \cite{Glatt-Holtz} for the incompressible Navier-Stokes equations in the 2D case with multiplicative noise.  \\
Let $\alpha \in [0,\frac{1}{2}]$,  assume there exists a unique maximal strong pathwise solution $(\rho,u,(\tau_n)_{n \in \mathbb{N}},\tau)$ as stated in Theorem \ref{Main Theorem local}. Our goal is to show that $\tau= \infty$ a.s. \\
Let $\tau_n$ be an increasing sequence of stopping times announcing $\tau,$ we recall that $\tau_n$ has the following form
$$ \tau_n= \tau^{\psi}_n \vee \tau^u_n$$
for
\begin{equation*}
\tau^{\psi}_n= \inf \{ t \in [0,\infty) | \; \| \psi(t) \|_{W^{2,\infty}} \ge n \},
\end{equation*}
and 
\begin{equation*}
\tau^{u}_n= \inf \{ t \in [0,\infty) | \; \| u(t) \|_{W^{2,\infty}} \ge n \}.
\end{equation*}
with the convention $\inf \emptyset= \infty.$ \\
We write
\begin{equation}
\{ \tau < \infty \}= \bigcup_{T=1}^{\infty}\{ \tau \le T \}= \bigcup_{T=1}^{\infty} \bigcap_{n=1}^{\infty} \{ \tau_n \le T \}
\end{equation}
By monotonicity of $\tau_n$ we have
\begin{equation}
\mathbb{P} \big( \bigcap_{n=1}^{\infty} \{ \tau_n \le T \} \big)= \lim_{N \rightarrow \infty}\mathbb{P} \big( \bigcap_{n=1}^{N} \{ \tau_n \le T \} \big)= \lim_{N \rightarrow \infty} \mathbb{P} \big( \tau_N \le T) \big).
\end{equation}
Hence in order to prove our claim it is enough to show that for any fixed $T< \infty$
\begin{equation}
\lim_{N \rightarrow \infty} \mathbb{P}( \tau_N \le T)= 0
\end{equation}
Following the analogy of the previous subsection, for $M > 0$ we define 
\begin{equation}
\gamma_M = \gamma^{(1)}_M \land \gamma^{(2)}_M \land \gamma^{(3)}_M \land 2T
\end{equation}
Then we have
\begin{equation}\label{prob int}
\mathbb{P}( \tau_N \le T) \le \mathbb{P}( \{ \tau_N \le T \} \cap \{ \gamma_M > T \})+ \mathbb{P}( \gamma_M \le T).
\end{equation}
We can estimate the first term in the right hand side of \eqref{prob int} as follows
\begin{equation}
\begin{split}
& \mathbb{P}( \{ \tau_N \le T \} \cap \{ \gamma_M > T \}) \\ &  \le \mathbb{P} \big( \big\{ \sup_{t \in [0, \tau_N \land T]} ( \| \psi \|^2_{H^4}+ \| u \|^2_{H^3} ) + \\ & \int_{0}^{\tau_N \land T} \int_{\mathbb{T}}e^{(\alpha-1) \psi}( | \partial_x u |^2 + | \partial_{xx} u |^2 + | \partial^3_x u |^2+ | \partial^4_x u|^2)ds \ge N \big\} \cap \{ \gamma_M > T \} \big) \\ & \le \mathbb{P} \big( \big\{ \sup_{t \in [0, \tau_N \land \gamma_M]} ( \| \psi \|^2_{H^4}+ \| u \|^2_{H^3} ) \\ & + \int_{0}^{\tau_N \land \gamma_M} \int_{\mathbb{T}}e^{(\alpha-1) \psi}( | \partial_x u |^2 + | \partial_{xx} u |^2 + | \partial^3_x u |^2+ | \partial^4_x u|^2)ds \ge N \big\} \big)
\end{split}
\end{equation}
and by using Markov inequality we infer 
\begin{equation}
\begin{split}
\mathbb{P}( \tau_N < T) &\le \frac{C}{N} \mathbb{E} \bigg( \sup_{t \in [0, \tau_N \land \gamma_M]} ( \| \psi \|^2_{H^4}+ \| u \|^2_{H^3} ) \\&+ \int_{0}^{\tau_N \land \gamma_M} \int_{\mathbb{T}}e^{(\alpha-1) \psi}( | \partial_x u |^2 + | \partial_{xx} u |^2 + | \partial^3_x u |^2+ | \partial^4_x u|^2)dxds \bigg) \\ & + \mathbb{P}(\gamma_M \le T)
\end{split}
\end{equation}
Hence from Proposition \ref{Prop global s+1 s} we get that for any fixed $M >0$ 
\begin{equation}\label{stima gamma M}
\lim_{N \rightarrow \infty} \mathbb{P}( \tau_N < T) \le \mathbb{P}(\gamma_M \le T).
\end{equation}
Finally we send $M \rightarrow \infty$ and we use the estimates \eqref{markov M s=1}, \eqref{markov M s=2} and \eqref{markov M s=3} to complete the proof.
\section{Proof of Theorem \ref{Main Theorem local}}
This section is dedicated to the proof of Theorem \ref{Main Theorem local}.
We mimic the analysis developed in the monograph \cite{Feir} to prove the local well posedness for the compressible Navier-Stokes equations.\\
The outline of the proof is the following: we consider an approximating system of \eqref{stoc quantum} by introducing suitable cut-off operator applied to the $W^{2,\infty}$ norm of the solution.  We prove the existence of a strong martingale solution for the approximating system, see Definition \ref{def strong mart},  and a pathwise uniqueness estimate. In order to prove the existence of an approximating solution we make use of a Galerkin approximation technique and a stochastic compactness argument based on an appropriate use of Skorokhod representation theorem and Prokhorov theorem,  see Theorem \ref{Skor} and Theorem \ref{Prok} in the Appendix. Finally we make use of the Gyongy-Krylov method in order to deduce the existence of a strong pathwise solution as stated in Theorem \ref{Main Theorem local},  see Theorem \ref{Gyong} and Section 2.10 in \cite{Feir}.  Concerning the initial conditions,  during the proof,  we assume an additional integrability condition with respect to the $\omega$ variable and then we remove it by considering an appropriate decomposition of the probability space $\Omega.$ \\ \\
In the deterministic setting, local existence results are usually proved by means of high order a priori estimates, however this method does not apply to our problem due to the loss of regularity with respect to the time variable in the stochastic setting.  We refer the reader to \cite{Cho 2004} for the optimal class of strong solutions and related blow up criteria in the deterministic setting of the compressible Navier-Stokes equations with vacuum and to \cite{Don-Pes} for similar results considering also the dependence of the external forces on the solution.
\subsection{The approximating system}
As in the previous section,  we rewrite system \eqref{stoc quantum} in the variables $(\psi,u)$ where $\psi=\log \rho.$
\\
\begin{equation}\label{stoch quantum rho u}
\text{d}\psi+[u \partial_x \psi+ \partial_x u]\text{d}t=0 
\end{equation}
\begin{equation}\label{stoch momentum u}
\begin{split}
& \text{d}u+[u\partial_x u + \gamma e^{(\gamma-1)\psi}\partial_x \psi]dt=[e^{\alpha-1)\psi}\partial_{xx}u+\alpha e^{(\alpha-1)\psi}\partial_x \psi \partial_x u]dt \\ & + [\partial_{xxx}\psi+ \partial_x\psi \partial_{xx} \psi]dt+ \mathbb{F}(\rho, u)\text{d}W
\end{split}
\end{equation}
\\
and for any $x \in \mathbb{T}, \; t \in [0,T]$ we define the following approximate system of \eqref{stoch quantum rho u}-\eqref{stoch momentum u}
\begin{equation}\label{approx R stoch quantum rho u}
\text{d}\psi+\varphi_R( \| u \|_{W_x^{2,\infty}})[ u\partial_x \psi]dt+ \partial_x u dt=0 
\end{equation}
\begin{equation}\label{approx R stoch quantum u}
\begin{split}
& \text{d}u-\partial_{xxx}\psi dt+\varphi_R( \| u\|_{W_x^{2,\infty}})[u\partial_x u ]dt+ \varphi_R( \| \psi\|_{W_x^{2,\infty}}) [\gamma e^{(\gamma-1)\psi} \partial_x \psi]dt = \\ & \varphi_R( \| \psi\|_{W_x^{2,\infty}}) [e^{(\alpha-1)\psi} \partial_{xx}u ]dt +\varphi_R( \| \psi\|_{W_x^{2,\infty}})[\alpha e^{(\alpha-1)\psi}\partial_x \psi \partial_x u ]dt \\ & + \varphi_R( \| \psi\|_{W_x^{2,\infty}})[\partial_x \psi \partial_{xx}\psi]dt+\varphi_R( \| u\|_{W_x^{2,\infty}}) \mathbb{F}(\rho,u)\text{d}W
\end{split}
\end{equation}
\\
with initial conditions 
\begin{equation} \label{Approx C.I}
\rho(0)=\rho_0, \quad u(0)=u_0,
\end{equation}
where $\varphi_R: [0,\infty) \rightarrow [0,1]$ are smooth functions satisfying 
\begin{equation*}
\varphi_R(y)=\begin{cases} 1, \quad 0 \le y \le R \\
0, \quad y \ge R+1.
\end{cases}
\end{equation*}
The cut-off operators $\varphi_R$ are introduced in order to deal with the nonlinearities of system \eqref{stoch quantum rho u}-\eqref{stoch momentum u}.  Indeed the truncated terms turn out to be globally Lipschitz continuous and the approximating system \eqref{approx R stoch quantum rho u}-\eqref{approx R stoch quantum u} admits global in time solutions.  Moreover $\varphi_R$ is applied to the $W^{2,\infty}$ norm of the solution, hence condition \eqref{limit cond} in Definition \ref{Def2} follows by construction.  \\
\\
We start introducing the definition of strong martingale solutions,  i.e.  solutions which are strong in PDEs and weak in probability sense.  For such solutions,  the Wiener process $W$ and the stochastic basis with right continuous filtration $(\Omega, \mathfrak{F},(\mathfrak{F}_t)_{t \ge 0},\mathbb{P})$ are not a priori given,  they are indeed part of the solution itself.
\begin{definition} (Strong martingale solution).  \label{def strong mart}
Let $\Lambda$ be a Borel probability measure on $H^{s+1}(\mathbb{T}) \times H^s(\mathbb{T}), \; s \in \mathbb{N}.$ A multiplet $$((\Omega, \mathfrak{F}, (\mathfrak{F}_t)_{t \ge 0}, \mathbb{P}),\psi,u, W)$$ is called a strong martingale solution to the approximated system \eqref{approx R stoch quantum rho u}-\eqref{Approx C.I} with the initial law $\Lambda$,  provided;
\begin{itemize}
\item[(1)]
$(\Omega, \mathfrak{F},(\mathfrak{F}_t)_{t \ge 0},\mathbb{P}), $ is a stochastic basis with a complete right-continuous filtration;
\item[(2)]
$W$ is an $(\mathfrak{F}_t)$-cylindrical Wiener process;
\item[(3)]
$\psi$ is a $H^{s+1}(\mathbb{T})$-valued $(\mathfrak{F}_t)$-progressively measurable stochastic process satisfying
 $$  \psi \in L^2(\Omega; C([0,T];H^{s+1}(\mathbb{T})))\quad \rho >0 \  \mathbb{P}-\text{a.s.}; $$
 \item[(4)]
 the velocity $u$ is a $H^s(\mathbb{T})$-valued $(\mathfrak{F}_t)$-progressively measurable stochastic process satisfying
 $$ u \in L^2(\Omega; C([0,T];H^s(\mathbb{T})) \cap L^2(\Omega; L^2(0,T;H^{s+1}(\mathbb{T})));$$
 \item[(5)]
 there exists an $\mathfrak{F}_0$-measurable random variable $[\rho_0,u_0]$ such that $\Lambda= \mathfrak{L}[\rho_0,u_0];$
\item[(6)]
the equations
$$
\psi(t)=\psi_0- \int_{0}^{t}\varphi_R( \| u\|_{W_x^{2,\infty }})[ u\partial_x \psi] ds-\int_{0}^{t}\partial_x u\text{d}s$$
\begin{equation*}
\begin{split}
u(t)&=u_0+\int_{0}^{t} \partial_{xxx}\psi ds-\int_{0}^{t}\varphi_R( \| u\|_{W_x^{2,\infty}})[u\partial_x u]ds \\ & - \int_{0}^{t}\varphi_R( \| \psi\|_{W_x^{2,\infty}})[\gamma e^{(\gamma-1)\psi} \partial_x \psi]ds +\int_{0}^{t}\varphi_R( \| \psi\|_{W_x^{2,\infty}})[e^{(\alpha-1)\psi} \partial_{xx}u ]ds \\ & +\int_{0}^{t}\varphi_R( \| \psi\|_{W_x^{2,\infty}})[\alpha e^{(\alpha-1)\psi}\partial_x \psi \partial_x u ]ds +\int_{0}^{t}\varphi_R( \| \psi \|_{W_x^{2,\infty}})[ \partial_x \psi \partial_{xx}\psi]ds \\ & +\int_{0}^{t} \varphi_R( \| u\|_{W_x^{2,\infty}}) \mathbb{F}(\rho,u)\text{d}W
\end{split}
\end{equation*}
holds for all $t \in [0,T] \ \mathbb{P}-a,s,;$
\end{itemize}
\end{definition}
\noindent
Now we define the strong pathwise solutions. These solutions are strong not only in PDEs but also in probability.  This means that the underlying probability space with right continuous filtration and the Brownian motions are a priori assigned.  This class of solutions is the one for which we aim to prove our results.
\begin{definition}(Strong pathwise solution). \label{strong path sol}
Let $(\Omega, \mathfrak{F},(\mathfrak{F}_t)_{t \ge 0},\mathbb{P}), $ be a stochastic basis with a complete right-continuous filtration and let $W$ be an $(\mathfrak{F}_t)$-cylindrical Wiener process.
Then $(\psi,u)$ is called a strong pathwise solution to the approximate system \eqref{approx R stoch quantum rho u}-\eqref{Approx C.I} with initial conditions $(\psi_0,u_0)$,  provided
\begin{itemize}
\item[(1)]
 $\psi$ is a $H^{s+1}(\mathbb{T})$-valued $(\mathfrak{F}_t)$-progressively measurable stochastic process satisfying
 $$  \psi \in L^2(\Omega; C([0,T];H^{s+1}(\mathbb{T}))), \quad \rho >0 \  \mathbb{P}-\text{a.s.}; $$
 \item[(2)]
 the velocity $u$ is a $H^s(\mathbb{T})$-valued $(\mathfrak{F}_t)$-progressively measurable stochastic process satisfying
 $$ u \in L^2(\Omega; C([0,T];H^s(\mathbb{T})) \cap L^2(\Omega; L^2(0,T;H^{s+1}(\mathbb{T})));$$
\item[(3)]
the equations
$$
\psi(t)=\psi_0- \int_{0}^{t}\varphi_R( \| u\|_{W_x^{2,\infty }})[ u\partial_x \psi] ds-\int_{0}^{t}\partial_x u\text{d}s$$
\begin{equation*}
\begin{split}
u(t)&=u_0+\int_{0}^{t} \partial_{xxx}\psi ds-\int_{0}^{t}\varphi_R( \| u\|_{W_x^{2,\infty}})[u\partial_x u]ds \\ & -\int_{0}^{t}\varphi_R( \| \psi\|_{W_x^{2,\infty}})[\gamma e^{(\gamma-1)\psi} \partial_x \psi]ds +\int_{0}^{t}\varphi_R( \| \psi\|_{W_x^{2,\infty}})[e^{(\alpha-1)\psi} \partial_{xx}u ]ds \\ & +\int_{0}^{t}\varphi_R( \| \psi\|_{W_x^{2,\infty}})[\alpha e^{(\alpha-1)\psi}\partial_x \psi \partial_x u ]ds +\int_{0}^{t}\varphi_R( \| \psi\|_{W_x^{2,\infty}})[ \partial_x \psi \partial_{xx}\psi]ds +\\ & \int_{0}^{t} \varphi_R( \| u\|_{W_x^{2,\infty}}) \mathbb{F}(\rho,u)\text{d}W
\end{split}
\end{equation*}
holds for all $t \in [0,T] \ \mathbb{P}$-a.s. ;
\end{itemize}
\end{definition}
\noindent
Our main result for the approximating system \eqref{approx R stoch quantum rho u}-\eqref{Approx C.I} is the following
\begin{theorem}\label{second thm}
Let the coefficients $G_k$ satisfy the hypothesis \eqref{G1}-\eqref{G2} and let the initial datum $$(\psi_0,u_0) \in L^p(\Omega, \mathfrak{F}_0, \mathbb{P}; H^{s+1}(\mathbb{T}) \times H^s(\mathbb{T})), $$
for all $1 \le p < \infty$ and for some $s > \frac{5}{2} .$ In addition,  suppose that \eqref{C.I STRONG} holds $\mathbb{P}$-a.s. for $s > \frac{5}{2} .$ 
Then:
\begin{itemize}
\item[(1)]
There exists a strong martingale solution to the problem \eqref{approx R stoch quantum rho u}-\eqref{Approx C.I} in the sense of Definition \ref{def strong mart} with initial law $\Lambda= \mathfrak{L}[(\psi_0,u_0)].$ Moreover,  there exists a deterministic constant $\bar \psi_R$ such that $$\psi(\cdot, t) \ge \overline{\psi}_R \quad \mathbb{P}-\text{a.s.} \ \text{for all} \ t \in [0,T]$$
and 
\begin{equation}
\mathbb{E} \bigg[ \sup_{t \in [0,T]} \|( \psi(t),u(t) \|^2_{H^{s+1}(\mathbb{T}) \times H^s(\mathbb{T})} + \int_{0}^{T} \| u \|^2_{H^{s+1}(\mathbb{T})} dt \bigg]^p \le c(R, \psi_0,u_0, p) < \infty
\end{equation}
for all $1 \le p < \infty$.
\item[(2)]
If $s > \frac{7}{2},$ then pathwise uniqueness holds.  Specifically,  if $(\psi^1,u^1),  (\psi^2,u^2)$ are two strong solutions to \eqref{approx R stoch quantum rho u} defined on the same stochastic basis with the same Wiener process $W$ and $$\mathbb{P}(\psi_0^1=\psi_0^2,  u_0^1=u_0^2)=1,  $$
then $$\mathbb{P}(\psi^1(t)=\psi^2(t),  u^1(t)=u^2(t), \  \text{for all} \ t \in [0,T])=1. $$
\end{itemize}
Consequently,  there exists a unique strong pathwise solution to \eqref{approx R stoch quantum rho u}-\eqref{Approx C.I} in the sense of Definition \ref{strong path sol}.
\end{theorem}
\subsubsection{Galerkin approximation}
In order to determine the existence of a strong martingale solution to system 	\eqref{approx R stoch quantum rho u}-\eqref{Approx C.I} we use a Galerkin approximation technique.  Let $H_m$ be the space of the trigonometric polynomials of order $m$,  namely $$ H_m= \bigg\{ v= \sum_{m,\max_{j=1,2,3} |m_j| \le m}^{}[ a_m \cos(\pi m \cdot x)+b_m \sin (\pi m\cdot x)] \; | \; a_m, b_m \in \mathbb{R}\bigg\}$$ and let $(w_m)_{m\in \mathbb{N}}$ be an orthonormal basis of $H_m,$ we consider ${\Pi}_m$ the associated $L^2$- orthogonal projection $${\Pi}_m : L^2 (\mathbb{T}) \longrightarrow H_m.$$
\\
For a given $f \in L^1(\mathbb{T}),$ the projection ${\Pi}_m[f]$ represents the $m$th cubic partial sum of the Fourier series of $f$ and we recall the following properties:
\begin{equation}
\| {\Pi}_m [f] \|_{W^{k,p}} \le c(k,p) \| f \|_{W^{k,p}},
\end{equation}
\begin{equation}
{\Pi}_m [f] \longrightarrow f \ \text{in} \ W^{k,p}(\mathbb{T}) \ \text{as} \ n\rightarrow \infty \  \text{for all} \ 1< p < \infty,
\end{equation}
whenever $ f\in W^{k,p}.$
\\
We set $\psi_m={\Pi}_m [\psi], \, u_m={\Pi}_m [u]$ and we consider the projected system of \eqref{approx R stoch quantum rho u}-\eqref{Approx C.I} on $H_m$
\begin{equation} \label{Gal psi}
\text{d}\langle \psi_m, w_i \rangle+ \varphi_R( \| u_m\|_{W_x^{2,\infty}})\langle u_m\partial_x \psi_m,w_i \rangle dt+ \langle \partial_x u_m,w_i \rangle dt=0 
\end{equation}
\begin{equation} \label{Gal u}
\begin{split}
&\text{d}\langle u_m,  w_i \rangle-\langle \partial_{xxx}\psi_m,w_i \rangle dt  +\varphi_R( \| u_m\|_{W_x^{2,\infty}}) \langle u_m \partial_x u_m,w_i \rangle dt \\ & + \varphi_R( \| \psi_m\|_{W_x^{2,\infty}})\langle \gamma e^{(\gamma-1)\psi_m} \partial_x \psi_m ,w_i \rangle dt=  \varphi_R( \| \psi_m\|_{W_x^{2,\infty}})\langle e^{(\alpha-1)\psi_m} \partial_{xx}u_m,w_i  \rangle dt \\ & + \varphi_R( \| \psi_m\|_{W_x^{2,\infty}})\langle \alpha e^{(\alpha-1)\psi_m}\partial_x \psi_m \partial_x u_m ,w_i \rangle dt + \varphi_R( \| \psi_m\|_{W_x^{2,\infty}})\langle \partial_x \psi_m \partial_{xx}\psi_m,w_i \rangle dt \\ & +\varphi_R( \| u_m\|_{W_x^{2,\infty}}) \langle \mathbb{F}(\rho_m,u_m),w_i \rangle \text{d}W.
\end{split}
\end{equation}
Since all norms defined on the finite dimensional space $H_m$ are equivalent,  then solutions of \eqref{Gal psi}-\eqref{Gal u} can be obtained by means of a Banach fixed point argument.  To this purpose we define $\mathfrak{B}= L^2(\Omega; C([0,T^*];H_m))$ and the map 
\begin{equation*}
\mathcal{F}: \mathfrak{B}\times\mathfrak{B} \longrightarrow \mathfrak{B}\times \mathfrak{B}
\end{equation*}
with $\mathcal{F}(\psi_m,u_m)=(\mathcal{F}_1(\psi_m,u_m), \mathcal{F}_2(\psi_m,u_m))$ satisfying 
\begin{equation}\label{F1}
\begin{split}
\langle \mathcal{F}_1(\psi_m,u_m)(\tau), w_i \rangle &= \langle \psi(0),w_i \rangle-\int_{0}^{\tau} \varphi_R( \| u_m\|_{W_x^{2,\infty}})\langle[ u_m\partial_x \psi_m,w_i] \rangle dt \\ & -\int_{0}^{\tau} \langle \partial_x u_m,w_i \rangle dt,
\end{split}
\end{equation}
\begin{equation}\label{F2}
\begin{split}
&\langle \mathcal{F}_2(\psi_m,u_m)(\tau),  w_i \rangle= \langle u(0),w_i \rangle+\int_{0}^{\tau} \langle \partial_{xxx}\psi_m,w_i \rangle dt \\ & -\varphi_R( \| u_m\|_{W_x^{2,\infty}}) \int_{0}^{\tau} \langle u_m \partial_x u_m,w_i \rangle dt -\varphi_R( \| \psi_m\|_{W_x^{2,\infty}})\int_{0}^{\tau} \langle \gamma e^{(\gamma-1)\psi_m} \partial_x \psi_m ,w_i \rangle dt \\ & + \varphi_R( \| \psi_m\|_{W_x^{2,\infty}})\int_{0}^{\tau} \langle [e^{(\alpha-1)\psi_m} \partial_{xx}u_m,w_i ] \rangle +\langle [\alpha e^{(\alpha-1)\psi_m}\partial_x \psi_m \partial_x u_m ,w_i \rangle]dt \\ & + \varphi_R( \| \psi_m\|_{W_x^{2,\infty}}) \int_{0}^{\tau} \langle [\partial_x \psi_m \partial_{xx}\psi_m,w_i]\rangle dt +\varphi_R( \| u_m\|_{W_x^{2,\infty}})\int_{0}^{\tau} \langle \mathbb{F}(\rho_m,u_m),w_i \rangle \text{d}W.
\end{split}
\end{equation}
For $(\psi_1,u_1)$ and $(\psi_2,u_2)$ solutions of \eqref{Gal psi}-\eqref{Gal u} and by denoting by $\mathcal{F}_{det}$ and $\mathcal{F}_{sto}$ the deterministic and stochastic integrals respectively,  we have $$\| \mathcal{F}(\psi_1,u_1)-\mathcal{F}(\psi_2,u_2) \|^2_B \le \|\mathcal{F}_{det}(\psi_1,u_1)-\mathcal{F}_{det}(\psi_2,u_2)\|^2_B+ \|\mathcal{F}_{sto}(\psi_1,u_1)-\mathcal{F}_{sto}(\psi_2,u_2)\|^2_B.$$
By the equivalence of the norms in $H_m$ we have that
\begin{equation}
\| \mathcal{F}_{det}(\psi_1,u_1)-\mathcal{F}_{det}(\psi_2,u_2) \|^2_B \le T^*C(m,R,T) \| (\psi_1,u_1)-(\psi_2,u_2)\|^2_B,
\end{equation}
while concerning the stochastic part we set similarly to \cite{Feir}, see Chapter 5,  Section 5.2.1 $J_R(w)=\varphi_{R+1}(\| w \|_{W^{2,\infty}})w$ and by using Burkholder-Davis-Gundy's inequality we have
\begin{equation}
\begin{split}
& \| \mathcal{F}_{sto}(\psi_1,u_1)-\mathcal{F}_{sto}(\psi_2,u_2)\|^2_B \\ & = \mathbb{E} \ \sup_{[0,T^*]} \bigg{ \| } \int_{0}^{t} (\varphi_R( \| u_1\|_{W_x^{2,\infty}}) \mathbb{F}(\rho_1,u_1)-(\varphi_R( \| u_2\|_{W_x^{2,\infty}}) \mathbb{F}(\rho_2,u_2) \text{d}W \bigg{ \| }^2_{H_m} \\ & \le \mathbb{E} \int_{0}^{T^*} \sum_{k=1}^{\infty} \| \varphi_R( \| u_1\|_{W_x^{2,\infty}}) F_k(\rho_1,J_R(u_1))-\varphi_R( \| u_2\|_{W_x^{2,\infty}}) F_k(\rho_2,J_R(u_2)) \text{d}W \| ^2_{H_m}dt \\ & \le \mathbb{E} \int_{0}^{T^*} | \varphi_R( \| u_1\|_{W_x^{2,\infty}})-\varphi_R( \| u_2\|_{W_x^{2,\infty}})|^2\sum_{k=1}^{\infty} \| F_k(\rho_1,J_R(u_1)) \|^2_{H_m} dt \\ & + \mathbb{E} \int_{0}^{T^*} \varphi_R( \| u_2\|_{W_x^{2,\infty}})^2\sum_{k=1}^{\infty} \| F_k(\rho_1,J_R(u_1))-F_k(\rho_2,J_R(u_2)) \|^2_{H_m} dt
\end{split}
\end{equation}
and so we get
\begin{equation}
\begin{split}
& \| \mathcal{F}_{sto}(\psi_1,u_1)-\mathcal{F}_{sto}(\psi_2,u_2)\|^2_B \le \mathbb{E} \| u_1-u_2 \|^2_{W^{2,\infty}}+ \mathbb{E} \int_{0}^{T^*} \| e^{\psi_1}-e^{\psi_2} \|^2_{L^2}ds \\ & +\mathbb{E} \int_{0}^{T^*} \| J_R(u_1)-J_R(u_2)\|^2_{L^2}ds \le T^*C(m,R,T) \| (\psi_1,u_1)-(\psi_2,u_2)\|^2_B.
\end{split}
\end{equation}
Hence by choosing the constant $C(m,R,T)$ in such a way $T^*C(m,R,T)< 1$ then $\mathcal{F}$ is a contraction for a deterministic small time $T^*>0.$  By means of Banach Fixed Point Theorem we get the existence of a unique solution in the interval $[0,T^*].$ In order to get the existence of a solution in the whole interval we decompose $[0,T]$ into small subintervals and we glue up the corresponding solutions together.
\subsubsection{Uniform estimates}
In this subsection we derive some uniform in $m$ high order estimates needed to apply compactness argument and to send $m\rightarrow \infty$.  Let $\beta$ be a multi-index such that $|\beta| \le s,$ we apply the operator $\partial^{\beta+1}_x$ to \eqref{Gal psi} and by recalling the invariance of $H_m$ with respect to the spatial derivative we get 
\begin{equation}
\text{d}\langle \partial_x^{\beta+1}\psi_m, w_i \rangle+ \varphi_R( \| u_m\|_{W_x^{2,\infty}})\langle[ \partial_x^{\beta+1}(u_m\partial_x \psi_m),w_i] \rangle dt+ \langle \partial_x^{\beta+2} u_m,w_i \rangle dt=0.
\end{equation}
Multiplying by $\partial_x^{\beta+1}\psi_m$ after integrating in space we have
\begin{equation}\label{alpha+1 psi}
\begin{split}
& \dfrac{1}{2}\text{d} \int_{} | \partial_x^{\beta+1} \psi_m|^2dx+\varphi_R( \| u_m\|_{W_x^{2,\infty}})\int_{}[ \partial_x^{\beta+1}(u_m\partial_x \psi_m)\partial_x^{\beta+1}\psi_m] dxdt \\ & + \int_{} \partial_x^{\beta+2} u_m \partial_x^{\beta+1}\psi_m dxdt=0.
\end{split}
\end{equation}
Similarly we differentiate \eqref{Gal u} with respect to $\partial_x^\beta$
\begin{equation}
\begin{split}
&\text{d}\langle \partial_x^\beta u_m,  w_i \rangle- \langle \partial_x^{\beta+3}\psi_m,w_i \rangle dt +\varphi_R( \| u_m\|_{W_x^{2,\infty}}) \langle \partial_x^\beta(u_m \partial_x u_m)\rangle dt \\ & +\varphi_R( \| \psi_m\|_{W_x^{2,\infty}}) \langle \partial_x^\beta(\gamma e^{(\gamma-1)\psi_m} \partial_x \psi_m) ,w_i \rangle dt=\varphi_R( \| \psi_m\|_{W_x^{2,\infty}}) \langle [\partial_x^\beta(e^{(\alpha-1)\psi_m} \partial_{xx}u_m),w_i ] \rangle dt \\ &+\varphi_R( \| \psi_m\|_{W_x^{2,\infty}})\langle [\partial_x^\beta(\alpha e^{(\alpha-1)\psi_m}\partial_x \psi_m \partial_x u_m) ,w_i \rangle]dt \\ & + \varphi_R( \| \psi_m\|_{W_x^{2,\infty}})\langle [\partial_x^\beta(\partial_x \psi_m \partial_{xx}\psi_m),w_i]\rangle dt +\varphi_R( \| u_m\|_{W_x^{2,\infty}}) \langle \partial_x^\beta\mathbb{F}(\rho_m,u_m),w_i \rangle \text{d}W.
\end{split}
\end{equation}
We apply It$\hat{\text{o}}$ formula to the function $F(D^m)=\int_{} | \partial_x^\beta u_m|^2dx$,  where $D^m=(d_1^m,....., d_n^m)$ are the coefficients in the expansion $u_m=\sum_{i=1}^{m} d_i^m w_i$
\begin{equation}\label{alpha u}
\begin{split}
& \dfrac{1}{2}\text{d} \int_{} | \partial_x^\beta u_m |^2 dx -\int_{} \partial_x^{\beta+3}\psi_m \partial_x^\beta u_m dxdt +\varphi_R( \| u_m\|_{W_x^{2,\infty}}) \int_{} \partial_x^\beta(u_m \partial_x u_m)\partial_x^\beta u_m dxdt \\ & + \varphi_R( \| \psi_m\|_{W_x^{2,\infty}}) \int_{} \partial_x^\beta(\gamma e^{(\gamma-1)\psi_m} \partial_x \psi_m)\partial_x^\beta u_m dxdt= \\ & +  \varphi_R( \| \psi_m\|_{W_x^{2,\infty}})\int_{} [\partial_x^\beta(e^{(\alpha-1)\psi_m} \partial_{xx}u_m)\partial_x^\beta u_m dxdt \\ & +\varphi_R( \| \psi_m\|_{W_x^{2,\infty}}) \int_{} [\partial_x^\beta(\alpha e^{(\alpha-1)\psi_m}\partial_x \psi_m \partial_x u_m)\partial_x^\beta u_m] dxdt \\ & + \varphi_R( \| \psi_m\|_{W_x^{2,\infty}}) \int_{} [\partial_x^\beta(\partial_x \psi_m \partial_{xx}\psi_m)\partial_x^\beta u_m dxdt \\ & + \varphi_R( \| u_m\|_{W_x^{2,\infty}}) \int_{} \partial_x^\beta\mathbb{F}(\rho_m,u_m)\partial_x^\beta u_m dx\text{d}W \\ & +\dfrac{1}{2} \sum_{k=1}^{\infty} \int_{} \varphi_R( \| u_m\|_{W_x^{2,\infty}}) | \partial_x^\beta F_k(\rho_m,u_m) |^2 dxdt.
\end{split}
\end{equation}
Summing up \eqref{alpha+1 psi} and \eqref{alpha u} and integrating by parts we end up with the following equality
\begin{equation}\label{Uniform est m}
\begin{split}
& \dfrac{1}{2}\text{d} \int_{} \bigg( | \partial_x^{\beta+1} \psi_m|^2+| \partial_x^\beta u_m |^2 \bigg) dx= -\varphi_R( \| u_m\|_{W_x^{2,\infty}})\int_{}[ \partial_x^{\beta+1}(u_m\partial_x \psi_m)\partial_x^{\beta+1}\psi_m] dxdt \\ & -\varphi_R( \| u_m\|_{W_x^{2,\infty}}) \int_{} \partial_x^\beta(u_m \partial_x u_m)\partial_x^\beta u_m dxdt - \varphi_R( \| \psi_m\|_{W_x^{2,\infty}}) \int_{}\partial_x^\beta(\gamma e^{(\gamma-1)\psi_m} \partial_x \psi_m)\partial_x^\beta u_m dxdt \\ & + \varphi_R( \| \psi_m\|_{W_x^{2,\infty}}) \int_{} [\partial_x^\beta(e^{(\alpha-1)\psi_m} \partial_{xx}u_m)\partial_x^\beta u_m dxdt \\ & +\varphi_R( \| \psi_m\|_{W_x^{2,\infty}}) \int_{} [\partial_x^\beta(\alpha e^{(\alpha-1)\psi_m}\partial_x \psi_m \partial_x u_m)\partial_x^\beta u_m]dxdt \\ & + \varphi_R( \| \psi_m\|_{W_x^{2,\infty}}) \int_{} [\partial_x^\beta(\partial_x \psi_m \partial_{xx}\psi_m)\partial_x^\beta u_m dxdt +\varphi_R( \| u_m\|_{W_x^{2,\infty}}) \int_{} \partial_x^\beta \mathbb{F}(\rho_m,u_m)\partial_x^\beta u_m dx\text{d}W \\ & +\dfrac{1}{2} \sum_{k=1}^{\infty} \int_{} \varphi_R( \| u_m\|_{W_x^{2,\infty}}) | \partial_x^\beta F_k(\rho_m,u_m) |^2 dxdt=\sum_{j=1}^{8} I_j.
\end{split}
\end{equation}
Now we rewrite and estimate the integrals in \eqref{Uniform est m}.  In the following we make an extensive use of H\"{o}lder,  Sobolev and Young inequalities,  Moser type commutator estimates, (see Proposition \ref{Moser}), in order to deal with the high order derivative terms.  \\
\begin{equation}\label{1}
\begin{split}
|I_1|= | \;& \varphi_R( \| u_m\|_{W_x^{2,\infty}})\int_{}[ \partial_x^{\beta+1}(u_m\partial_x \psi_m)\partial_x^{\beta+1}\psi_m] dxdt|= \\ & |\varphi_R( \| u_m\|_{W_x^{2,\infty}})\int_{}[ \partial_x^{\beta+1}(u_m\partial_x \psi_m)-u_m\partial_x^{\beta+1} \partial_x \psi_m]\partial_x^{\beta+1}\psi_mdxdt \\ & +\varphi_R( \| u_m\|_{W_x^{2,\infty}})\int_{}u_m\partial_x^{\beta+1} \partial_x \psi_m\partial_x^{\beta+1}\psi_mdxdt|=I_{1,1}+I_{1,2} .
\end{split}
\end{equation}
The terms in \eqref{1} can be estimated as follows
\begin{equation}
\begin{split}
I_{1,1} &  \le C \varphi_R( \| u_m\|_{W_x^{2,\infty}}) ( \| u_m \|_{H^{s+1}} \| \psi_m \|_{W^{1,\infty}}+ \| u_m \|_{W^{1,\infty}} \| \psi_m \|_{H^{s+1}}) \| \psi_m \|_{H^{s+1}} \\ & \le C(R)( \| u_m\|_{H^{s+1}}^2 + \| \psi_m \|_{H^{s+1}}^2),
\end{split}
\end{equation}
and
\begin{equation}
\begin{split}
I_{1,2} & \le | \varphi_R( \| u_m\|_{W_x^{2,\infty}})\int_{}u_m \partial_x \partial_x^{\beta+1} \psi_m\partial_x^{\beta+1}\psi_mdx | \\ & =| \varphi_R( \| u_m\|_{W_x^{2,\infty}}) \int_{} u_m \partial_x \bigg( \dfrac{|\partial_x^{\beta+1}\psi_m|^2}{2} \bigg ) | \\ & \le C \varphi_R( \| u_m\|_{W_x^{2,\infty}}) \| \partial_x u_m \|_{L^\infty} \| \partial_x^{\beta+1} \psi_m \|_{L^2}^2  \le C(R) \| \psi_m \|_{H^{s+1}}^2.
\end{split}
\end{equation}
Concerning the second integral we have 
\begin{equation}
\begin{split}
|I_2|= | \; & \varphi_R( \| u_m\|_{W_x^{2,\infty}}) \int_{} [\partial_x^\beta(u_m \partial_x u_m)-u_m \partial_x^{\beta} \partial_x u_m]\partial_x^\beta u_m dx \\ & +\varphi_R( \| u_m\|_{W_x^{2,\infty}}) \int_{} u_m \partial_x \partial_x^{\beta} u_m \partial_x^{\beta}u_m dx|= I_{2,1}+I_{2,2}
\end{split}
\end{equation}
and we get 
\begin{equation}
\begin{split}
I_{2,1} & \le  C\varphi_R( \| u_m\|_{W_x^{2,\infty}}) ( \| u_m\|_{H^s} \| \partial_x u_m \|_{L^\infty}+ \| \partial_x u_m \|_{L^\infty} \| \partial_x u_m \|_{H^{s-1}}) \| \partial_x^\beta u_m \|_{L^2} \\ & \le C(R) \| u_m \|_{H^s}^2,
\end{split}
\end{equation}
\begin{equation}
\begin{split}
I_{2,2} &= | \varphi_R( \| u_m\|_{W_x^{2,\infty}}) \int_{} u_m \partial_x \bigg( \dfrac{| \partial_x^\beta u_m |^2}{2} \bigg ) dx | \\ & = | -\frac{1}{2}\varphi_R( \| u_m\|_{W_x^{2,\infty}}) \int_{} \partial_x u_m | \partial_x^\beta u_m |^2 dx | \le C(R) \| u_m \|_{H^s}^2.
\end{split}
\end{equation}
The pressure term can be estimated as follows
\begin{equation}
\begin{split}
|I_3| & \le \varphi_R( \| \psi_m\|_{W_x^{2,\infty}}) \| \partial_x^\beta(\gamma e^{(\gamma-1)\psi_m} \partial_x \psi_m) \|_{L^2} \| \partial_x^\beta u_m \|_{L^2} \\ & \le C \varphi_R( \| u_m\|_{W_x^{2,\infty}}) ( \| e^{(\gamma-1)\psi_m} \|_{L^\infty} \| \psi_m \|_{H^{s+1}} + \| \psi_m \|_{W^{1,\infty}} \| e^{(\gamma-1)\psi_m} \|_{H^s}) \| u_m \|_{H^s} \\ & \le C(R)( \| \psi_m \|_{H^{s+1}}^2+ \| u_m\|_{H^s}^2).
\end{split}
\end{equation}
The first viscous term can be rewritten as follows 
\begin{equation}
\begin{split}
|I_4| & =| \varphi_R( \| \psi_m\|_{W_x^{2,\infty}})\int_{} [\partial_x^\beta(e^{(\alpha-1)\psi_m} \partial_{xx}u_m)-e^{(\alpha-1)\psi_m} \partial_x^\beta \partial_{xx}u_m]\partial_x^\beta u_m \\ & + \varphi_R( \| \psi_m\|_{W_x^{2,\infty}})\int_{} e^{(\alpha-1)\psi_m} \partial_{xx}\partial_x^\beta u_m\partial_x^\beta u_m dx|= I_{4,1}+I_{4,2}
\end{split}
\end{equation}
and we have 
\begin{equation}
\begin{split}
I_{4,1} & \le \varphi_R( \| \psi_m\|_{W_x^{2,\infty}})\| \partial_x^\beta(e^{(\alpha-1)\psi_m} \partial_{xx}u_m)-e^{(\alpha-1)\psi_m} \partial_x^\beta \partial_{xx}u_m \|_{L^2} \| \partial_x^\beta u_m \|_{L^2} \\ & \le C \varphi_R( \| \psi_m\|_{W_x^{2,\infty}})\| e^{(\alpha-1)\psi_m} \|_{H^s} \| \partial_{xx}u_m \|_{L^\infty} \| \partial_x^\beta u_m \|_{L^2} \\ & +C\varphi_R( \| \psi_m\|_{W_x^{2,\infty}}) \| \partial_x(e^{(\alpha-1)\psi_m} ) \|_{L^\infty} \| \partial_{xx}u_m \|_{H^{s-1}})\| \partial_x^\beta u_m \|_{L^2}  \\ & \le C\varphi_R( \| \psi_m\|_{W_x^{2,\infty}}) \| e^{(\alpha-1)\psi_m} \|_{H^s} \| u_m \|_{W^{2,\infty}} \| u_m \|_{H^s} \\ & + C\varphi_R( \| \psi_m\|_{W_x^{2,\infty}})\|  e^{(\alpha-1)\psi_m} \|_{W^{1,\infty}} \| u_m \|_{H^{s+1}} \| u_m \|_{H^s},
\end{split}
\end{equation}
moreover
\begin{equation}\label{dissip m}
\begin{split}
I_{4,2} = & -\varphi_R( \| \psi_m\|_{W_x^{2,\infty}})\int_{} \partial_x(e^{(\alpha-1)\psi_m}) \partial_x \partial_x^\beta u_m \partial_x^\beta u_m dx  \\ & -\varphi_R( \| \psi_m\|_{W_x^{2,\infty}})\int_{}e^{(\alpha-1)\psi_m} \partial_x \partial_x^\beta u_m \partial_x \partial_x^\beta u_m dx  \\ & = \varphi_R( \| \psi_m\|_{W_x^{2,\infty}}) \int_{}\partial_{xx} (e^{(\alpha-1)\psi_m}) \dfrac{| \partial_x^\beta u_m |^2}{2} \\ & - \varphi_R( \| \psi_m\|_{W_x^{2,\infty}}) \int_{} e^{(\alpha-1)\psi_m} | \partial_x^{\beta+1}u_m |^2 dx,
\end{split}
\end{equation}
where we point out that the last term of \eqref{dissip m} has a sign and it is the dissipation term in the energy inequality.
Similarly we estimate the remaining integrals
\begin{equation}
\begin{split}
|I_5|= & |\varphi_R( \| \psi_m\|_{W_x^{2,\infty}}) \int_{} \partial_x^\beta \bigg( \dfrac{\partial_x (e^{\alpha \psi})}{e^\psi} \partial_x u_m \bigg)\partial_x^\beta u_m dx | \\ & \le \delta \| u_m \|^2_{H^{s+1}}+ C(R,\delta) \| u_m \|^2_{H^s}.
\end{split}
\end{equation}
The nonlinear quantum term can be estimated as follows
\begin{equation}
\begin{split}
|I_6|=& |-\varphi_R( \| \psi_m\|_{W_x^{2,\infty}}) \int_{} \partial_x^{\beta-1}(\partial_x \psi_m \partial_{xx}\psi_m)\partial_x^{\beta +1}u_m| \\ & \le C \varphi_R( \| \psi_m\|_{W_x^{2,\infty}})  \| \partial_x \psi_m \|_{L^\infty} \| \partial_x^{\beta-1}( \partial_{xx}\psi_m) \|_{L^2}\| \partial_x^{\beta+1} u_m \|_{L^2}  \\ & + C \varphi_R( \| \psi_m\|_{W_x^{2,\infty}})\| \partial_{xx}\psi_m \|_{L^\infty} \| \partial_x^{\beta-1}(\partial_x \psi_m) \|_{L^2} \| \partial_x^{\beta+1} u_m \|_{L^2} \\ & \le C \varphi_R( \| \psi_m\|_{W_x^{2,\infty}})( \| \partial_x \psi_m \|_{L^\infty} \| \psi_m \|_{H^{s+1}}+ \| \psi_m \|_{W^{2,\infty}} \| \psi_m \|_{H^s}) \| u_m \|_{H^{s+1}} \\ & \le C(R) \| \psi_m \|_{H^{s+1}}^2+ \| u_m \|_{H^{s+1}}^2.
\end{split}
\end{equation}
Concerning the stochastic integral $I_7$ we have
\begin{equation}\label{last m}
\begin{split}
& \mathbb{E} \bigg[ \sup_{t \in [0,T]} \bigg| \int_{0}^{t} \varphi_R( \| u_m\|_{W_x^{2,\infty}}) \int_{\mathbb{T}} \partial_x^\beta \mathbb{F}(\rho_m,u_m) \partial_x^\beta u_m dxdW \bigg| \bigg]^p \\ & \lesssim \mathbb{E} \bigg[ \sum_{k=1}^{\infty} \int_{0}^{T}  \varphi_R( \| u_m\|_{W_x^{2,\infty}})^2 \bigg(\int_{\mathbb{T}} \partial_x^\beta F_k(\rho_m,u_m) \partial_x^\beta u_m dx \bigg)^2 dxdt \bigg]^{\frac{p}{2}} \\ & \lesssim \mathbb{E} \bigg[ \int_{0}^{T} \varphi_R( \| u_m\|_{W_x^{2,\infty}})^2 \bigg( \sum_{k=1}^{\infty} \| \partial_x^\beta F_k(\rho_m,u_m) \|^2_{L^\infty} \bigg) \| \partial_x^\beta u_m \|^2_{L^2} dt \bigg]^{\frac{p}{2}} \\ & \lesssim \mathbb{E} \bigg[ \int_{0}^{T} 1+\| \psi_m \|^2_{H^{s+1}} + \| u_m \|^2_{H^s} dt \bigg]^p.
\end{split}
\end{equation}
The  It$\hat{\text{o}}$ correction term is estimated as follows
\begin{equation}\label{I_8}
\begin{split}
& I_8= \int_{0}^{t} \varphi_R( \| u_m\|_{W_x^{2,\infty}}) \sum_{k=1}^{\infty} \| F_k \|^2_{C^{s-1}} \|(\psi_m,u_m) \|^{2(s-1)}_{L_x^\infty} \| (\psi_m,u_m)\|^2_{H^s}ds\\ &  \le c(R,T) \int_{0}^{t} \| (\psi_m,u_m) \|^2_{H^s}ds,
\end{split}
\end{equation}
where we used Proposition \ref{Moser} and we consider the $C^{s-1}$ norm of $F_k$ restricted to range$[(\psi_m,u_m)]$ which is a bounded set.
Summing up \eqref{1}-\eqref{I_8} we get the following inequality
\begin{equation}
\begin{split}
& \mathbb{E} \bigg[ \sup_{t \in [0,T] } (( \| \psi_m \|^2_{H^{s+1}}+ \| u_m\|^2_{H^{s}} )+ \int_{0}^{T} \varphi_R( \| \psi_m\|_{W_x^{2,\infty}}) \int_{\mathbb{T}} | \partial_x^{\beta+1}u_m|^2 dxdt \bigg]^p \\ & \lesssim c(R,T,s) \mathbb{E} \bigg[ \| \psi_0 \|^{2}_{H^{s+1}} + \| u_0 \|^{2}_{H^s} + \int_{0}^{T} \| \psi_m \|^2_{H^{s+1}}+ \| u_m\|^2_{H^{s}} dt +1 \bigg]^p,
\end{split}
\end{equation}
from which after applying Gronwall Lemma we get 
\begin{equation}\label{reg}
\begin{split}
& \mathbb{E} \bigg[ \sup_{[0,T]} ( \| \psi_m \|^2_{H^{s+1}}+ \| u_m\|^2_{H^{s}} )+ \int_{0}^{T} \varphi_R( \| \psi_m\|_{W_x^{2,\infty}}) \int_{\mathbb{T}} | \partial_x^{\beta+1}u_m|^2 dxdt \bigg]^p  \\ & \le c(R,T,s) \mathbb{E} \bigg[ \| \psi_0 \|_{H^{s+1}}^{2p}+ \| u_0 \|_{H^s}^{2p}+1 \bigg].
\end{split}
\end{equation}
\subsubsection{Compactness}
We define the path space 
$\mathcal{X}=\mathcal{X}_\psi\times\mathcal{X}_u\times\mathcal{X}_W$
where $$\mathcal{X}_\psi=C([0,T];H^\zeta(\mathbb{T}), \ \mathcal{X}_u=C([0,T];H^l(\mathbb{T}),\ \mathcal{X}_W=C([0,T];\mathfrak{U}_0) $$
with $\zeta < s+1$ such that $\zeta > \frac{7}{2}$ and $l< s$ such that $l> \frac{5}{2}.$ 
\\
We denote by $\mathcal{L}[\psi_m],\ \mathcal{L}[u_m],\ \mathcal{L}[W]$ the law of $\psi_m, u_m,  W$ respectively.  Their joint law on $\mathcal{X}$ is denoted by $\mathcal{L}[\psi_m,u_m,W].$ 
In the following Lemma we prove the tightness of the laws.
\begin{lemma}\label{Tightness}
The following results hold
\begin{itemize}
\item[1)]
The set $\{ \mathcal{L}[u_m]; \ m\in \mathbb{N} \}$ is tight on $\mathcal{X}_u$
\item[2)]
The set $\{ \mathcal{L}[\psi_m]; \ m\in \mathbb{N} \}$ is tight on $\mathcal{X}_\psi$
\item[3)]
The set $\{ \mathcal{L}[W] \}$ is tight on $\mathcal{X}_W$
\end{itemize}
\end{lemma}
\begin{proof} 
$1)$ 
By \eqref{Gal u} we have 
\begin{equation}
\begin{split}
& u_m(\tau)= {\Pi}_mu_0+\int_{0}^{\tau} {\Pi}_m[ \partial_{xxx}\psi_m]dt \\ & -\varphi_R( \| u_m\|_{W_x^{2,\infty}}) \int_{0}^{\tau} {\Pi}_m[ u_m \partial_x u_m] +\varphi_R( \| \psi_m\|_{W_x^{2,\infty}})\int_{0}^{\tau}  {\Pi}_m[\gamma e^{(\gamma-1)\psi_m} \partial_x \psi_m] dt \\ &+ \varphi_R( \| \psi_m\|_{W_x^{2,\infty}})\int_{0}^{\tau} {\Pi}_m[e^{(\alpha-1)\psi_m} \partial_{xx}u_m] dt \\ & +\varphi_R( \| \psi_m\|_{W_x^{2,\infty}}) \int_{0}^{\tau}{\Pi}_m [\alpha e^{(\alpha-1)\psi_m}\partial_x \psi_m \partial_x u_m]dt \\ & + \varphi_R( \| \psi_m\|_{W_x^{2,\infty}}) \int_{0}^{\tau} {\Pi}_m[\partial_x \psi_m \partial_{xx}\psi_m] dt +\varphi_R( \| u_m\|_{W_x^{2,\infty}})\int_{0}^{\tau} {\Pi}_m \mathbb{F}(\rho_m,u_m) \text{d}W.
\end{split}
\end{equation}
Now we decompose $u_m$ into two parts $u_m=Y_m+Z_m$ where 
\begin{equation}
\begin{split}
&Y_m(\tau)={\Pi}_mu_0+\int_{0}^{\tau} {\Pi}_m[ \partial_{xxx}\psi_m]dt -\varphi_R( \| u_m\|_{W_x^{2,\infty}}) \int_{0}^{\tau} {\Pi}_m[ u_m \partial_x u_m] dt \\ & + \varphi_R( \| \psi_m\|_{W_x^{2,\infty}})\int_{0}^{\tau}{\Pi}_m[\gamma e^{(\gamma-1)\psi_m} \partial_x \psi_m] dt  \\ & +\varphi_R( \| \psi_m\|_{W_x^{2,\infty}})\int_{0}^{\tau} {\Pi}_m[e^{(\alpha-1)\psi_m} \partial_{xx}u_m] dt \\ & +\varphi_R( \| \psi_m\|_{W_x^{2,\infty}}) \int_{0}^{\tau}{\Pi}_m [\alpha e^{(\alpha-1)\psi_m}\partial_x \psi_m \partial_x u_m]dt \\ & + \varphi_R( \| \psi_m\|_{W_x^{2,\infty}}) \int_{0}^{\tau} {\Pi}_m[\partial_x \psi_m \partial_{xx}\psi_m] dt
\end{split}
\end{equation}
and 
\begin{equation}
Z_m(\tau)=\varphi_R( \| u_m\|_{W_x^{2,\infty}})\int_{0}^{\tau} {\Pi}_m \mathbb{F}(\rho_m,u_m) \text{d}W.
\end{equation}
By using \eqref{reg} the following Holder estimate holds for any $k \in (0,1)$
\begin{equation}
\mathbb{E} [ \| Y_n \|_{C^k_t L^2_x}] \le c(R),
\end{equation}
while concerning the stochastic term the same results holds for $k \in (0,\frac{1}{2}),$ see Lemma \ref{Integrability} in the Appendix.  Now by recalling the following embedding relation 
\begin{equation}\label{emb}
C([0,T];H^s(\mathbb{T}) \cap C^{k}([0,T];L^2(\mathbb{T})) \xhookrightarrow{C}C([0,T];H^l(\mathbb{T})), \quad k>0,\; l<s,
\end{equation}
which follows directly from Ascoli Arzelà theorem,  we get the tightness of the laws.
\\
$2)$ Similarly,  we prove the Holder continuity estimate for $\psi_m$ and the tightness follows from the embedding \eqref{emb} with $\zeta$ instead of $l.$
\\
$3)$ $\mathcal{L}[W]$ is tight since it is a Radon measure on a Polish space $\mathcal{X}_W.$
\end{proof}
\noindent
As a consequence we get the tightness of the joint law
\begin{corollary}
The set $\{\mathcal{L}[\psi_m,u_m,W]; \ m\in \mathbb{N} \}$ is tight on $\mathcal{X}.$
\end{corollary}
\noindent
Since the path space $\mathcal{X}$ is a Polish space,  we make use of Skorokhod representation theorem combined with  Prokhorov's theorem, see Theorem \ref{Skor} and Theorem \ref{Prok} in the Appendix.  Hence we infer the following convergence result.
\begin{lemma}\label{conv stoch}
There exists a complete probability space $(\tilde{\Omega}, \tilde{\mathfrak{F}},\tilde{\mathbb{P})}, $ with $\mathcal{X}$-valued Borel-measurable random variables $(\tilde{\psi}_m,\tilde{u}_m,\tilde{W}_m),  m\in \mathbb{N},$ and $(\tilde{\psi},\tilde{u},\tilde{W}), $ such that up to a subsequence:
\begin{itemize}
\item[(1)]
the law of $(\tilde{\psi}_m,\tilde{u}_m,\tilde{W}_m), $ is given by $\mathcal{L}[\psi_m,u_m,W]; \ m\in \mathbb{N};$
\item[(2)] 
the law of $(\tilde{\psi},\tilde{u},\tilde{W})$ is a Radon measure;
\item[(3)]
$(\tilde{\psi}_m,\tilde{u}_m,\tilde{W}_m), $ converges $\tilde{\mathbb{P}}-$a.s. to $(\tilde{\psi},\tilde{u},\tilde{W})$ int the topology of $\mathcal{X}$ i.e. 
\begin{equation*}
\tilde{\psi}_m \rightarrow \tilde{\psi} \ in \ C[0,T]; H^\zeta(\mathbb{T})),
\end{equation*}
\begin{equation*}
\tilde{u}_m \rightarrow \tilde{u} \ in \ C[0,T]; H^l(\mathbb{T})),
\end{equation*}
\begin{equation*}
\tilde{W}_m \rightarrow \tilde{W} \ in \ C[0,T]; \mathfrak{U}_0),
\end{equation*}
\end{itemize}
as $m \rightarrow \infty \ \tilde{\mathbb{P}}-$a.s.
\end{lemma}
\subsubsection{Identification of the limit}
In this section we identify the limit obtained in the previous Lemma with a strong martingale solution of \eqref{approx R stoch quantum rho u}-\eqref{approx R stoch quantum u}.  \\
Since $\tilde{\psi}$ and $\tilde{u}$ are stochastic processes with continuous trajectories,  then they are progressively measurable stochastic processes with respect to their canonical filtration respectively,  so they are progressively measurable with respect to the canonical filtration generated by $[ \tilde{\psi}, \tilde{u}, \tilde{W}]$ i.e. 
\begin{equation}
\tilde{\mathfrak{F}}_t:= \sigma \bigg( \sigma_t [\tilde{\psi}] \cup \sigma_t [\tilde{u}] \cup \bigcup_{k=1}^{\infty} \sigma_t [\tilde{W}_k] \bigg), \quad t \in [0,T].
\end{equation}
\\
Then the process $\tilde{W}$ is a cylindrical Wiener process with respect to its canonical filtration and moreover it is a cylindrical Wiener process with respect to $(\tilde{\mathfrak{F}}_t)_{t \ge 0}.$ \\
Furthermore $(\tilde{\mathfrak{F}}_t)_{t \ge 0}$ is non-anticipative with respect to $\tilde{W}$, indeed for every $m \in \mathbb{N}, \; \tilde{W}_m= \sum_{k=1}^{\infty} e_k  \tilde{W}_{m,k} $ is a cylindrical Wiener process with respect to 
\begin{equation}
\sigma \bigg( \sigma_t [\tilde{\psi}_m] \cup \sigma_t [\tilde{u}_m] \cup \bigcup_{k=1}^{\infty} \sigma_t [\tilde{W}_{m,k}] \bigg), \quad t \in [0,T].
\end{equation}
and so this filtration is non-anticipative with respect to $\tilde{W}_m$,  our claim then follows. \\
Next we claim that $[\tilde{\psi},\tilde{u},\tilde{W}]$ is a strong martingale solution to \eqref{approx R stoch quantum rho u}-\eqref{Approx C.I}. \\
First we observe that by virtue of Skorokhod Theorem, the stochastic process $[\tilde{\psi}_m,\tilde{u}_m]$ solves \eqref{Gal psi}- \eqref{Gal u} on the new probability space and passing to the limit we deduce $[\tilde{\psi}, \tilde{u}]$ solves \eqref{approx R stoch quantum rho u}. Similarly,  by using the same lines of argument, the equations \eqref{approx R stoch quantum u} is solved by $[\tilde{ \psi}_m, \tilde{u}_m, \tilde{W}_m]$ on the new probability space.\\
Finally by Lemma \ref{conv stoch} and the uniform estimates \eqref{reg}, we pass to the limit completing the existence part of Theorem \ref{second thm}.
Concerning the strong continuity in time of $(\tilde{\psi}, \tilde{u}),$ it follows by using Gelfand variational approach. Indeed since the stochastic integral has continuous trajectories in $H^{s}$ and the deterministic counterpart in the momentum equation belongs to $L^2(0,T;H^{s-1})$ a.s., the momentum equation \eqref{approx R stoch quantum u} is solved in the Gelfand triplet $$H^{s+1}(\mathbb{T}) \hookrightarrow H^{s}(\mathbb{T})\hookrightarrow H^{s-1}(\mathbb{T}).$$ We refer the reader to Theorem 2.4.3 in \cite{Feir} for a detailed discussion.
\subsubsection{Pathwise uniqueness}
This subsection is dedicated to the proof of a pathwise uniqueness estimate.  This is crucial in order to apply our convergence argument,  we will see that an additional space integrability condition $s> \frac{7}{2}$ will be needed.
Let $(\psi_1,u_1), \ (\psi_2,u_2)$ be two different solutions of \eqref{approx R stoch quantum rho u}-\eqref{Approx C.I},  their difference satisfies
\begin{equation}\label{diff1}
\begin{split}
& \text{d}(\psi_1-\psi_2)+\varphi_R( \| u_1\|_{W_x^{2,\infty}})[ u_1\partial_x \psi_1]dt \\ & -\varphi_R( \| u_2\|_{W_x^{2,\infty}})[ u_2\partial_x \psi_2]dt+ \partial_x (u_1-u_2) dt=0,
\end{split}
\end{equation}
\begin{equation}\label{diff2}
\begin{split}
& \text{d}(u_1-u_2)=\partial_{xxx}(\psi_1-\psi_2)dt -\varphi_R( \| u_1\|_{W_x^{2,\infty}})[u_1\partial_x u_1]dt -\varphi_R( \| \psi_1\|_{W_x^{2,\infty}})[ \gamma e^{(\gamma-1)\psi_1} \partial_x \psi_1]dt \\ & +\varphi_R( \| u_2\|_{W_x^{2,\infty}})[u_2\partial_x u_2]dt + \varphi_R( \| \psi_2\|_{W_x^{2,\infty}})[\gamma e^{(\gamma-1)\psi_2} \partial_x \psi_2]dt \\ & +\varphi_R( \| \psi_1\|_{W_x^{2,\infty}})[e^{(\alpha-1)\psi_1} \partial_{xx}u_1]dt -\varphi_R( \| \psi_2\|_{W_x^{2,\infty}})[e^{(\alpha-1)\psi_2} \partial_{xx}u_2]dt \\ & +\varphi_R( \| \psi_1\|_{W_x^{2,\infty}})[\alpha e^{(\alpha-1)\psi_1}\partial_x \psi_1 \partial_x u_1 ] dt -\varphi_R( \| \psi_2\|_{W_x^{2,\infty}})[\alpha e^{(\alpha-1)\psi_2}\partial_x \psi_2 \partial_x u_2 ]dt \\ &+ \varphi_R( \| \psi_1\|_{W_x^{2,\infty}})[\partial_x \psi_1 \partial_{xx}\psi_1]dt -\varphi_R( \| \psi_2\|_{W_x^{2,\infty}})[\partial_x \psi_2 \partial_{xx}\psi_2]dt \\ & +\varphi_R( \| u_1\|_{W_x^{2,\infty}}) \mathbb{F}(\rho_1,u_1)\text{d}W-\varphi_R( \| u_2\|_{W_x^{2,\infty}}) \mathbb{F}(\rho_2,u_2)\text{d}W.
\end{split}
\end{equation}
We apply $\partial_x^{\beta+1}$ to \eqref{diff1} and $\partial_x^{\beta}$ to \eqref{diff2} with $| \beta | \le s'$
\begin{equation}
\begin{split}
& \text{d}\partial_x^{\beta+1}(\psi_1-\psi_2)+\varphi_R( \| u_1\|_{W_x^{2,\infty}})[ \partial_x^{\beta+1}(u_1\partial_x \psi_1)]dt \\ & -\varphi_R( \| u_2\|_{W_x^{2,\infty}})[ \partial_x^{\beta+1}(u_2\partial_x \psi_2)]dt+ \partial_x^{\beta+2} (u_1-u_2) dt=0,
\end{split}
\end{equation}
\begin{equation}
\begin{split}
& \text{d}\partial_x^{\beta}(u_1-u_2)=\partial_x^{\beta+3}(u_1-u_2)dt-\varphi_R( \| u_1\|_{W_x^{2,\infty}})[\partial_x^{\beta}(u_1\partial_x u_1)]dt \\ & -\varphi_R( \| \psi_1\|_{W_x^{2,\infty}}) [\partial_x^{\beta}(\gamma e^{(\gamma-1)\psi_1} \partial_x \psi_1)]dt \\ & +\varphi_R( \| u_2\|_{W_x^{2,\infty}})[\partial_x^{\beta}(u_2\partial_x u_2)]dt + \varphi_R( \| \psi_2\|_{W_x^{2,\infty}})[\partial_x^{\beta}(\gamma e^{(\gamma-1)\psi_2} \partial_x \psi_2)]dt \\ &+\varphi_R( \| \psi_1\|_{W_x^{2,\infty}})[\partial_x^{\beta}(e^{(\alpha-1)\psi_1} \partial_{xx}u_1)]dt-\varphi_R( \| \psi_2\|_{W_x^{2,\infty}})[\partial_x^{\beta}(e^{(\alpha-1)\psi_2} \partial_{xx}u_2)]dt \\ & +\varphi_R( \| \psi_1\|_{W_x^{2,\infty}})[\partial_x^{\beta}(\alpha e^{(\alpha-1)\psi_1}\partial_x \psi_1 \partial_x u_1 )]-\varphi_R( \| \psi_2\|_{W_x^{2,\infty}})[\partial_x^{\beta}(\alpha e^{(\alpha-1)\psi_2}\partial_x \psi_2 \partial_x u_2) ] \\ & + \varphi_R( \| \psi_1\|_{W_x^{2,\infty}})[\partial_x^{\beta}(\partial_x \psi_1 \partial_{xx}\psi_1)]dt-\varphi_R( \| \psi_2\|_{W_x^{2,\infty}})[\partial_x^{\beta}(\partial_x \psi_2 \partial_{xx}\psi_2)]dt \\ & +\varphi_R( \| u_1\|_{W_x^{2,\infty}}) \partial_x^{\beta}\mathbb{F}(\rho_1,u_1)\text{d}W-\varphi_R( \| u_2\|_{W_x^{2,\infty}}) \partial_x^{\beta}\mathbb{F}(\rho_2,u_2)\text{d}W.
\end{split}
\end{equation}
We multiply the first equation by $\partial_x^{\beta+1}(\psi_1-\psi_2)$ and integrate in space and we apply It$\hat{\text{o}}$ formula with $F(\partial_x^{\beta}(u_1-u_2))=\dfrac{1}{2} \int_{} | \partial_x^{\beta}(u_1-u_2) |^2 dx$ to the second equation.  Summing up the two resulting equations and after integrating by parts we get
\begin{equation}\label{eq path uniq}
\begin{split}
&\dfrac{1}{2}\text{d}\int_{}| \partial _{x}^{\beta +1}( \psi _{1}-\psi _{2})| ^{2}+ | \partial_x^{\beta}(u_1-u_2)|^2 dx=-\varphi_R ( \| u_1 \|) \int_{} \partial_x^{\beta+1} (u_1\partial_x \psi_1) \partial_x^{\beta+1}(\psi_1-\psi_2) dxdt \\ & +\varphi_R ( \| u_2 \|) \int_{} \partial_x^{\beta+1} (u_2\partial_x \psi_2) \partial_x^{\beta+1}(\psi_1-\psi_2) dxdt- \varphi_R( \| u_1 \|) \int_{} \partial_x^{\beta}(u_1 \partial_xu_1)\partial_x^\beta(u_1-u_2)dxdt \\ & -\varphi_R( \| \psi_1 \|) \int_{} \partial_x^{\beta}(\gamma e^{(\gamma-1) \psi_1} \partial_x \psi_1)\partial_x^\beta(u_1-u_2)dxdt+\varphi_R( \| u_2 \|) \int_{} \partial_x^{\beta}(u_2 \partial_xu_2)\partial_x^\beta(u_1-u_2)dxdt \\ & +\varphi_R( \| \psi_2 \|) \int_{} \partial_x^{\beta}(\gamma e^{(\gamma-1) \psi_2} \partial_x \psi_2)\partial_x^\beta(u_1-u_2)dxdt\\ & + \varphi_R( \| \psi_1\|_{W_x^{2,\infty}})\int_{} \partial_x^\beta(e^{(\alpha-1)\psi_1} \partial_{xx}u_1)\partial_x^\beta(u_1-u_2)dxdt \\ & -\varphi_R( \| \psi_2\|_{W_x^{2,\infty}})\int_{} \partial_x^\beta(e^{(\alpha-1)\psi_2} \partial_{xx}u_2)\partial_x^\beta(u_1-u_2)dxdt \\ & + \varphi_R( \| \psi_1 \|) \int_{} \partial_x^\beta(\alpha e^{(\alpha-1)\psi_1} \partial_x \psi_1 \partial_xu_1) \partial_x^\beta(u_1-u_2) dxdt \\ & -\varphi_R( \| \psi_2 \|) \int_{} \partial_x^\beta(\alpha e^{(\alpha-1)\psi_2} \partial_x \psi_2 \partial_xu_2) \partial_x^\beta(u_1-u_2) dxdt \\ & + \varphi_R( \|\psi_1 \|) \int_{} \partial_x^\beta(\partial_x \psi_1 \partial_{xx}\psi_1) \partial_x^\beta(u_1-u_2)dxdt  \\ & -\varphi_R( \|\psi_2 \|) \int_{} \partial_x^\beta(\partial_x \psi_2 \partial_{xx}\psi_2) \partial_x^\beta(u_1-u_2)dxdt+ \varphi_R( \|u_1 \|) \int_{} \partial_x^\beta \mathbb{F}(\rho_1,u_1)\partial_x^\beta(u_1-u_2)\text{d}W \\ & -\varphi_R( \|u_2 \|) \int_{} \partial_x^\beta \mathbb{F}(\rho_2,u_2)\partial_x^\beta(u_1-u_2)\text{d}W \\ & +\dfrac{1}{2} \sum_{k=1}^{\infty} \bigg[ \varphi_R( \|u_1 \|) \partial_x^\beta F_k(\rho_1,u_1)-\varphi_R( \|u_2 \|) \partial_x^\beta F_k(\rho_2,u_2) \bigg]^2dt.
\end{split}
\end{equation}
In the following we mimic the computations performed in Section 4.1.2, we rewrite and estimate the integrals in \eqref{eq path uniq}. We first observe that 
\begin{equation}
| \varphi_R( \| u_1\|_{W_x^{2,\infty}})-\varphi_R( \| u_2\|_{W_x^{2,\infty}}) | \le C_1(R) \| u_1-u_2 \|_{W_x^{2,\infty}} \le C_2(R) \| u_1-u_2 \|_{H^{s'}}
\end{equation}
for $s'> \frac{5}{2}.$
Concerning the highest order terms in \eqref{eq path uniq} we have
\begin{equation}
\begin{split}
& |- \varphi_R( \| u_1\|_{W_x^{2,\infty}}) \int_{} (u_1\partial_x \partial_x^{\beta+1} \psi_1-u_2 \partial_x \partial_x^{\beta+1} \psi_2) \partial_x^{\beta+1}(\psi_1-\psi_2)dx|= \\ & | \varphi_R( \| u_1\|_{W_x^{2,\infty}}) \int_{} (u_1-u_2) \partial_x \partial_x^{\beta+1} \psi_1 \partial_x^{\beta+1}(\psi_1-\psi_2)+ \varphi_R( \| u_1\|_{W_x^{2,\infty}}) \int_{} \partial_x u_2 \dfrac{| \partial_x^{\beta+1}(\psi_1-\psi_2)|}{2}| \\ & \le \varphi_R( \| u_1\|_{W_x^{2,\infty}}) (\| u_1-u_2\|_{L^\infty} \| \psi_1 \|_{H^{s'+2}} \| \psi_1-\psi_2 \|_{H^{s'+1}} + \| u_2 \|_{W^{1,\infty}} \| \psi_1-\psi_2 \|^2_{H^{s'+1}}) \\ & \le C_1(R)+ \| \psi_1 \|_{H^{s'+2}}^2 \| \psi_1-\psi_2 \|^2_{H^{s'+1}}+ C_2 \| \psi_1-\psi_2 \|_{H^{s'+1}}^2,
\end{split}
\end{equation}
\begin{equation}
\begin{split}
& |- \varphi_R( \| u_1 \|_{W_x^{2,\infty}}) \int_{} (u_1\partial_x \partial_x^{\beta}u_1-u_2\partial_x \partial_x^{\beta} u_2)\partial_x^\beta(u_1-u_2)dx |= \\ & | \varphi_R( \| u_1 \|_{W_x^{2,\infty}}) \int_{} (u_1-u_2) \partial_x \partial_x^{\beta}u_1\partial_x^{\beta}(u_1-u_2)+u_2 \partial_x \partial_x^\beta(u_1-u_2)dx | \\ & \le \varphi_R( \| u_1 \|_{W_x^{2,\infty}}) \| u_1-u_2\|_{L^\infty} \|u_1 \|_{H^{s'+1}} \|u_1-u_2\|_{H^{s'}}+ \| u_2\|_{W^{1,\infty}} \|u_1-u_2 \|_{H^{s'}}^2 \\ & \le C_1(R) \| u_1-u_2 \|^2_{H^{s'}}+C_2\| u_1-u_2 \|^2_{H^{s'}},
\end{split}
\end{equation}
\\
\begin{equation}
\begin{split}
& | \varphi_R( \| u_1 \|_{W_x^{2,\infty}}) \int_{} \gamma e^{(\gamma-1) \psi_1}\partial_x \partial_x^{\beta} \psi_1 \partial_x^{\beta}(u_1-u_2) dx | \\ & \le  C \varphi_R( \| u_1 \|_{W_x^{2,\infty}}) \| \psi_1\|_{L^\infty} \| \psi_1\|_{H^{s'+1}} \| u_1-u_2 \|_{H^{s'}} \\ & \le C(R) \| \psi_1 \|_{H^{s'+1}}^2+ \| u_1-u_2\|_{H^{s'}}^2,
\end{split}
\end{equation}
\\
\begin{equation}
\begin{split}
& | \varphi_R( \| \psi_2 \|_{W_x^{2,\infty}}) \int_{} \gamma e^{(\gamma-1) \psi_2}\partial_x \partial_x^{\beta} \psi_2 \partial_x^{\beta}(u_1-u_2) dx | \\ & \le  C \varphi_R( \| \psi_2 \|_{W_x^{2,\infty}}) \| \psi_2\|_{L^\infty} \| \psi_2\|_{H^{s'+1}} \| u_1-u_2 \|_{H^{s'}} \\ & \le C(R) \| \psi_2 \|_{H^{s'+1}}^2+ \| u_1-u_2\|_{H^{s'}}^2,
\end{split}
\end{equation}
\\
\begin{equation}
\begin{split}
& \varphi_R( \| \psi_1\|_{W_x^{2,\infty}})| \int_{} (e^{(\alpha-1)\psi_1} \partial_{xx} \partial_x^{\beta}u_1-e^{(\alpha-1)\psi_2} \partial_{xx} \partial_x^{\beta}u_2) \partial_x^{\beta}(u_1-u_2) dx |= \\ & \varphi_R( \| \psi_1\|_{W_x^{2,\infty}})| \int_{}(e^{(\alpha-1)\psi_1}-e^{(\alpha-1)\psi_2}) \partial_{xx} \partial_x^{\beta} u_1 \partial_x^{\beta}(u_1-u_2)\\ & + e^{(\alpha-1)\psi_2} \partial_{xx} \partial_x^{\beta}(u_1-u_2) \partial_x^{\beta}(u_1-u_2) dx |,
\end{split}
\end{equation}
that can be estimate as
\begin{equation}
\begin{split}
& \varphi_R( \| \psi_1\|_{W_x^{2,\infty}})| \int_{}(e^{(\alpha-1)\psi_1}-e^{(\alpha-1)\psi_2}) \partial_{xx} \partial_x^{\beta} u_1 \partial_x^{\beta}(u_1-u_2) |  \\ & \le C \| \psi_1 \|_{L^\infty} \| \psi_2 \|_{L^\infty} \| u_1 \|_{H^{s'+2}} \| u_1-u_2 \|_{H^{s'}},
\end{split}
\end{equation}
and 
\begin{equation}
\begin{split}
& \varphi_R( \| \psi_1\|_{W_x^{2,\infty}})|\int_{} e^{(\alpha-1)\psi_2} \partial_{xx} \partial_x^{\beta}(u_1-u_2) \partial_x^{\beta}(u_1-u_2) dx |= \\ & \varphi_R( \| \psi_1\|_{W_x^{2,\infty}})| \int_{} \partial_{xx} (e^{(\alpha-1) \psi_2}) \dfrac{| \partial_x^{\beta}(u_1-u_2)|^2}{2}- \int_{} e^{(\alpha-1)\psi_2} | \partial_x^{\beta+1}(u_1-u_2) |^2 dx \\ & \le C\varphi_R( \| \psi_1\|_{W_x^{2,\infty}}) (\| \psi_2\|_{W^{2,\infty}} \| u_1-u_2 \|_{H^{s'}}^2+ \| \psi_2\|_{L^\infty} \| u_1-u_2 \|_{H^{s'+1}}^2),
\end{split}
\end{equation}
\\
\begin{equation}
\begin{split}
& | \varphi_R( \| \psi_1 \| ) \int_{} \alpha e^{(\alpha-1) \psi_1} \partial_x \psi_1 \partial_x^\beta \partial_x u_1 \partial_x^\beta(u_1-u_2) dx |  \\ & + |\varphi_R( \| \psi_2 \| ) \int_{} \alpha e^{(\alpha-1) \psi_2} \partial_x \psi_2 \partial_x^\beta \partial_x u_2 \partial_x^\beta(u_1-u_2) dx | \\ &  \le C (\| \psi_1 \|_{W^{1,\infty}}^2 \| u_1 \|_{H^{s'+1}}+\| \psi_2 \|_{W^{1,\infty}}^2 \| u_2 \|_{H^{s'+1}} ) \| u_1-u_2 \|_{H^{s'}}. 
\end{split}
\end{equation}
Concerning the quantum term 
\begin{equation}
\begin{split} 
\varphi_R( \| \psi_1 \|_{W_x^{2,\infty}}) \int_{} \partial_x^{\beta}( \partial_x \psi_1 \partial_{xx}\psi_1 ) \partial_x^{\beta}(u_1-u_2)-\varphi_R( \| \psi_2 \|_{W_x^{2,\infty}}) \int_{} \partial_x^{\beta}( \partial_x \psi_2 \partial_{xx}\psi_2 ) \partial_x^{\beta}(u_1-u_2),
\end{split}
\end{equation}
the high order part is given by
\begin{equation}
\begin{split}
& | \varphi_R( \| \psi_1 \|_{W_x^{2,\infty}}) \int_{} \partial_x \psi_1 \partial_{xx} \partial_x^\beta \psi_1 \partial_x^{\beta}(u_1-u_2) dx | \\ & + | \varphi_R( \| \psi_2 \|_{W_x^{2,\infty}}) \int_{} \partial_x \psi_2 \partial_{xx} \partial_x^\beta \psi_2 \partial_x^{\beta}(u_1-u_2) dx | \\ & \le ( \varphi_R( \| \psi_1 \|_{W_x^{2,\infty}}) \| \psi_1\|_{W^{1,\infty}} \| \psi_1 \|_{H^{s'+2}}+\varphi_R( \| \psi_2 \|_{W_x^{2,\infty}}) \| \psi_2\|_{W^{1,\infty}} \| \psi_2 \|_{H^{s'+2}} ) \| u_1-u_2 \|_{H^{s'}}.
\end{split}
\end{equation}
For $s' > \frac{5}{2}$ summing up all these estimates we deduce that 
\begin{equation}
\begin{split}
& \text{d}( \| \psi_1-\psi_2\|_{H^{s'+1}}^2+ \| u_1-u_2 \|_{H^{s'}}^2 )  \le G(t) ( \| \psi_1-\psi_2\|_{H^{s'+1}}^2+ \| u_1-u_2 \|_{H^{s'}}^2 ) \\ & + \sum_{ | \beta | \le s'} [ \varphi( \| u_1 \|_{W^{2,\infty}} ) \partial_x^{\beta} \mathbb{F}( \rho_1,u_1)- \varphi( \| u_2 \|_{W^{2,\infty}} ) \partial_x^{\beta} \mathbb{F}( \rho_2,u_2)] \partial_x^{\beta} (u_1-u_2) \text{d}W,
\end{split}
\end{equation}
with $G(t)$ defined by $$G(t)= C(R) [ 1+ \| \psi_1 \|_{H^{s'+2}}^2 + \| \psi_2 \|_{H^{s'+2}}^2 + \| u_1 \|_{H^{s'+2}}^2+ \| u_2 \|_{H^{s'+2}}^2].$$
Now,  for $s \ge s'+1$, by using the a priori estimates recovered in Section 4.1.2 we get that $G \in L^1(0,T)$ and by It$\hat{\text{o}}$ product rule
\begin{equation}\label{exp path uniq}
\begin{split}
& \text{d} [ e^{-\int_{0}^{t} G(\sigma) d\sigma} ( \| \psi_1-\psi_2 \|_{H^{s'+1}}^2 + \| u_1-u_2 \|_{H^{s'}}^2)]= \\ & -G(t) e^{-\int_{0}^{t} G(\sigma) d\sigma}( \| \psi_1-\psi_2\|_{H^{s'+1}}^2+ \| u_1-u_2 \|_{H^{s'}}^2 ) dt  \\ & + e^{-\int_{0}^{t} G(\sigma) d\sigma}  \text{d}( \| \psi_1-\psi_2\|_{H^{s'+1}}^2+ \| u_1-u_2 \|_{H^{s'}}^2 ) \\ & \le \sum_{ | \beta | \le s'} e^{-\int_{0}^{t} G(\sigma) d\sigma} \varphi( \| u_1 \|_{W^{2,\infty}} ) \partial_x^{\beta} \mathbb{F}( \rho_1,u_1)\partial_x^{\beta}(u_1-u_2) \text{d}W \\ & -\sum_{ | \beta | \le s'}e^{-\int_{0}^{t} G(\sigma) d\sigma} \varphi( \| u_2 \|_{W^{2,\infty}} ) \partial_x^{\beta} \mathbb{F}( \rho_2,u_2)\partial_x^{\beta}(u_1-u_2) \text{d}W.
\end{split}
\end{equation}
Next, integrating in $[0,t]$ and taking the expectation we get 
\begin{equation}
\mathbb{E} [e^{-\int_{0}^{t} G(\sigma) d\sigma}( \| \psi_1-\psi_2\|_{H^{s'+1}}^2+ \| u_1-u_2 \|_{H^{s'}}^2 ) ]=0,
\end{equation}
whenever 
\begin{equation}
\mathbb{E} [( \| \psi^0_1-\psi^0_2\|_{H^{s'+1}}^2+ \| u^0_1-u^0_2 \|_{H^{s'}}^2 ) ]=0,
\end{equation} 
indeed the stochastic integrals in the left hand side of \eqref{exp path uniq} vanish due to their zero mean.
\\
Now since $e^{.-\int_{0}^{t} G(\sigma) d\sigma} > 0 \quad \mathbb{P}-a.s$
and $\psi_i,u_i$ have continuous trajectories in $H^{s'}(\mathbb{T})\; a.s. $ then pathwise uniqueness holds.
\subsubsection{Existence of a strong pathwise approximate solution}
Once pathwise uniqueness and existence of a strong martingale solution have been established,  we make use of the Gyongy-Krylov characterization of convergence in probability in order to prove the existence of a pathwise solution, see Lemma \ref{Gyong} in the Appendix and section 2.10 in \cite{Feir} for further details.
We start considering a regular initial data satisfying \eqref{C.I STRONG} for $s> \frac{7}{2}$ so that pathwise uniqueness holds.  \\ Let $(\psi_m, \, u_m),\, (\psi_n, \, u_n)$ be solutions of \eqref{Gal psi}-\eqref{Gal u},  $\mathfrak{L}[\psi_m, \, u_m,\, \psi_n, \, u_n,W]$ the joint law on $\mathcal{X}^{J}=\mathcal{X}_\psi \times \mathcal{X}_u \times \mathcal{X}_\psi \times \mathcal{X}_u \times \mathcal{X}_W$ then the following result follows by the same lines of argument of Lemma \ref{Tightness}
\begin{lemma}
The set $ \{ \mathfrak{L}[\psi_m, \, u_m,\, \psi_n, \, u_n,W]; m, \,  n \in \mathbb{N} \} $ is tight on $\mathcal{X}^{J}.$
\end{lemma}
\noindent
Let us take a subsequence $[\psi_{m_k,} \, u_{m_k},\, \psi_{n_k}, \, u_{n_k},W]_{k\in \mathbb{N}}$, then by using Skorokhod representation theorem we get the existence of a complete probability space $(\overline{\Omega}, \mathfrak{\overline{F}},\overline{\mathbb{P}})$ with a sequence of random variables $$[\hat{\psi}_{m_k,} \, \hat{u}_{m_k},\, \bar \psi_{n_k}, \, \bar u_{n_k},\tilde{W}_k], \quad k \in \mathbb{N}$$
and a random variable $[\hat{\psi},\, \hat{u},\, \bar \psi
, \, \bar u_,\tilde{W}]$ such that $$[\hat{\psi}_{m_k,} \, \hat{u}_{m_k},\, \bar \psi_{n_k}, \, \bar u_{n_k},\tilde{W}_k] \rightarrow [\hat{\psi},\, \hat{u},\, \bar \psi, \, \bar u_,\tilde{W}] \quad  \text{in} \; \mathcal{X}^{J}, \; \mathbb{\overline{P}} -a.s. $$
and $$\mathfrak{L}[\hat{\psi}_{m_k,} \, \hat{u}_{m_k},\, \bar \psi_{n_k}, \, \bar u_{n_k},\tilde{W}_k]=\mathfrak{L}[{\psi}_{m_k},\, {u}_{m_k},\, \psi_{n_k} \, u_{n_k},{W}] \quad  \text{in} \; \mathcal{X}^{J}$$
Moreover $\mathfrak{L}[{\psi}_{m_k},\, {u}_{m_k},\, \psi_{n_k} \, u_{n_k},{W}]$ converges weakly to the measure $\mathfrak{L}[\hat \psi,\hat u, \bar \psi, \bar u, \tilde{W}].$ \\
Next,  similarly to the previous section we can prove that both $[\hat \psi,\hat u, \tilde{W}] $ and $[ \bar \psi,\bar u, \tilde{W}]$  are strong martingale solution to the approximate system \eqref{approx R stoch quantum rho u}-\eqref{approx R stoch quantum u}. \\
Finally,  since $\psi_{n_k}(0)= \Pi_{n_k}\psi_0,$  $ \psi_{m_k}(0)= \Pi_{m_k}\psi_0$ we get for every $l \le n_k \land m_k,$
$$\mathbb{ \bar {P}}( \Pi_l \hat \psi_{n_k}(0)=\Pi_l \bar \psi_{m_k}(0))= \mathbb{P}( \Pi_l  \psi_{n_k}(0)=\Pi_l  \psi_{m_k}(0))=1 $$
which implies $\mathbb{\bar P}(\hat \psi(0)=\bar \psi(0))=1.$ \\  Similarly for $u$ we have that since $u_{n_k}(0)= \Pi_{n_k}u_0,$  $ u_{m_k}(0)= \Pi_{m_k}u_0,$ then for every $l \le n_k \land m_k,$
$$\mathbb{ \bar {P}}( \Pi_l \hat u_{n_k}(0)=\Pi_l \hat u_{m_k}(0))= \mathbb{P}( \Pi_l  u_{n_k}(0)=\Pi_l  u_{m_k}(0))=1 $$
 which implies
$\mathbb{\bar P}(\hat u(0)=\bar u(0))=1$ \\
Finally by using pathwise uniqueness,  we get the following result
$$\mathfrak{L}[ \hat \psi,  \hat u,  \bar \psi, \bar u]([ \psi_1,u_1, \psi_2,u_2]; [\psi_1,u_1]=[\psi_2,u_2])= \mathbb{ \bar P}([ \hat \psi, \hat u]=[ \bar \psi,  \bar u ])=1$$
from which we deduce that the original sequence $[\psi_m,u_m]$ defined on the initial probability space $(\Omega, \mathfrak{F}, \mathbb{P}),$ converges in probability in the topology of $\mathcal{X}_\psi \times \mathcal{X}_u$ to a random variable $[\psi,u].$ Passing to a subsequence we assume the convergence is almost surely and by using the same line of argument of the previous section we identify the limit as the unique strong pathwise solution of the approximate problem \eqref{approx R stoch quantum rho u}-\eqref{approx R stoch quantum u}.  We denote such a solution $[\psi_R,u_R].$
\subsection{Existence and uniqueness for the original system}
After having established the existence and uniqueness of a strong pathwise solution of the approximating system \eqref{approx R stoch quantum rho u}-\eqref{approx R stoch quantum u},  we introduce suitable stopping times in order to show the convergence as $R\rightarrow \infty$ of the approximate solution to the solution of system \eqref{stoc quantum} and complete the proof of Theorem \ref{Main Theorem local}. We point out that the initial conditions \eqref{C.I} are not assumed to be integrable in the $\omega$ variable,  hence we first consider an additional integrability assumption and then we remove it.
\subsubsection{Uniqueness}
We consider the following additional assumption on the initial data 
\begin{equation}\label{add initial}
\rho_0 \in L^{\infty}(\Omega; \mathfrak{F}_0, \mathbb{P}, H^{s+1}(\mathbb{T})),  \quad 
u_0\in L^{\infty}(\Omega; \mathfrak{F}_0, \mathbb{P}, H^{s}(\mathbb{T}))
\end{equation}
With this assumption,  pathwise uniqueness for the original system \eqref{stoc quantum}-\eqref{C.I} follows from the pathwise uniqueness proved in Theorem \ref{second thm}. Indeed let $[\rho^i,u^i, (\tau^i_R),\tau^i]. \;i=1,2,$ be two maximal strong pathwise solution of \eqref{stoc quantum}-\eqref{C.I} satisfying \eqref{add initial},  then $$[\psi^i:= \log \rho^i,u^i], \quad i=1,2$$
are both solution of \eqref{approx R stoch quantum rho u}-\eqref{approx R stoch quantum u} up to the stopping time $\tau^1_R \land \tau^2_R$ with the same initial condition.  We have $$ \mathbb{P}([\rho^1,u^1](t \land \tau^1_R \land \tau^2_R)=[\rho^2,u^2](t \land \tau^1_R \land \tau^2_R), \; \text{for all} \,  t \in [0,T])=1.$$
By sending $R \rightarrow \infty, $ by dominated convergence theorem we get $$ \mathbb{P}([\rho^1,u^1](t \land \tau^1 \land \tau^2)=[\rho^2,u^2](t \land \tau^1 \land \tau^2), \; \text{for all} \,  t \in [0,T])=1,$$
which implies that the two solutions coincides up to the stopping time $\tau^1 \land \tau^2$ and by maximality of $\tau^1$ and $ \tau^2$ we get $\tau^1=\tau^2$ a.s.
\\
Now in order to remove the additional hypothesis \eqref{add initial} we define for $K > 0$  $$\Omega_K= \bigg\{ \omega \in \Omega \; \text{s.t.} \; \| \psi_0(\omega) \|_{H^{s+1}} < K, \, \| u_0(\omega) \|_{H^s} < K \bigg\}.$$
We have that $\Omega= \bigcup_{K \in \mathbb{R}} \Omega_K$ and since $\Omega_K$ is $\mathfrak{F}_0$-measurable for any $K \in \mathbb{R}$ we can rewrite the a priori estimates derived in Section 4.1.2 by using $\mathbb{E}[ 1_{\Omega_K} \cdot]$ instead of $\mathbb{E}[\cdot ]$ obtaining 
\begin{equation}
\begin{split}
\mathbb{E} \bigg[ 1_{\Omega_K} \bigg( \sup_{ t \in [0, T \land \tau^i_R]} \| ( \psi^i(t), u^i(t))\|^2_{H^{s+1} \times H^{s}} + \int_{0}^{T \land \tau^i_R} \| u^i(t) \|^2_{H^{s+1}} dt \bigg)^p \bigg] \lesssim c(R,T,s,K),
\end{split}
\end{equation}
for $i=1,2.$
By applying the method used in Section 4.1.5 we have the following pathwise uniqueness result in $\Omega_K$
$$ \mathbb{P}(1_{\Omega_K}[\rho^1,u^1](t \land \tau^1_R \land \tau^2_R)=1_{\Omega_K}[\rho^2,u^2](t \land \tau^1_R \land \tau^2_R), \; \text{for all} \,  t \in [0,T])=1.$$
Finally,  we observe that sending $R,K \rightarrow \infty$ then $1_{\Omega_K} \rightarrow 1_{\Omega}, \; \tau^i_R \rightarrow \tau^i$, for $i=1,2.$ Hence by applying the dominated convergence theorem we get the following result $$ \mathbb{P}([\rho^1,u^1](t \land \tau^1 \land \tau^2)=[\rho^2,u^2](t \land \tau^1 \land \tau^2), \; \text{for all} \,  t \in [0,T])=1,$$ hence uniqueness.
\subsubsection{Existence for bounded initial data}
Similarly to the uniqueness proof we first require the additional assumption \eqref{add initial} and then we will remove it.  \\
Let $[\psi_R,u_R]$ be the strong solution for the approximate problem \eqref{approx R stoch quantum rho u}-\eqref{approx R stoch quantum u} we define $$ \tau_R= \tau^{\psi}_R \vee \tau^u_R$$ where 
\begin{equation*}
\tau^{\psi}_R= \inf \{ t \in [0,T] | \; \| \psi_R(t) \|_{W^{2,\infty}} \ge R \},
\end{equation*}
and 
\begin{equation*}
\tau^{u}_R= \inf \{ t \in [0,T] | \; \| u_R(t) \|_{W^{2,\infty}} \ge R \}.
\end{equation*}
with the convention $\inf \emptyset=T$.
Since $\psi_R$ and $u_R$ have continuous trajectories in $H^{s+1}$ and $H^{s}$ respectively,  they are embedded into $W^{2,\infty},$ then $\tau^{\psi}_R,  \; \tau^u_R$ are both well-defined stopping time and due to the additional assumption \eqref{add initial} they are positive a.s., as a consequence also $\tau_R$ is positive a.s. Therefore the unique solution $[\psi_R,u_R]$ of the approximate problem \eqref{approx R stoch quantum rho u}-\eqref{approx R stoch quantum u} with initial condition $(\psi_0:= \log \rho_0,u_0)$ generates the local strong pathwise solution $$( \rho_R:= e^{\psi_R}, u_R,\tau_R) $$ to the original problem \eqref{stoc quantum} with initial condition $(\rho_0,u_0).$
\subsubsection{Existence for general initial data}
As already mentioned,  we remove the additional hypothesis \eqref{add initial} on the initial data.  Let $[\psi_R,u_R]$ be the solution of the approximate problem \eqref{approx R stoch quantum rho u}-\eqref{approx R stoch quantum u},  we define 
\begin{equation*}
\begin{split}
& \tau_K=\tau_K^1 \land \tau_K^2 \\ &\tau_K^1= \inf \{ t \in [0,T] \;| \; \| u_R(t) \|_{H^s} \ge K \},  \\ & \tau_K^2= \inf \{ t \in [0,T] \; | \;  \| \psi_R(t) \|_{H^{s+1}} \ge K \}
\end{split}
\end{equation*}
with $K=K(R)$ such that $K(R) \rightarrow \infty$ as $R \rightarrow \infty$ and $K(R) < R \min \bigg\{ 1, \dfrac{1}{c_1}, \dfrac{1}{c_2} \bigg\}$
where $c_1,c_2$ are the constant of the embedding inequalities
$$ \| \psi \|_{W^{2,\infty}} \le c_1 \| \psi \|_{H^{s+1}}, \quad \| u \|_{W^{2,\infty}} \le c_2 \| u \|_{H^s}.$$
Since $s> \frac{5}{2}$ then by definition of $\tau_K$ we have that on the interval $[0,\tau_K),$ $$ \sup_{ t \in [0,\tau_K]} \| \psi_R(t) \|_{W^{2,\infty}}< R, \quad \sup_{ t \in [0, \tau_K]} \| u_R(t) \|_{W^{2,\infty}} < R, \quad \mathbb{P} \text{-a.s}. \quad.$$
Now we can use the same construction used in Theorem \ref{second thm} to construct solutions with stopping time $\tau_K$ having general initial data.  Indeed,  given $[\psi_0,u_0]$ an $\mathfrak{F}_0$-measurable random variable satisfying \eqref{C.I STRONG}, we define the set 
$$ U_{K(R)}= \bigg\{ [\psi,u] \in H^{s+1}(\mathbb{T}) \times H^s(\mathbb{T}) \; \bigg| \; \| \psi \|_{H^{s+1}} < K, \; \| u \|_{H^s} < K \bigg\}.$$
By using Theorem \ref{second thm} we infer the existence of a unique solution $[\psi_M,u_M]$ to \eqref{approx R stoch quantum rho u}-\eqref{approx R stoch quantum u} with $M=R$ and with initial data $$[\psi_0.u_0] 1_{[\psi_0,u_0] \in \{ U_{K(M)}  \bigcup_{J=1}^{M-1} U_{K(J)} \} }$$ and we observe that it is also a solution of \eqref{stoc quantum} up to the stopping time $\tau_{K(M)}.$ Finally we define $$[\psi,u]= \sum_{k=1}^{\infty} [\psi_M,u_M] 1_{[\psi_0,u_0] \in \{ u_{K(M)}  \bigcup_{J=1}^{M-1} U_{K(J)} \} },$$ which turns out to be a solution of \eqref{stoc quantum} with initial data $[\psi_0,u_0],$ up to the a.s. strictly positive stopping time $$ \tau= \sum_{k=1}^{\infty} \tau_{K(M)} 1_{[\psi_0,u_0] \in \{ u_{K(M)}  \bigcup_{J=1}^{M-1} U_{K(J)} \} }.$$
We observe that $[\psi,u]$ has a.s. continuous trajectories in $H^{s+1}(\mathbb{T}) \times H^s(\mathbb{T})$ and $u \in L^2(0,T; H^{s+1} (\mathbb{T}))$ $\mathbb{P}$-a.s.  indeed we take $\Omega_M \subset \Omega, \; M \in \mathbb{N},$ a collection of disjoint sets satisfying $\bigcup_{M} \Omega_M= \Omega$ such that $[\psi,u](\omega)=[\psi_M,u_M](\omega)$ for a.e. $\omega \in \Omega_M.$
By using Theorem \ref{second thm} we have that the trajectories of $[\psi_M,u_M]$ are a.s.  continuous in $H^{s+1}(\mathbb{T}) \times H^s(\mathbb{T})$ and integrability on $\omega$ is not assumed.  Finally by defining $( \rho_R:= e^{\psi_R}, u_R,\tau_R) $ we get the existence of a strong pathwise solution to system \eqref{stoc quantum} up to the strictly positive stopping time $\tau.$
\subsubsection{Existence of a maximal strong solution}
Let $\tau$ be the maximal time of existence of the solution solution $(\rho,u)$ and denote by $\mathcal{S}$ the set of all possible a.s. strictly positive stopping times corresponding to the solution starting from the initial datum $[\rho_0,u_0].$ The set $\mathcal{S}$ is non empty and it is closed with respect to finite maximum and finite minimum operations. \\
Let $\tau= \text{ess} \sup_{\sigma \in \mathcal{S}} \sigma,$ then we have that there exists an increasing sequence $\{ \sigma_M \}_{M \in \mathbb{N}} \subset \mathcal{S}$ such that $\sigma_M \rightarrow \tau $ a.s.  as $M \rightarrow \infty.$ Let us denote $(\rho_M,u_M)$ be the corresponding sequence of solution on the interval $[0, \sigma_M].$ Since uniqueness holds,  we set $(\rho,u):=(\rho_M,u_M)$ on $[0, \sigma_M]$ which is a well defined solution also on $\bigcup_M [0,\sigma_M].$
For each $R \in \mathbb{N},$ define $$\tau_R= \tau \land [ \tau^{\psi}_R \vee \tau^u_R).$$
Then $(\rho,u)$ is a solution on $[0, \sigma_M \land \tau_R]$ and sending $M \rightarrow \infty$ we get that $(\rho,u)$ is a solution on $[0, \tau_R].$
We observe that $\tau_R$ is not a.s.  strictly positive unless the initial datum $(\psi_0,u_0)$ satisfy $\| \psi_0 \|_{W^{2,\infty}} < R, \;  \text{and} \; \| u_0 \|_{W^{2,\infty}}< R. $ 
Since $(\psi_0,u_0) \in H^{s+1}(\mathbb{T}) \times H^{s}(\mathbb{T})$ then we deduce that for almost every $\omega$,  there exists $R=R(\omega)$ such that $\tau_{R(\omega)}(\omega) >0.$ \\
Finally, in order to have strict positivity,  we define $\tilde{\tau}_R= \sigma_R \land \tau_R.p$ Therefore each triplet $(\rho,u,\tilde{\tau}_R), \; R \in \mathbb{N}$ is a local strong pathwise solution with an a.s. positive stopping time.  \\
Now we observe that the construction of the local strong pathwise solution on $[0, \tilde{\tau}_R]$ can be extended to a solution on $[0, \tilde{\tau}_R+ \sigma]$ for $\mathbb{P}$-a.s positive stopping time $\sigma.$ Indeed we can construct a new solution having $(\rho(\tilde{\tau}_R),u(\tilde{\tau}_R))$ as initial condition as follows: \\
we define a stochastic basis $$(\tilde{\Omega},\mathfrak{\tilde{F}},(\mathfrak{\tilde{F}}_t)_{t \ge 0}, \tilde{\mathbb{P}}):=(\Omega,\mathfrak{F},(\mathfrak{F}_{{\tilde{\tau}_R+t)}_{t \ge 0}}, \mathbb{P})$$ and the cylindrical $(\mathfrak{\tilde{F}_t})$-Wiener process $$ \tilde{W}= (\tilde{W}_k)_{k\in \mathbb{N}}, \quad \tilde{W}_k(t):=W_k(\tilde{\tau}_R+t)-W_k(\tilde{\tau}_R), \; k \in \mathbb{N}.$$
The initial datum $(\rho(\tilde{\tau}_R),u(\tilde{\tau}_R))$ satisfy the assumption on initial condition and so we obtain that $(\tilde{\rho},\tilde{u}, \sigma),$ is a local strong pathwise solution to \eqref{stoc quantum}-\eqref{C.I} relative to the cylindrical Wiener process $\tilde{W}$ on the stochastic basis $(\tilde{\Omega},\mathfrak{\tilde{F}},(\mathfrak{\tilde{F}}_t)_{t \ge 0}, \tilde{\mathbb{P}})$ with a $\mathbb{P}$-a.s.strictly positive stopping time $\sigma.$
Since uniqueness holds,  we set
\begin{equation}
(\tilde{\rho},\tilde{u})(t):= \begin{cases} (\rho,u)(t),  \; t \le \tilde{\tau}_R \\ (\tilde{\rho},\tilde{u})(t), \; \tilde{\tau}_R < t \le \tilde{\tau}_R+\sigma,
\end{cases}
\end{equation}
which is a local strong pathwise solution to \eqref{stoc quantum}-\eqref{C.I} up to the stopping time $\tilde{\tau}_R+\sigma > \tilde{\tau}_R$
In order to show that $\tilde{\tau}_R< \tau$ on $[\tau < T]$ we assume by contradiction that $\mathbb{P}(\tilde{\tau}_R=\tau < T) >0.$ Then $\tilde{\tau}_R + \sigma \in \mathcal{S}$ and $\mathbb{P}(\tau <\tilde{\tau}_R+\sigma ) >0.$ which is a contradiction with respect to the maximality of $\tau.$ Hence we conclude that $\{ \tilde{\tau}_R \}$ is an increasing sequence of stopping times converging to $\tau$.  Moreover on the set  $[\tau < T]$ we have 
\begin{equation*}
\max \{ \sup_{t \in [0,\tilde{\tau}_R]} \| \psi (t) \|_{W_x^{2,\infty}},\sup_{t \in [0,\tilde{\tau}_R]}  \| u(t) \|_{W_x^{2,\infty}} \} \ge R,
\end{equation*}
which completes the existence part.
\subsection*{Acknowledgments}
The authors gratefully acknowledge the partial support by the Gruppo
Na\-zio\-na\-le per l’Analisi Matematica, la Probabilit\`a e le loro
Applicazioni (GNAMPA) of the Istituto Nazionale di Alta Matematica
(INdAM), and by the PRIN 2020 ``Nonlinear evolution PDEs, fluid
dynamics and transport equations: theoretical foundations and
applications'' and by the PRIN2022
``Classical equations of compressible fluids mechanics: existence and
properties of non-classical solutions''. The first and the third authors gratefully acknowledge the partial support by PRIN2022-PNRR ``Some
mathematical approaches to climate change and its impacts.''

\appendix
\section{Appendix}
In this appendix we state some useful theorems and propositions we frequently refer to in our work. We start by recalling the so-called Burkholder-Davis-Gundy inequality and an high order integrability lemma for the stochastic integral.
\begin{proposition}(Burkholder-Davis-Gundy's inequality) \label{BDG}\\
Let $H$ be a separable Hilbert space.  Let $p \in (0, \infty).$ There exists a constant $C_p>0$ such that,  for every $(\mathfrak{F}_t)$-progressively measurable stochastic process $\mathbb{G}$ satisfying 
\begin{equation}
\mathbb{E} \int_{0}^{T} \| \mathbb{G}(t) \|^2_{L_{2}(\mathfrak{U},H)}dt \ < \infty,
\end{equation}
the following holds:
\begin{equation}
\mathbb{E} \sup_{ t \in [0,T]} \bigg{ \| } \int_{0}^{t} \mathbb{G}(s)dW(s) \bigg{ \| }^p_H \le C_p \mathbb{E} \bigg( \int_{0}^{T} \| \mathbb{G}(s) \|^2_{L_{2}(\mathfrak{U},H)} ds \bigg)^{\frac{p}{2}}.
\end{equation}
\end{proposition}
\noindent
\begin{lemma} \label{Integrability}
Let $G_k=G_k(\rho,q)$ satisfy \eqref{G1}-\eqref{G2} for a non-negative integer $s.$ Let $p \ge 2, \; r \in [0,\frac{1}{2}).$ Suppose that $$ \rho, q \in L^{\beta p}(\Omega \times (0,T); H^s(\mathbb{T}^3)), \quad \beta= \max\{s,1\}$$
Then the following holds:
\begin{itemize}
\item[(1)]
If $s=0,$ then $$ t \longmapsto \int_{0}^{t} \mathbb{G}(\rho,q) dW \in L^P(\Omega; W^{r,p} (0,T;L^2(\mathbb{T}^3)))$$
and 
$$\mathbb{E} \bigg[ \bigg\| \int_{0}^{t} \mathbb{G}(\rho,q)dW \bigg\|^p_{W^{r,p}_t L^2_x} \bigg] \le c(r,p) \mathbb{E} \bigg[ \int_{0}^{T} \| [\rho,q] \|^p_{L^2_x} dt \bigg].$$
\item[(2)]
If $s> \frac{3}{2},$ then $$ t \longmapsto \int_{0}^{t} \mathbb{G}(\rho,q) dW \in L^P(\Omega; W^{r,p} (0,T;H^s(\mathbb{T}^3)))$$
and 
$$\mathbb{E} \bigg[ \bigg\| \int_{0}^{t} \mathbb{G}(\rho,q)dW \bigg\|^p_{W^{r,p}_t H^s_x} \bigg] \le c(r,p) \mathbb{E} \bigg[ \int_{0}^{T} \| [\rho,q] \|^{sp}_{H^s_x} dt \bigg].$$
\end{itemize}
\end{lemma}
\noindent
The next is an H\"{o}lder continuity result for the stochastic integral due to Kolmogorov
\begin{theorem}(Kolmogorov continuity theorem) \label{Kolmogorov} \\
Let $U$ be a stochastic process taking values in a separable Banach space $X.$ Assume that there exists a constant $K>0, \,a \ge 1, \, b>0$ such that, for all $s,t \in [0,T],$ $$ \mathbb{E} \| U(t)-U(s) \|^a_{X} \le K |t-s|^{1+b}.$$
Then there exists $V,$ a modification of $U,$ which has $\mathbb{P}$-a.s.  H\"{o}lder continuous trajectories with exponent $\gamma$ for every $\gamma \in (0, \frac{b}{a}).$
In addition,  we have $$\mathbb{E} \| V \|^a_{C^\gamma_t X} \lesssim K,$$
where the proportional constant does not depend on $V.$
\end{theorem}
\noindent
We recall the following It$\hat{\text{o}}$ Lemma,  see \cite{Feir},  Theorem A.4.1 for the stronger version of this result.
\begin{theorem}\label{Ito Lemma}
Let $W$ be an $\mathfrak{F}_t$-cylindrical Wiener process on the stochastic basis \newline $(\Omega, \mathfrak{F},(\mathfrak{F}_t
)_{t \ge 0}, \mathbb{P}).$
Let $(r,s)$ be a pair of stochastic processes on $(\Omega, \mathfrak{F},(\mathfrak{F}_t
)_{t \ge 0}, \mathbb{P})$ satisfying
\begin{equation}
\text{d}r=[Dr]\text{d}t+[\mathbb{D}r] \text{d}W, \quad \quad \text{d}s=[Ds]\text{d}t+[\mathbb{D}s] \text{d}W
\end{equation}
on the cylinder $(0,T) \times \mathbb{T}^d.$ Now suppose that the following 
\begin{equation}
r \in C^{\infty} ([0,T] \times \mathbb{T}^d), \quad \quad s \in C^{\infty} ([0,T] \times \mathbb{T}^d)
\end{equation}
holds $\mathbb{P}$-a.s. and that for all $1 \le q < \infty$
\begin{equation}
\mathbb{E} \bigg[ \sup_{t \in [0,T]} \| r \|^2_{W^{1,q}_x} \bigg]^q+\mathbb{E} \bigg[ \sup_{t \in [0,T]} \| s \|^2_{W^{1,q}_x} \bigg]^q \lesssim_q 1
\end{equation}
Furthermore,  assume that $[Dr], \; [Ds], \; [\mathbb{D}r], \; [\mathbb{D}s]$ are progressively measurable and that
\begin{equation}
\begin{split}
& [Dr], \; [Ds], \;  \in L^q(\Omega; L^q(0,T;W^{1,q}_x) \\ &
[\mathbb{D}r], \; [\mathbb{D}s] \in L^2( \Omega; L^2(0,T;L_2(\mathfrak{U};L^2_x) )
\end{split}
\end{equation}
and 
\begin{equation}
\bigg( \sum_{k \in \mathbb{N}} |[Dr](e_k)|^q\bigg)^{\frac{1}{q}}, \;\bigg( \sum_{k \in \mathbb{N}} |[Ds](e_k)|^q\bigg)^{\frac{1}{q}} \in L^q(\Omega \times (0,T) \times \mathbb{T}^d) 
\end{equation}
holds. Finally,  for some $\lambda \ge 0.$ let $Q$ be $(\lambda +2)$-continuously differentiable function such that 
\begin{equation}
\mathbb{E} \sup_{t \in [0,T]} \| Q^j(r) \|^2_{W^{\lambda,q'} \cap C_x } < \infty, \quad j=0,1,2.
\end{equation}
Then 
\begin{equation}
\begin{split}
& \int_{\mathbb{T}^d} (sQ(r))(t) dx= \int_{\mathbb{T}^d} (s_0Q(r_0))(t) dx+\int_{0}^{t} \int_{\mathbb{T}^d} \bigg[ sQ'(r)[Dr]+\frac{1}{2} \sum_{k \in \mathbb{N}} sQ''(r) | [\mathbb{D}r](e_k)|^2 \bigg] dxdt' \\ & +\int_{0}^{t} \int_{\mathbb{T}^d} Q(r)[Ds] dxdt'+\sum_{k\in\mathbb{N}} \int_{0}^{t} \int_{\mathbb{T}^d} [\mathbb{D}s](e_k)[\mathbb{D}r](e_k) dxdt' \\ & + \sum_{k\in \mathbb{N}} \int_{0}^{t} \int_{\mathbb{T}^d} \bigg[ sQ'(r)[\mathbb{D}r](e_k)+Q(r)[\mathbb{D}s](e_k)\bigg] dx dW_k(t').
\end{split}
\end{equation}
\end{theorem}
\noindent
\begin{theorem}(It$\hat{\text{o}}$'s formula) \\
Let $W$ be a cylindrical Wiener process in a separable Hilbert space $\mathfrak{U}.$ Let $\mathbb{G}$ be an $L_{2}(\mathfrak{U},H)$-valued stochastically integrable process, let $g$ be an $H$-valued progressively measurable Bochner integrable process, and let $U(0)$ be a $\mathfrak{F}_0$-measurable $H$-valued random variable so that the process $$ U(t)= U(0)+ \int_{0}^{t} g(s)ds +\int_{0}^{t} \mathbb{G}(s) dW(s)$$ is well defined.
Let $F: \, H \longrightarrow \mathbb{R}$ be a function such that $F, \, F',\, F''$ are uniformly continuous on bounded subset of $H.$ Then the following holds $\mathbb{P}$-a.s. for all $t \in [0,T]$
\begin{equation}
\begin{split}
& F(U(t))=F(U(0))+ \int_{0}^{t} \langle F'(U(s)),g(s) \rangle ds+ \int_{0}^{t} \langle F'(U(s)),G(s) dW(s) \rangle \\ & 
+\dfrac{1}{2} \int_{0}^{t} \text{Tr}(G(s)^*F{''}(U(s))G(s))ds
\end{split}
\end{equation}
where $\text{Tr} A= \sum_{k=1}^{\infty} \langle A e_k,e_k \rangle$ for $A$ being a bounded linear operator on H.
\end{theorem}
\noindent
\begin{proposition}(Stochastic Gronwall Lemma)\label{stoch gronw} \\
Fix $t>0$ and assume that $X,Y,Z,R \; : \; [0,t) \times \Omega \rightarrow \mathbb{R}$ are real valued,  non-negative stochastic processes.  Let $\tau < t$ be a stopping time so that 
\begin{equation}
\mathbb{E} \int_{0}^{\tau} (RX+Z)ds < \infty.
\end{equation}
Assume, moreover that for some fixed constant $k$
\begin{equation}
\int_{0}^{\tau} Rds < k,  \; a.s.
\end{equation}
Suppose that for all stopping times $0\le \tau_a< \tau_b \le \tau$
\begin{equation}
\mathbb{E} \bigg( \sup_{t \in [\tau_a,\tau_b]} X + \int_{\tau_a}^{\tau_b} Yds \bigg) \le c_0 \mathbb{E} \bigg( X(\tau_a) + \int_{\tau_a}^{\tau_b}(RX+Z)ds \bigg),
\end{equation}
where $c_0$ is independent of the choice of $\tau_a,\tau_b.$ Then 
\begin{equation}
\mathbb{E} \bigg( \sup_{t \in [0,\tau]} X + \int_{0}^{\tau} Y ds \bigg) \le c \mathbb{E} \bigg( X(0) + \int_{0}^{\tau} Zds \bigg),
\end{equation}
where $c=c(c_0,t,k).$
\end{proposition}
Next we state some useful compactness results in the stochastic setting
\begin{theorem} (Prokhorov's theorem) \label{Prok} \\
Let $X$ be a Polish space and $\mathcal{M}$ a collection of probability measures on $X$.  Then $\mathcal{M}$ is tight if and only if it is relatively weakly compact
\end{theorem}
\begin{theorem} (Skorokhod's theorem) \label{Skor} \\
Let $X$ be a Polish space and let $\mu_n, \; n \in \mathbb{N}_{0},$ be probability measures on $X$ such that $\mu_n$ converges weakly to $\mu_0.$ Then on some probability space there exists $X$-valued random variables $U_n, \; n \in \mathbb{N}_0,$ such that the law of $U_n$ is $\mu_n, \; n \in \mathbb{N}_0,$ and $U_n(\omega) \rightarrow U_0(\omega)$ in $X$ a.s.
\end{theorem}
\begin{theorem} \label{Gyong}
Let $X$ be a Polish space equipped with the Borel $\sigma$-algebra. A sequence of $X$-valued random variables $(U_n)_{n \in \mathbb{N}}$ converges in probability if and only if for every sequence of joint laws of $(U_{n_k},U_{m_k})_{k \in \mathbb{N}}$ there exists a further subsequence which converges weakly to a probability measure $\mu$ such that $$\mu((x,y) \in X \times X; \; x=y)=1.$$
\end{theorem}
\noindent
The next proposition concerns standard Moser type commutator estimates that are useful in order to estimates high order derivatives
\begin{proposition} \label{Moser}
\begin{itemize}
\item[(1)]For $u,v \in H^s \cap L^\infty (\mathbb{T}^N)$ and $\beta \in \mathbb{N}^N_0, \; |\beta| \le s,$ we have 
\begin{equation}
\| \partial^\beta_x (uv) \|_{L^2_x} \le c_s( \| u \|_{L^\infty_x} \| \nabla^s_x v \|_{L^2_x}+ \| v \|_{L^\infty_x} \| \nabla^s_x u \|_{L^2_x}).
\end{equation}
\item[(2)] For $u \in H^s(\mathbb{T}^N), \; \nabla_x u \in L^\infty(\mathbb{T}^N), \; v \in H^{s-1} \cap L^\infty (\mathbb{T}^N) $ and $\beta \in \mathbb{N}^N_0, \; |\beta| \le s,$ we have 
\begin{equation}
\| \partial^\beta_x (uv)-u \partial^\beta_x v \|_{L^2_x} \le c_s( \| \nabla_x u \|_{L^\infty_x} \| \nabla^{s-1}_x v \|_{L^2_x}+ \| v \|_{L^\infty_x} \| \nabla^s_x u \|_{L^2_x}).
\end{equation}
\item[(3)] Let $u \in H^s \cap C(\mathbb{T}^3)$ and let $F$ be an $s$-times continuously differentiable function in an open neighbourhood of the compact set $G= \text{range}[u].$ Then we have,  for all $\beta \in \mathbb{N}^N_0, \; |\beta| \le s,$
\begin{equation} 
\| \partial^\beta_x F(u) \|_{L^2_x} \le c_s \| \partial_u F \|_{C^{s-1}(G)} \| u \|^{| \beta|-1}_{L^\infty_x} \| \partial^\beta_x u \|_{L^2_x}.
\end{equation}
\end{itemize}
\end{proposition}
\noindent
The following Lemma states a useful functional inequality for function with $H^2$ space regularity on the $d$-dimensional torus, we refer to \cite{Alazard} Proposition $2.8$ for a detailed discussion. We point out that the constant $\dfrac{9}{16}$ is the optimal one.
\begin{lemma}\label{Functional ineq }
For any $d \ge 1 $ and any positive function $f$ in $H^2(\mathbb{T}^d),$
\begin{equation}
\int_{\mathbb{T}^d} | \nabla f^\frac{1}{2}|^4  dx \le \dfrac{9}{16} \int_{\mathbb{T}^d} (\Delta f)^2 dx
\end{equation}
\end{lemma}


\noindent
\addcontentsline{toc}{chapter}{Bibliography}
\begin{thebibliography}{40}

\bibitem{Alazard} T. Alazard and D. Bresch,
\newblock \textit{Functional inequalities and strong Lyapunov functionals for free surface flows in fluid dynamics},
\newblock Interfaces Free Bound.,(2023).

\bibitem{Ant0} P. Antonelli, G. Cianfarani Carnevale, C. Lattanzio, and S. Spirito,
\newblock \textit{Relaxation limit from the quantum Navier-Stokes equations to the quantum drift-diffusion equation},
\newblock J. Nonlinear Sci. 31 (2021), no. 5, Paper No. 71, 32 pp. 1.

\bibitem{Ant1} P. Antonelli and P. Marcati,
\newblock \textit{On the finite energy weak solutions to a system in quantum fluid dynamics},
\newblock Comm.Math. Phys., 287, (2009),  no.2,  657-686.

\bibitem{Ant2} P. Antonelli and P. Marcati,
\newblock \textit{The Quantum Hydrodynamics system in two space dimensions},
\newblock Arch. Ration. Mech. Anal., 203, (2012), no.2, 499-527.

\bibitem{Ant3} P. Antonelli, P. Marcati, and H. Zheng, 
\newblock \textit{An intrinsically hydrodynamic approach to multidimensional {QHD} systems},
\newblock Arch. Ration. Mech. Anal., 247, (2023), Paper No. 24, 58.

\bibitem{Ant4} P. Antonelli, P. Marcati, and H. Zheng, 
\newblock \textit{Genuine hydrodynamic analysis to the 1-{D} {QHD} system:
              existence, dispersion and stability},
\newblock Comm. Math. Phys., 383, (2021),  no.3, 2113--2161.

\bibitem{Ant5} P. Antonelli, L. E. Hientzsch, and P. Marcati, 
\newblock \textit{On the low Mach number limit for quantum Navier-Stokes equations},
\newblock SIAM J. Math. Anal., 52, (2020), no.6,  6105-6139.

\bibitem{Spirito} P. Antonelli and S. Spirito,
\newblock \textit{Global existence of weak solutions to the Navier-Stokes-Korteweg equations},
\newblock Ann.  Inst.  H.  Poincarè C. Anal. Non Linèaire 39 (2022),  no.1, pp 171-200.

\bibitem{Spirito2} P. Antonelli and S. Spirito,
\newblock \textit{On the compactness of finite energy weak solutions to the quantum Navier-Stokes equations},
\newblock J.  Hyperbolic Differ.  Equ.,  15, (2018),  no.1,  pp. 133-147.

\bibitem{Aud} C. Audiard and B. Haspot,
\newblock \textit{Global well-posedness of the Euler-Korteweg system for small irrotational data},
\newblock Comm. Math. Phys., 351, (2017), no. 1,  201-247.

\bibitem{Benz} S. Benzoni-Gavage, R. Danchin, and S. Descombes,
\newblock \textit{On the well-posedness for the Euler-Korteweg model in several space dimensions},
\newblock Indiana Univ. Math. J., 56, (2007),  no.4, 1499-1579.

\bibitem{Feir} D. Breit,  E. Feireisl,  and M. Hofmanov, 
\newblock Stochastically Forced Compressible Fluid Flows,
\newblock De Gruyter Series in Applied and Numerical Mathematics. De Gruyter (2018).

\bibitem{BD} D. Bresch and B. Desjardins,
\newblock \textit{Sur un modèòe de Saint-Venant visqueux et sa limite quasi géostrophique, [ On viscous shallow-water equations (Saint-Venant model) and the quasi-geostrophic limit]},
\newblock C.R. Math. Acad. Sci. Paris, 335, (2002),  no.12, 1079-1084.

\bibitem{Bresch} D. Bresch, A. Vasseur,  and C. Yu,
\newblock \textit{Global existence of entropy-weak solutions to the compressible Navier-Stokes equations with non-linear density dependent viscosities},
\newblock J. Eur. Math. Soc., 24, (2022),  no.5, 1791--1837.

\bibitem{Brull} S. Brull and F. Méhats,
\newblock \textit{Derivation of viscous correction terms for the isothermal quantum Euler model},
\newblock ZAMM Z. Angew. Math. Mech, 90, (2010), no.3,  219-230.


\bibitem{Burtea1} C. Burtea and B. Haspot,
\newblock \textit{Existence of global strong solution for the Navier–Stokes–Korteweg system in one dimension for strongly degenerate viscosity coefficients},
\newblock Pure Appl. Anal.,  4, (2022), no. 3, 449-485.

\bibitem{Burtea2} C. Burtea and B. Haspot,
\newblock \textit{Vanishing capillarity limit of the {N}avier-{S}tokes-{K}orteweg system in one dimension with degenerate viscosity coefficient and discontinuous initial density},
\newblock SIAM J. Math. Anal., 54, (2022),  no.2, 1428--1469.

\bibitem{Donatelli} M. Caggio and D. Donatelli, 
\newblock \textit{High Mach number limit for Korteweg fluids with density dependent viscosity},
\newblock J. Differential Equations, 277,  (2021),  1-37.

\bibitem{Charve} F.  Charve and B. Haspot,
\newblock \textit{Existence of global strong solution and vanishing capillarity-viscosity limit in one dimension for the Korteweg system},
\newblock SIAM J.  Math. Anal.  45, (2013), no.2,  469-494.

\bibitem{Chen} Z. Chen, X. Chai, B. Dong,  and H. Zhao,
\newblock \textit{Global classical solutions to the one-dimensional compressible fluid models of Korteweg type for large initial data},
\newblock J. Differential Equations, 259, (2015), no.8, 4376-4411.

\bibitem{Cho 2004} Y. Cho,  H. J.  Choe, and H.  Kim,
\newblock \textit{Unique solvability of the initial boundary value problems for compressible viscous fluids},
\newblock J.  Math. Pures Appl., 83, (2004), no.2, 243-275.

\bibitem{Const} P. Constantin, T.D. Drivas,  H.Q. Nguyen,  and F.  Pasqualotto,
\newblock \textit{Compressible fluids and active potentials},
\newblock Ann. Inst. H. Poincarè C. Anal. Non Linèaire,  37,  (2020), no.1, pp. 145-180.

\bibitem{Const2} P. Constantin, T.D. Drivas,  and R. Shvydkoy,
\newblock \textit{Entropy hierarchies for equations of compressible fluids and self-organized dynamics},
\newblock SIAM J. Math. Anal., 52, (2020), no. 3, 3073-3092.

\bibitem{Coti} M. Coti-Zelati,  N. Glatt-Holtz,  and K. Trivisa,
\newblock \textit{Invariant measure for the stochastic compressible Navier-Stokes equations},
\newblock Appl. Math. Optim., 83, (2021), no. 3, 1487-1522.

\bibitem{Da Prato} G. Da Prato and J. Zabczyk,
\newblock Stochastic Equations in Infinite Dimensions,
\newblock volume 44 of \textit{Encyclopedia of Mathematics and its Applications},  Cambridge University Press, Cambridge, 1992.


\bibitem{Feir. Don.} D. Donatelli, E. Feireisl,  and P. Marcati,
\newblock \textit{Well/ill Posedness for the Euler-Korteweg-Poisson System and Related Problems},
\newblock Comm.  Partial Differential Equations, 40, (2015),  no.7,  pp. 1314-1335.

\bibitem{Don-Pes} D. Donatelli and L. Pescatore,
\newblock \textit{Blow up criteria for a fluid dynamical model arising in astrophysics},
\newblock  J. Hyperbolic Differ.  Equ., 20, (2023), no.3,  pp. 629-668.

\bibitem{Deb} A. Debussche, N. Glatt–Holtz, R. Temam, and M. Ziane,
\newblock \textit{Global Existence and Regularity for the 3D Stochastic Primitive Equations of the Ocean and Atmosphere with Multiplicative White Noise},
\newblock Nonlinearity, 25, (2012), no.7, 2093–2118.

\bibitem{Dunn} J.E. Dunn and J. Serrin,
\newblock \textit{On the thermomechanics of interstitial working},
\newblock Arch. Ration. Mech. Anal., 88,  no.2 (1985),  95-133.

\bibitem{Flandoli} F.  Flandoli,
\newblock An introduction to 3D stochastic fluid dynamics,
\newblock In \textit{SPDE in Hydrodynamic: Recrent Progress and Prospects}, volume 1942 of \textit{Lecture Notes in Math}., 51-150 Springer, Berlin, 2008.

\bibitem{Le Floch} P. Germain and P.  LeFloch,
\newblock \textit{Finite Energy Method for Compressible Fluids: The Navier-Stokes-Korteweg Model},
\newblock Comm. Pure Appl. Math., 69, (2015), no.1, 3-61.

\bibitem{Glatt-Holtz} N. Glatt-Holtz and M. Ziane.
\newblock \textit{Strong pathwise solutions of the stochastic Navier-Stokes system.}
\newblock Adv.  Differential Equations,  14, (2009), no.5-6: 567-600.

\bibitem{Haspot} B. Haspot,
\newblock \textit{Existence of Global Weak Solution for Compressible Fluid Models of Korteweg Type},
\newblock J. Math. Fluid Mech., 13, (2011),  no.2, 223-249.

\bibitem{Hatt} H. Hattori and D. Li,
\newblock \textit{Solutions for two dimensional systems for materials of Korteweg type},
\newblock SIAM J.Math.Anal., 25, (1994), no.1,  85-98.

\bibitem{Jung qns} A. J\"{u}ngel,  
\newblock \textit{Dissipative quantum fluid models},
\newblock Riv.  Mat. Univ.  Parma 3 (2012),  no.2, 217--290.

\bibitem{Lacroix} I. Lacroix-Violet and A. Vasseur,
\newblock \textit{Global weak solutions to the compressible quantum Navier-Stokes equation and its semi-classical limit},
\newblock J. Math. Pures Appl., 114 (2017), 191--210.

\bibitem{Landau} L. Landau and E. Lifschitz, 
\newblock Quantum Mechanics: Non-relativistic Theory.
\newblock Pergamon Press,  New York, 1977.

\bibitem{Mellet} A. Mellet and A. Vasseur,
\newblock \textit{Existence and uniqueness of global strong solutions for one-dimensional compressible Navier-Stokes equations},
\newblock SIAM J. Math. Anal., 39, (2008), no. 4,  pp. 1344–1365.

\bibitem{Mik1} R. Mikulevicius and B.L. Rozovskii,
\newblock \textit{Stochastic Navier-Stokes equations for turbulent flows},
\newblock SIAM J. Math. Anal., 35, (2004), no.5, 1250-1310.

\bibitem{Mik2} R. Mikulevicius and B.L. Rozovskii,
\newblock \textit{Global l2-solutions of stochastic Navier-Stokes equations},
\newblock Ann.Probab., 33, (2005), no.1, 137-176.
\end{thebibliography}
\end{document}